\documentclass[11pt]{preprint}
\usepackage[full]{textcomp}
\usepackage[osf]{newtxtext} 
\usepackage[cal=boondoxo]{mathalfa}
\usepackage{colortbl}

\usepackage[all,cmtip]{xy}
\usepackage{comment}

\usepackage{amssymb}
\usepackage{mathtools}
\usepackage{hyperref}
\usepackage{breakurl}
\usepackage{mhenvs}
\usepackage{mhequ} 
\usepackage{mhsymb}
\usepackage{booktabs}
\usepackage{tikz}
\usepackage{tcolorbox}
\usepackage{mathrsfs}
\usepackage[utf8]{inputenc}
\usepackage{longtable}
\usepackage{wrapfig}
\usepackage{subcaption}
\usepackage{mathrsfs}
\usepackage{epsfig}
\usepackage{microtype}
\usepackage{comment}
\usepackage{wasysym}
\usepackage{centernot}
\usepackage{enumitem}
\usepackage{bm}
\usepackage{stackrel}
\usepackage{graphicx}
\usepackage{axodraw}
\usepackage{xspace}
\makeatletter
\newcommand{\globalcolor}[1]{%
  \color{#1}\global\let\default@color\current@color
}
\makeatother

\usetikzlibrary{calc}
\usetikzlibrary{decorations}
\usetikzlibrary{positioning}
\usetikzlibrary{shapes}
\usetikzlibrary{external}

\definecolor{blush}{rgb}{0.87, 0.36, 0.51}
	\definecolor{brightcerulean}{rgb}{0.11, 0.67, 0.84}
	\definecolor{greenryb}{rgb}{0.4, 0.69, 0.2}

\newif\ifdark
\darkfalse

\ifdark
\definecolor{darkred}{rgb}{0.9,0.2,0.2}
\definecolor{darkblue}{rgb}{0.7,0.3,1}
\definecolor{darkgreen}{rgb}{0.1,0.9,0.1}
\definecolor{franck}{rgb}{0,0.8,1}
\definecolor{pagebackground}{rgb}{.15,.21,.18}
\definecolor{pageforeground}{rgb}{.84,.84,.85}
\pagecolor{pagebackground}
\AtBeginDocument{\globalcolor{pageforeground}}
\definecolor{symbols}{rgb}{0,0.7,1}
\colorlet{connection}{red!80!black}
\colorlet{boxcolor}{blue!50}

\else

\definecolor{darkred}{rgb}{0.7,0.1,0.1}
\definecolor{darkblue}{rgb}{0.4,0.1,0.8}
\definecolor{darkgreen}{rgb}{0.1,0.7,0.1}
\definecolor{franck}{rgb}{0,0,1}
\definecolor{pagebackground}{rgb}{1,1,1}
\definecolor{pageforeground}{rgb}{0,0,0}
\colorlet{symbols}{blue!90!black}
\colorlet{connection}{red!30!black}
\colorlet{boxcolor}{blue!50!black}

\fi

\def\slash{\leavevmode\unskip\kern0.18em/\penalty\exhyphenpenalty\kern0.18em}
\def\dash{\leavevmode\unskip\kern0.18em--\penalty\exhyphenpenalty\kern0.18em}

\DeclareMathAlphabet{\mathbbm}{U}{bbm}{m}{n}

\DeclareFontFamily{U}{BOONDOX-calo}{\skewchar\font=45 }
\DeclareFontShape{U}{BOONDOX-calo}{m}{n}{
  <-> s*[1.05] BOONDOX-r-calo}{}
\DeclareFontShape{U}{BOONDOX-calo}{b}{n}{
  <-> s*[1.05] BOONDOX-b-calo}{}
\DeclareMathAlphabet{\mcb}{U}{BOONDOX-calo}{m}{n}
\SetMathAlphabet{\mcb}{bold}{U}{BOONDOX-calo}{b}{n}

\setlist{noitemsep,topsep=4pt,leftmargin=1.5em}

\DeclareMathAlphabet{\mathbbm}{U}{bbm}{m}{n}

\DeclareMathAlphabet{\mcb}{U}{BOONDOX-calo}{m}{n}
\SetMathAlphabet{\mcb}{bold}{U}{BOONDOX-calo}{b}{n}
\DeclareFontFamily{U}{mathx}{\hyphenchar\font45}
\DeclareFontShape{U}{mathx}{m}{n}{
      <5> <6> <7> <8> <9> <10>
      <10.95> <12> <14.4> <17.28> <20.74> <24.88>
      mathx10
      }{}
\DeclareSymbolFont{mathx}{U}{mathx}{m}{n}
\DeclareMathSymbol{\bigtimes}{1}{mathx}{"91}

\def\s{\mathfrak{s}}

\providecommand{\figures}{false}
{ \ifthenelse{\equal{\figures}{false}} {#1}{\[ {\rm Figure \ missing !} \]} }{}
\def\id{\mathrm{id}}

\def\CP{\mathcal{P}}

\def\CA{\mathcal{A}}

\def\CC{\mathcal{C}}
\def\CQ{\mathcal{Q}}

\def\CM{\mathcal{M}}
\def\CT{\mathcal{T}}

\tikzstyle{tinydots}=[dash pattern=on \pgflinewidth off \pgflinewidth]
\tikzstyle{superdense}=[dash pattern=on 4pt off 1pt]

\def\root{\mathrm{root}}

\def\nonroot{\mathrm{non-root}}

\newcommand{\beq}{\begin{equation}}
\newcommand{\eeq}{\end{equation}}

\usepackage{empheq}

\newcommand{\T}{\mathbf{T}}

\def\Labe{\mathfrak{e}}

\def\Labn{\mathfrak{n}}

\def\Labhom{\mathfrak{t}}

\def\Lab{\mathfrak{L}}

\def\Deltam{\Delta^{\!-}}

\def\Deltap{\Delta^{\!+}}

\def\hattimes{\mathbin{\hat\otimes}}
\def\${|\!|\!|}

\newcounter{theorem}

\newtheorem{ex}[theorem]{Example}

\newenvironment{DIFnomarkup}{}{} 

\theorembodyfont{\rmfamily}

\newcommand{\wt}{\widetilde} \newcommand{\wh}{\widehat}
\newcommand{\ccurvearrowright}{{\curvearrowright\hskip -3.4mm\curvearrowright}}
\newcommand{\rrightarrow}{{\to\hskip -4.9mm\raise 1pt\hbox{$\to$}}}
\newcommand{\uuparrow}{{\uparrow\hskip-4.5pt\uparrow}}
\newcommand{\lbutcher}{\circ\hskip -4.5pt\to}
\def\diagramme #1{\vskip 4mm \centerline {#1} \vskip 4mm}

\newfont{\indic}{bbmss12}

\def\Nabla_#1{\nabla_{\!#1}}

\makeatletter
\pgfdeclareshape{crosscircle}
{
  \inheritsavedanchors[from=circle] 
  \inheritanchorborder[from=circle]
  \inheritanchor[from=circle]{north}
  \inheritanchor[from=circle]{north west}
  \inheritanchor[from=circle]{north east}
  \inheritanchor[from=circle]{center}
  \inheritanchor[from=circle]{west}
  \inheritanchor[from=circle]{east}
  \inheritanchor[from=circle]{mid}
  \inheritanchor[from=circle]{mid west}
  \inheritanchor[from=circle]{mid east}
  \inheritanchor[from=circle]{base}
  \inheritanchor[from=circle]{base west}
  \inheritanchor[from=circle]{base east}
  \inheritanchor[from=circle]{south}
  \inheritanchor[from=circle]{south west}
  \inheritanchor[from=circle]{south east}
  \inheritbackgroundpath[from=circle]
  \foregroundpath{
    \centerpoint%
    \pgf@xc=\pgf@x%
    \pgf@yc=\pgf@y%
    \pgfutil@tempdima=\radius%
    \pgfmathsetlength{\pgf@xb}{\pgfkeysvalueof{/pgf/outer xsep}}%
    \pgfmathsetlength{\pgf@yb}{\pgfkeysvalueof{/pgf/outer ysep}}%
    \ifdim\pgf@xb<\pgf@yb%
      \advance\pgfutil@tempdima by-\pgf@yb%
    \else%
      \advance\pgfutil@tempdima by-\pgf@xb%
    \fi%
    \pgfpathmoveto{\pgfpointadd{\pgfqpoint{\pgf@xc}{\pgf@yc}}{\pgfqpoint{-0.707107\pgfutil@tempdima}{-0.707107\pgfutil@tempdima}}}
    \pgfpathlineto{\pgfpointadd{\pgfqpoint{\pgf@xc}{\pgf@yc}}{\pgfqpoint{0.707107\pgfutil@tempdima}{0.707107\pgfutil@tempdima}}}
    \pgfpathmoveto{\pgfpointadd{\pgfqpoint{\pgf@xc}{\pgf@yc}}{\pgfqpoint{-0.707107\pgfutil@tempdima}{0.707107\pgfutil@tempdima}}}
    \pgfpathlineto{\pgfpointadd{\pgfqpoint{\pgf@xc}{\pgf@yc}}{\pgfqpoint{0.707107\pgfutil@tempdima}{-0.707107\pgfutil@tempdima}}}
  }
}
\makeatother

\def\symbol#1{\textcolor{symbols}{#1}}

\def\decorate#1#2{
        \ifnum#2>0
    		\foreach \count in {1,...,#2}{
	       	let
				\p1 = (sourcenode.center),
                \p2 = (sourcenode.east),
				\n1 = {\x2-\x1},
				\n2 = {1mm},
				\n3 = {(1.3+0.6*(\count-1))*\n1},
				\n4 = {0.7*\n1}
			in 
        		node[rectangle,fill=symbols,rotate=30,inner sep=0pt,minimum width=0.2*\n2,minimum height=\n2] at ($(sourcenode.center) + (\n3,\n4)$) {}
				}
		\fi
        \ifnum#1>0
    		\foreach \count in {1,...,#1}{
	       	let
				\p1 = (sourcenode.center),
                \p2 = (sourcenode.east),
				\n1 = {\x2-\x1},
				\n2 = {1mm},
				\n3 = {(1.3+0.6*(\count-1))*\n1},
				\n4 = {0.7*\n1}
			in 
        		node[rectangle,fill=symbols,rotate=-30,inner sep=0pt,minimum width=0.2*\n2,minimum height=\n2] at ($(sourcenode.center) + (-\n3,\n4)$) {}
				}
		\fi
}

\tikzset{
    dectriangle/.style 2 args={
        triangle,
        alias=sourcenode,
        append after command={\decorate{#1}{#2}}
    },
    dectriangle/.default={0}{0},
}

\tikzset{
	cross/.style={path picture={ 
  		\draw[symbols]
			(path picture bounding box.south east) -- (path picture bounding box.north west) (path picture bounding box.south west) -- (path picture bounding box.north east);
		}},
root/.style={circle,fill=green!50!black,inner sep=0pt, minimum size=1.2mm},
        dot/.style={circle,fill=pageforeground,inner sep=0pt, minimum size=1mm},
        dotred/.style={circle,fill=pageforeground!50!pagebackground,inner sep=0pt, minimum size=2mm},
        var/.style={circle,fill=pageforeground!10!pagebackground,draw=pageforeground,inner sep=0pt, minimum size=3mm},
        kernel/.style={semithick,shorten >=2pt,shorten <=2pt},
        kernels/.style={snake=zigzag,shorten >=2pt,shorten <=2pt,segment amplitude=1pt,segment length=4pt,line before snake=2pt,line after snake=5pt,},
        rho/.style={densely dashed,semithick,shorten >=2pt,shorten <=2pt},
           testfcn/.style={dotted,semithick,shorten >=2pt,shorten <=2pt},
        renorm/.style={shape=circle,fill=pagebackground,inner sep=1pt},
        labl/.style={shape=rectangle,fill=pagebackground,inner sep=1pt},
        xic/.style={very thin,circle,draw=symbols,fill=symbols,inner sep=0pt,minimum size=1.2mm},
        g/.style={very thin,rectangle,draw=symbols,fill=symbols!10!pagebackground,inner sep=0pt,minimum width=2.5mm,minimum height=1.2mm},
        xi/.style={very thin,circle,draw=symbols,fill=symbols!10!pagebackground,inner sep=0pt,minimum size=1.2mm},
	xies/.style={very thin,rectangle,fill=green!50!black!25,draw=symbols,inner sep=0pt,minimum size=1.1mm},
	xiesf/.style={very thin,rectangle,fill=green!50!black,draw=symbols,inner sep=0pt,minimum size=1.1mm},
        xix/.style={very thin,crosscircle,fill=symbols!10!pagebackground,draw=symbols,inner sep=0pt,minimum size=1.2mm},
        X/.style={very thin,cross,rectangle,fill=pagebackground,draw=symbols,inner sep=0pt,minimum size=1.2mm},
	xib/.style={thin,circle,fill=symbols!10!pagebackground,draw=symbols,inner sep=0pt,minimum size=1.6mm},
	xie/.style={thin,circle,fill=green!50!black,draw=symbols,inner sep=0pt,minimum size=1.6mm},
	xid/.style={thin,circle,fill=symbols,draw=symbols,inner sep=0pt,minimum size=1.6mm},
	xibx/.style={thin,crosscircle,fill=symbols!10!pagebackground,draw=symbols,inner sep=0pt,minimum size=1.6mm},
	kernels2/.style={very thick,draw=connection,segment length=12pt},
	keps/.style={thin,draw=symbols,->},
	kepspr/.style={thick,draw=connection,->},
	krho/.style={thin,draw=symbols,superdense,->},
	krhopr/.style={thick,draw=connection,superdense},
	triangle/.style = { regular polygon, regular polygon sides=3},
	not/.style={thin,circle,draw=connection,fill=connection,inner sep=0pt,minimum size=0.5mm},
	diff/.style = {very thin,draw=symbols,triangle,fill=red!50!black,inner sep=0pt,minimum size=1.6mm},
	diff1/.style = {very thin,dectriangle={1}{0},fill=red!50!black,draw=symbols,inner sep=0pt,minimum size=1.6mm},
	diff2/.style = {very thin,dectriangle={1}{1},fill=red!50!black,draw=symbols,inner sep=0pt,minimum size=1.6mm},
		diffmini/.style = {very thin,rectangle,fill=black,draw=black,inner sep=0pt,minimum size=0.75mm},
	 kernelsmod/.style={very thick,draw=connection,segment length=12pt},
	 rec/.style = {very thin,rectangle,fill=black,draw=black,inner sep=0pt,minimum size=2mm},
	cerc/.style={very thin,circle,draw=black,fill=symbols,inner sep=0pt,minimum size=2mm},
	stars/.style={very thin,star,star points=6,star point ratio=0.5, draw=black,fill=red,inner sep=0pt,minimum size=0.7mm},
	>=stealth,
        }
        \tikzset{
root/.style={circle,fill=black!50,inner sep=0pt, minimum size=3mm},
        circ/.style={circle,fill=white,draw=black,very thin,inner sep=.5pt, minimum size=1.2mm},
        round1/.style={fill=white,outer sep = 0,inner sep=2pt,rounded corners=1mm,draw,text=black,thin,minimum size=1.2mm},
          circ1/.style={circle,fill=red!10,draw=red,very thin,inner sep=.5pt, minimum size=1.2mm},
        rect/.style={fill=white,outer sep = 0,inner sep=2pt,rectangle,draw,text=black,thin,minimum size=1.2mm},
        rect1/.style={fill=white,outer sep = 0,inner sep=2pt,rectangle,draw,text=black,thin,minimum size=1.2mm},
        round2/.style={fill=red!10,outer sep = 0,inner sep=2pt,rounded corners=1mm,draw,text=black,thin,minimum size=1.2mm},
       round3/.style={fill=blue!10,outer sep = 0,inner sep=2pt,rounded corners=1mm,draw,text=black,thin,minimum size=1.2mm}, 
        rect2/.style={fill=black!10,outer sep = 0,inner sep=2pt,rectangle,draw,text=black,thin,minimum size=1.2mm},
        dot/.style={circle,fill=black,inner sep=0pt, minimum size=1.2mm},
        dotred/.style={circle,fill=black!50,inner sep=0pt, minimum size=2mm},
        var/.style={circle,fill=black!10,draw=black,inner sep=0pt, minimum size=3mm},
        kernel/.style={semithick,shorten >=2pt,shorten <=2pt},
         diag/.style={thin,shorten >=4pt,shorten <=4pt},
        kernel1/.style={thick},
        kernels/.style={snake=zigzag,shorten >=2pt,shorten <=2pt,segment amplitude=1pt,segment length=4pt,line before snake=2pt,line after snake=5pt,},
		kernels1/.style={snake=zigzag,segment amplitude=0.5pt,segment length=2pt},
		rho1/.style={densely dotted,semithick},
        rho/.style={densely dashed,semithick,shorten >=2pt,shorten <=2pt},
           testfcn/.style={dotted,semithick,shorten >=2pt,shorten <=2pt},
           visible/.style={draw, circle, fill, inner sep=0.25ex},
        renorm/.style={shape=circle,fill=white,inner sep=1pt},
        labl/.style={shape=rectangle,fill=white,inner sep=1pt},
        xic/.style={very thin,circle,fill=symbols,draw=black,inner sep=0pt,minimum size=1.2mm},
        xi/.style={very thin,circle,fill=blue!10,draw=black,inner sep=0pt,minimum size=1.2mm},
	xib/.style={very thin,circle,fill=blue!10,draw=black,inner sep=0pt,minimum size=1.6mm},
	xie/.style={very thin,circle,fill=green!50!black,draw=black,inner sep=0pt,minimum size=1mm},
	xid/.style={very thin,circle,fill=symbols,draw=black,inner sep=0pt,minimum size=1.6mm},
	edgetype/.style={very thin,circle,draw=black,inner sep=0pt,minimum size=5mm},
	nodetype/.style={very thick,circle,draw=black,inner sep=0pt,minimum size=5mm},
	kernels2/.style={very thick,draw=connection,segment length=12pt},
clean/.style={thin,circle,fill=black,inner sep=0pt,minimum size=1mm},	not/.style={thin,circle,fill=symbols,draw=connection,fill=connection,inner sep=0pt,minimum size=0.8mm},
	>=stealth,
        }

\makeatletter
\def\DeclareSymbol#1#2#3{%
	\expandafter\gdef\csname MH@symb@#1\endcsname{\tikzsetnextfilename{symbol#1}%
	\tikz[baseline=#2,scale=0.15,draw=symbols,line join=round]{#3}}%
	\expandafter\gdef\csname MH@symb@#1s\endcsname{\scalebox{0.75}{\tikzsetnextfilename{symbol#1}%
	\tikz[baseline=#2,scale=0.15,draw=symbols,line join=round]{#3}}}%
	\expandafter\gdef\csname MH@symb@#1ss\endcsname{\scalebox{0.65}{\tikzsetnextfilename{symbol#1}%
	\tikz[baseline=#2,scale=0.15,draw=symbols,line join=round]{#3}}}%
	}
\def\<#1>{\ifthenelse{\boolean{mmode}}{\mathchoice{\csname MH@symb@#1\endcsname}{\csname MH@symb@#1\endcsname}{\csname MH@symb@#1s\endcsname}{\csname MH@symb@#1ss\endcsname}}{\csname MH@symb@#1\endcsname}}
\makeatother
\def\MHmakebox#1#2#3{\tikz[baseline=#3,line width=#1,cross/.style={path picture={ 
  \draw[black](path picture bounding box.south) -- (path picture bounding box.north) (path picture bounding box.west) -- (path picture bounding box.east);
}}]{\draw[white] (0,0) rectangle (#2,#2);
\draw[black,cross] (0.1em,0.1em) rectangle (#2-0.1em,#2-0.1em);}}

\DeclareSymbol{Xi22}{0.5}{\draw (0,0) node[xi] {} -- (-1,1) node[not] {} -- (0,2) node[xi] {};} 
\DeclareSymbol{Xi2}{-2}{\draw (-1,-0.25) node[xi] {} -- (0,1) node[xi] {};} 
\DeclareSymbol{Xi2b}{-2}{\draw (-1,-0.25) node[xic] {} -- (0,1) node[xic] {};} 
\DeclareSymbol{Xi2g}{-2}{\draw (-1,-0.25) node[xies] {} -- (0,1) node[xi] {};} 
\DeclareSymbol{Xi2g2}{-2}{\draw (-1,-0.25) node[xi] {} -- (0,1) node[xies] {};} 
\DeclareSymbol{cXi2}{-2}{\draw (0,-0.25) node[xi] {} -- (-1,1) node[xic] {};}
\DeclareSymbol{Xi3}{0}{\draw (0,0) node[xi] {} -- (-1,1) node[xi] {} -- (0,2) node[xi] {};}
\DeclareSymbol{XiIIXi}{0}{\draw (0,0) node[xi] {} -- (-1,1); \draw[kernels2] (-1,1) node[not] {} -- (0,2) node[xi] {};}

\DeclareSymbol{Xi4}{2}{\draw (0,0) node[xi] {} -- (-1,1) node[xi] {} -- (0,2) node[xi] {} -- (-1,3) node[xi] {};}
\DeclareSymbol{Xi4_1}{2}{\draw (0,0) node[xic] {} -- (-1,1) node[xic] {} -- (0,2) node[xi] {} -- (-1,3) node[xi] {};}
\DeclareSymbol{Xi4_2}{2}{\draw (0,0) node[xic] {} -- (-1,1) node[xi] {} -- (0,2) node[xi] {} -- (-1,3) node[xic] {};}
\DeclareSymbol{Xi2X}{-2}{\draw (0,-0.25) node[xi] {} -- (-1,1) node[xix] {};}
\DeclareSymbol{XXi2}{-2}{\draw (0,-0.25) node[xix] {} -- (-1,1) node[xi] {};}
\DeclareSymbol{IIXi}{0}{\draw (0,-0.25) node[not] {} -- (-1,1) node[xi] {} -- (0,2) node[xi] {};}
\DeclareSymbol{IXi^2}{-1}{\draw (-1,1) node[xi] {} -- (0,0) node[not] {} -- (1,1) node[xi] {};}
\DeclareSymbol{IIXi^2}{-4}{\draw (0,-1.5) node[not] {} -- (0,0);
\draw[kernels2] (-1,1) node[xi] {} -- (0,0) node[not] {} -- (1,1) node[xi] {};}
\DeclareSymbol{XiX}{-2.8}{\node[xibx] {};}
\DeclareSymbol{tauX}{-2.8}{ \node[X] {};}
\DeclareSymbol{Xi}{-2.8}{\node[xib] {};}

\DeclareSymbol{IXiX}{-1}{\draw (0,-0.25) node[not] {} -- (-1,1) node[xix] {};}
\DeclareSymbol{IXi3}{2}{\draw (0,-0.25) node[not] {} -- (-1,1) node[xi] {} -- (0,2) node[xi] {} -- (-1,3) node[xi] {};}
\DeclareSymbol{IXi}{-2}{\draw (0,-0.25) node[not] {} -- (-1,1) node[xi] {};}
\DeclareSymbol{XiI}{-2}{\draw (0,-0.25) node[xi] {} -- (-1,1) node[not] {};}

\DeclareSymbol{Xi4b}{0}{\draw(0,1.5) node[xi] {} -- (0,0); \draw (-1,1) node[xi] {} -- (0,0) node[xi] {} -- (1,1) node[xi] {};}
\DeclareSymbol{Xi4b'}{0}{\draw(0,1.5) node[xi] {} -- (0,-0.2); \draw (-1,1) node[xi] {} -- (0,-0.2) node[not] {} -- (1,1) node[xi] {};}
\DeclareSymbol{Xi4c}{0}{\draw (0,1) -- (0.8,2.2) node[xi] {};\draw (0,-0.25) node[xi] {} -- (0,1) node[xi] {} -- (-0.8,2.2) node[xi] {};}
\DeclareSymbol{Xi4d}{-4.5}{\draw (0,-1.5) node[not] {} -- (0,0); \draw (-1,1) node[xi] {} -- (0,0) node[xi] {} -- (1,1) node[xi] {};}
\DeclareSymbol{Xi4e}{0}{\draw (0,2) node[xi] {} -- (-1,1) node[xi] {} -- (0,0) node[xi] {} -- (1,1) node[xi] {};}
\DeclareSymbol{Xi4e'}{0}{\draw (0,2) node[xi] {} -- (-1,1) node[xi] {} -- (0,-0.2) node[not] {} -- (1,1) node[xi] {};}

\DeclareSymbol{Xitwo}
{0}{\draw[kernels2] (0,0) node[not] {} -- (-1,1) node[not] {}
-- (-2,2) node[not]{} -- (-3,3) node[xi]  {};
\draw[kernels2] (0,0) -- (1,1) node[xi] {};
\draw[kernels2] (-1,1) -- (0,2) node[xi] {};
\draw[kernels2] (-2,2) -- (-1,3) node[xi] {};}

\DeclareSymbol{IXitwo}
{0}{\draw (-.7,1.2) node[xi] {} -- (0,-0.2) -- (.7,1.2) node[xi] {};}
\DeclareSymbol{I1Xitwo}
{0}{\draw[kernels2] (0,0) node[not] {} -- (-1,1) node[xi] {};
\draw[kernels2] (0,0) -- (1,1) node[xi] {};}

\DeclareSymbol{I1Xitwobis}
{0}{\draw[kernels2] (0,0) node[not] {} -- (-1,1) node[xies] {};
\draw[kernels2] (0,0) -- (1,1) node[xies] {};}

\DeclareSymbol{I1Xitwog}
{0}{\draw[kernels2] (0,0) node[not] {} -- (-1,1) node[xies] {};
\draw[kernels2] (0,0) -- (1,1) node[xi] {};}

\DeclareSymbol{cI1Xitwo}
{0}{\draw[kernels2] (0,0) node[not] {} -- (-1,1) node[xic] {};
\draw[kernels2] (0,0) -- (1,1) node[xi] {};}

\DeclareSymbol{I1IXi3}{0}{\draw (0,0) node[xi] {} -- (-1,1) ; 
\draw[kernels2] (-1,1) node[not] {} -- (0,2) node[xi] {};
\draw[kernels2] (-1,1) node[not] {} -- (-2,2) node[xi] {};}

\DeclareSymbol{I1Xi3c}{-1}{\draw[kernels2](0,1.5) node[xi] {} -- (0,0) node[not] {}; \draw (-1,1) node[xi] {} -- (0,0) ; \draw[kernels2] (0,0) -- (1,1) node[xi] {};}

\DeclareSymbol{I1Xi3cbis}{-1}{\draw[kernels2](0,1.5) node[xies] {} -- (0,0) node[not] {}; \draw (-1,1) node[xies] {} -- (0,0) ; \draw[kernels2] (0,0) -- (1,1) node[xies] {};}

\DeclareSymbol{I1IXi3b}{0}{\draw[kernels2] (0,0) node[not] {} -- (-1,1) ; \draw[kernels2] (0,0)   -- (1,1) node[xi] {} ;
\draw (-1,1) node[xi] {} -- (0,2) node[xi] {};
}

\DeclareSymbol{I1IXi3c}{0}{\draw[kernels2] (0,0) node[not] {} -- (-1,1) ; \draw[kernels2] (0,0)   -- (1,1) node[xi] {} ;
\draw[kernels2] (-1,1) node[not] {} -- (0,2) node[xi] {};
\draw[kernels2] (-1,1) node[not] {} -- (-2,2) node[xi] {};}

\DeclareSymbol{I1IXi3cbis}{0}{\draw[kernels2] (0,0) node[not] {} -- (-1,1) ; \draw[kernels2] (0,0)   -- (1,1) node[xies] {} ;
\draw[kernels2] (-1,1) node[not] {} -- (0,2) node[xies] {};
\draw[kernels2] (-1,1) node[not] {} -- (-2,2) node[xies] {};}

\DeclareSymbol{I1Xi}{0}{\draw[kernels2] (0,0) node[not] {} -- (-1,1)  node[xi] {} ;}

\DeclareSymbol{I1Xi4a}{2}{\draw[kernels2] (0,0) node[not] {} -- (-1,1) ; \draw[kernels2] (0,0) node[not] {} -- (1,1) node[xi] {} ;
\draw (-1,1) node[xi] {} -- (0,2) node[xi] {} -- (-1,3) node[xi] {};}

\DeclareSymbol{cI1Xi4a}{2}{\draw[kernels2] (0,0) node[not] {} -- (-1,1) ; \draw[kernels2] (0,0) node[not] {} -- (1,1) node[xic] {} ;
\draw (-1,1) node[xic] {} -- (0,2) node[xi] {} -- (-1,3) node[xi] {};}

\DeclareSymbol{I1Xi4b}{2}{\draw (0,0) node[xi] {} -- (-1,1) node[xi] {} -- (0,2) ; \draw[kernels2] (0,2) node[not] {} -- (-1,3) node[xi] {};\draw[kernels2] (0,2)  -- (1,3) node[xi] {};
}

\DeclareSymbol{cI1Xi4b}{2}{\draw (0,0) node[xic] {} -- (-1,1) node[xic] {} -- (0,2) ; \draw[kernels2] (0,2) node[not] {} -- (-1,3) node[xi] {};\draw[kernels2] (0,2)  -- (1,3) node[xi] {};
}

\DeclareSymbol{I1Xi4c}{2}{\draw (0,0) node[xi] {} -- (-1,1) node[not] {}; \draw[kernels2] (-1,1) -- (0,2) ; 
\draw[kernels2] (-1,1) -- (-2,2) node[xi] {} ;
\draw (0,2) node[xi] {} -- (-1,3) node[xi] {};}

\DeclareSymbol{cI1Xi4c}{2}{\draw (0,0) node[xic] {} -- (-1,1) node[not] {}; \draw[kernels2] (-1,1) -- (0,2) ; 
\draw[kernels2] (-1,1) -- (-2,2) node[xic] {} ;
\draw (0,2) node[xi] {} -- (-1,3) node[xi] {};}
\DeclareMathOperator*{\bplus}{
\tikzexternaldisable
\mathchoice{\MHmakebox{.17ex}{1.5em}{0.48em}}{\MHmakebox{.12ex}{1.1em}{0.28em}}{}{}
\tikzexternalenable
}
\DeclareSymbol{I1Xi4ab}{2}{\draw[kernels2] (0,0) node[not] {} -- (-1,1) ; \draw[kernels2] (0,0) node[not] {} -- (1,1) node[xi] {};\draw (-1,1) node[xi] {} -- (0,2) ; \draw[kernels2] (0,2) node[not] {} -- (-1,3) node[xi] {};\draw[kernels2] (0,2)  -- (1,3) node[xi] {}; }

\DeclareSymbol{cI1Xi4ab}{2}{\draw[kernels2] (0,0) node[not] {} -- (-1,1) ; \draw[kernels2] (0,0) node[not] {} -- (1,1) node[xic] {};\draw (-1,1) node[xic] {} -- (0,2) ; \draw[kernels2] (0,2) node[not] {} -- (-1,3) node[xi] {};\draw[kernels2] (0,2)  -- (1,3) node[xi] {}; }

\DeclareSymbol{I1Xi4bc}{2}{\draw (0,0) node[xi] {} -- (-1,1) node[not] {}; \draw[kernels2] (-1,1) -- (0,2) ; 
\draw[kernels2] (-1,1) -- (-2,2) node[xi] {} ; \draw[kernels2] (0,2) node[not] {} -- (-1,3) node[xi] {};\draw[kernels2] (0,2)  -- (1,3) node[xi] {};
}

\DeclareSymbol{cI1Xi4bc}{2}{\draw (0,0) node[xic] {} -- (-1,1) node[not] {}; \draw[kernels2] (-1,1) -- (0,2) ; 
\draw[kernels2] (-1,1) -- (-2,2) node[xic] {} ; \draw[kernels2] (0,2) node[not] {} -- (-1,3) node[xi] {};\draw[kernels2] (0,2)  -- (1,3) node[xi] {};
}

\DeclareSymbol{I1Xi4abcc1}{2}{\draw[kernels2] (0,0) node[not] {} -- (-1,1) node[not] {}
-- (-2,2) node[not]{} -- (-3,3) node[xic]  {};
\draw[kernels2] (0,0) -- (1,1) node[xic] {};
\draw[kernels2] (-1,1) -- (0,2) node[xi] {};
\draw[kernels2] (-2,2) -- (-1,3) node[xi] {};
}

\DeclareSymbol{I1Xi4abcc1b}{2}{\draw[kernels2] (0,0) node[not] {} -- (-1,1) node[not] {}
-- (-2,2) node[not]{} -- (-3,3) node[xi]  {};
\draw[kernels2] (0,0) -- (1,1) node[xic] {};
\draw[kernels2] (-1,1) -- (0,2) node[xic] {};
\draw[kernels2] (-2,2) -- (-1,3) node[xi] {};
}

\DeclareSymbol{I1Xi4abcc2}{2}{\draw[kernels2] (0,0) node[not] {} -- (-1,1) node[not] {}
-- (-2,2) node[not]{} -- (-3,3) node[xic]  {};
\draw[kernels2] (0,0) -- (1,1) node[xi] {};
\draw[kernels2] (-1,1) -- (0,2) node[xi] {};
\draw[kernels2] (-2,2) -- (-1,3) node[xic] {};
}

\DeclareSymbol{I1Xi4ac}{2}{\draw[kernels2] (0,0) node[not] {} -- (-1,1) ; \draw[kernels2] (0,0) node[not] {} -- (1,1) node[xi] {}; 
\draw[kernels2] (-1,1) node[not] {} -- (0,2) ; 
\draw[kernels2] (-1,1) -- (-2,2) node[xi] {} ;
\draw (0,2) node[xi] {} -- (-1,3) node[xi] {};}

\DeclareSymbol{cI1Xi4ac}{2}{\draw[kernels2] (0,0) node[not] {} -- (-1,1) ; \draw[kernels2] (0,0) node[not] {} -- (1,1) node[xic] {}; 
\draw[kernels2] (-1,1) node[not] {} -- (0,2) ; 
\draw[kernels2] (-1,1) -- (-2,2) node[xic] {} ;
\draw (0,2) node[xi] {} -- (-1,3) node[xi] {};}

\DeclareSymbol{I1Xi4acc1}{2}{\draw[kernels2] (0,0) node[not] {} -- (-1,1) ; \draw[kernels2] (0,0) node[not] {} -- (1,1) node[xic] {}; 
\draw[kernels2] (-1,1) node[not] {} -- (0,2) ; 
\draw[kernels2] (-1,1) -- (-2,2) node[xi] {} ;
\draw (0,2) node[xic] {} -- (-1,3) node[xi] {};}

\DeclareSymbol{I1Xi4acc2}{2}{\draw[kernels2] (0,0) node[not] {} -- (-1,1) ; \draw[kernels2] (0,0) node[not] {} -- (1,1) node[xic] {}; 
\draw[kernels2] (-1,1) node[not] {} -- (0,2) ; 
\draw[kernels2] (-1,1) -- (-2,2) node[xi] {} ;
\draw (0,2) node[xi] {} -- (-1,3) node[xic] {};}

\DeclareSymbol{2I1Xi4}{2}{\draw[kernels2] (0,0) node[not] {} -- (-1,1) node[not] {};
\draw[kernels2] (0,0) -- (1,1) node[not] {};
\draw[kernels2] (-1,1) -- (-1.5,2.5) node[xi] {};
\draw[kernels2] (-1,1) -- (-0.5,2.5) node[xi] {};
\draw[kernels2] (1,1) -- (0.5,2.5) node[xi] {};
\draw[kernels2] (1,1) -- (1.5,2.5) node[xi] {};
}

\DeclareSymbol{2I1Xi4dis}{2}{\draw[kernels2] (0,0) node[not] {} -- (-1,1) node[not] {};
\draw[kernels2] (0,0) -- (1,1) node[not] {};
\draw[kernels2] (-1,1) -- (-1.5,2.5) node[xies] {};
\draw[kernels2] (-1,1) -- (-0.5,2.5) node[xies] {};
\draw[kernels2] (1,1) -- (0.5,2.5) node[xies] {};
\draw[kernels2] (1,1) -- (1.5,2.5) node[xies] {};
}

\DeclareSymbol{2I1Xi4c1}{2}{\draw[kernels2] (0,0) node[not] {} -- (-1,1) node[not] {};
\draw[kernels2] (0,0) -- (1,1) node[not] {};
\draw[kernels2] (-1,1) -- (-1.5,2.5) node[xic] {};
\draw[kernels2] (-1,1) -- (-0.5,2.5) node[xi] {};
\draw[kernels2] (1,1) -- (0.5,2.5) node[xic] {};
\draw[kernels2] (1,1) -- (1.5,2.5) node[xi] {};
}

\DeclareSymbol{2I1Xi4c2}{2}{\draw[kernels2] (0,0) node[not] {} -- (-1,1) node[not] {};
\draw[kernels2] (0,0) -- (1,1) node[not] {};
\draw[kernels2] (-1,1) -- (-1.5,2.5) node[xic] {};
\draw[kernels2] (-1,1) -- (-0.5,2.5) node[xic] {};
\draw[kernels2] (1,1) -- (0.5,2.5) node[xi] {};
\draw[kernels2] (1,1) -- (1.5,2.5) node[xi] {};
}

\DeclareSymbol{2I1Xi4b}{2}{\draw[kernels2] (0,0) node[not] {} -- (-1,1) ;
\draw[kernels2] (0,0) -- (1,1);
\draw (-1,1) node[xi] {} -- (-1,2.5) node[xi] {};
\draw (1,1)  node[xi] {} -- (1,2.5) node[xi] {};
}

\DeclareSymbol{2I1Xi4bb}{2}{\draw[kernels2] (0,0) node[not] {} -- (-1,1) ;
\draw[kernels2] (0,0) -- (1,1);
\draw (-1,1) node[xi] {} -- (-1,2.5) node[xiesf] {};
\draw (1,1)  node[xi] {} -- (1,2.5) node[xic] {};
}

\DeclareSymbol{2I1Xi4c}{2}{\draw[kernels2] (0,0) node[not] {} -- (-1,1);
\draw[kernels2] (0,0) -- (1,1) node[not] {};
\draw (-1,1)  node[xi] {} -- (-1,2.5) node[xi] {};
\draw[kernels2] (1,1) -- (0.4,2.5) node[xi] {};
\draw[kernels2] (1,1) -- (1.6,2.5) node[xi] {};
}

\DeclareSymbol{2I1Xi4cc1}{2}{\draw[kernels2] (0,0) node[not] {} -- (-1,1);
\draw[kernels2] (0,0) -- (1,1) node[not] {};
\draw (-1,1)  node[xic] {} -- (-1,2.5) node[xi] {};
\draw[kernels2] (1,1) -- (0.4,2.5) node[xic] {};
\draw[kernels2] (1,1) -- (1.6,2.5) node[xi] {};
}

\DeclareSymbol{2I1Xi4cc2}{2}{\draw[kernels2] (0,0) node[not] {} -- (-1,1);
\draw[kernels2] (0,0) -- (1,1) node[not] {};
\draw (-1,1)  node[xic] {} -- (-1,2.5) node[xic] {};
\draw[kernels2] (1,1) -- (0.4,2.5) node[xi] {};
\draw[kernels2] (1,1) -- (1.6,2.5) node[xi] {};
}

\DeclareSymbol{Xi4ba}{0}{\draw(-0.5,1.5) node[xi] {} -- (0,0); \draw (-1.5,1) node[xi] {} -- (0,0) node[not] {}; \draw[kernels2] (0,0) -- (1.5,1) node[xi] {};
\draw[kernels2] (0,0) -- (0.5,1.5) node[xi] {} ;}

\DeclareSymbol{Xi4badis}{0}{\draw(-0.5,1.5) node[xies] {} -- (0,0); \draw (-1.5,1) node[xies] {} -- (0,0) node[not] {}; \draw[kernels2] (0,0) -- (1.5,1) node[xies] {};
\draw[kernels2] (0,0) -- (0.5,1.5) node[xies] {} ;}

\DeclareSymbol{Xi4ba1}{0}{\draw(-0.5,1.5) node[xi] {} -- (0,0); \draw (-1.5,1) node[xi] {} -- (0,0) node[not] {}; \draw[kernels2] (0,0) -- (1.5,1) node[xic] {};
\draw[kernels2] (0,0) -- (0.5,1.5) node[xic] {} ;}

\DeclareSymbol{Xi4ba1b}{0}{\draw(-0.5,1.5) node[xic] {} -- (0,0); \draw (-1.5,1) node[xic] {} -- (0,0) node[not] {}; \draw[kernels2] (0,0) -- (1.5,1) node[xi] {};
\draw[kernels2] (0,0) -- (0.5,1.5) node[xi] {} ;}

\DeclareSymbol{Xi4ba1bdiff}{0}{\draw(-0.5,1.5) node[xic] {} -- (0,0); \draw (-1.5,1) node[xic] {} -- (0,0) node[not] {}; \draw (0,0) -- (1.5,1) node[xi] {};
\draw (0,0) -- (0.5,1.5) node[xi] {};
\draw(0,0) node[diff] {};}

\DeclareSymbol{Xi4ba1bb}{0}{\draw(-0.5,1.5) node[xic] {} -- (0,0); \draw (-1.5,1) node[xiesf] {} -- (0,0) node[not] {}; \draw[kernels2] (0,0) -- (1.5,1) node[xi] {};
\draw[kernels2] (0,0) -- (0.5,1.5) node[xi] {} ;}

\DeclareSymbol{Xi4ba2}{0}{\draw(-0.5,1.5) node[xi] {} -- (0,0); \draw (-1.5,1) node[xic] {} -- (0,0) node[not] {}; \draw[kernels2] (0,0) -- (1.5,1) node[xi] {};
\draw[kernels2] (0,0) -- (0.5,1.5) node[xic] {} ;}

\DeclareSymbol{Xi4ba2b}{0}{\draw(-0.5,1.5) node[xi] {} -- (0,0); \draw (-1.5,1) node[xic] {} -- (0,0) node[not] {}; \draw[kernels2] (0,0) -- (1.5,1) node[xi] {};
\draw[kernels2] (0,0) -- (0.5,1.5) node[xiesf] {} ;}

\DeclareSymbol{Xi4ca}{0}{\draw (0,1) -- (-1,2.2) node[xi] {};\draw (0,-0.25) node[xi] {} -- (0,1) ; \draw[kernels2] (0,1) node[not] {} -- (1,2.2) node[xi] {};
\draw[kernels2] (0,1) {} -- (0,2.7) node[xi] {};
}

\DeclareSymbol{Xi4cb}{0}{\draw (-1,1) -- (-2,2) node[xi] {};\draw[kernels2] (0,0)  -- (-1,1) node[xi] {} ; \draw[kernels2] (0,0) node[not] {} -- (1,1) node[xi] {} ; 
\draw (-1,1) node[xi] {} -- (0,2) node[xi] {};}

\DeclareSymbol{Xi4cbb}{0}{\draw (-1,1) -- (-2,2) node[xiesf] {};\draw[kernels2] (0,0)  -- (-1,1) node[xi] {} ; \draw[kernels2] (0,0) node[not] {} -- (1,1) node[xi] {} ; 
\draw (-1,1) node[xi] {} -- (0,2) node[xic] {};}

\DeclareSymbol{Xi4cbc1}{0}{\draw (-1,1) -- (-2,2) node[xic] {};\draw[kernels2] (0,0)  -- (-1,1) node[xic] {} ; \draw[kernels2] (0,0) node[not] {} -- (1,1) node[xi] {} ; 
\draw (-1,1) node[xic] {} -- (0,2) node[xi] {};}

\DeclareSymbol{Xi4cbc2}{0}{\draw (-1,1) -- (-2,2) node[xi] {};\draw[kernels2] (0,0)  -- (-1,1) node[xi] {} ; \draw[kernels2] (0,0) node[not] {} -- (1,1) node[xic] {} ; 
\draw (-1,1) node[xic] {} -- (0,2) node[xi] {};}

\DeclareSymbol{Xi4cab}{0}{\draw (-1,1) -- (-2,2) node[xi] {};\draw[kernels2] (0,0)  -- (-1,1); \draw[kernels2] (0,0) node[not] {} -- (1,1) node[xi] {} ; 
\draw[kernels2] (-1,1)  {} -- (0,2) node[xi] {};
\draw[kernels2] (-1,1) node[not] {} -- (-1,2.5) node[xi] {};
}

\DeclareSymbol{Xi4cabdis}{0}{\draw (-1,1) -- (-2,2) node[xies] {};\draw[kernels2] (0,0)  -- (-1,1); \draw[kernels2] (0,0) node[not] {} -- (1,1) node[xies] {} ; 
\draw[kernels2] (-1,1)  {} -- (0,2) node[xies] {};
\draw[kernels2] (-1,1) node[not] {} -- (-1,2.5) node[xies] {};
}

\DeclareSymbol{Xi4cabc1}{0}{\draw (-1,1) -- (-2,2) node[xi] {};\draw[kernels2] (0,0)  -- (-1,1); \draw[kernels2] (0,0) node[not] {} -- (1,1) node[xic] {} ; 
\draw[kernels2] (-1,1)  {} -- (0,2) node[xic] {};
\draw[kernels2] (-1,1) node[not] {} -- (-1,2.5) node[xi] {};
}

\DeclareSymbol{Xi4cabc2}{0}{\draw (-1,1) -- (-2,2) node[xic] {};\draw[kernels2] (0,0)  -- (-1,1); \draw[kernels2] (0,0) node[not] {} -- (1,1) node[xic] {} ; 
\draw[kernels2] (-1,1)  {} -- (0,2) node[xi] {};
\draw[kernels2] (-1,1) node[not] {} -- (-1,2.5) node[xi] {};
}

\DeclareSymbol{Xi4ea}{1.5}{\draw (-1,2.5) node[xi] {} -- (-1,1) node[xi] {} -- (0,0); 
 \draw[kernels2] (0,0)  -- (1,1) node[xi] {};
\draw[kernels2] (0,0) node[not] {} -- (0,1.5) node[xi] {}; }

\DeclareSymbol{Xi4eac1}{1.5}{\draw (-1,2.5) node[xic] {} -- (-1,1) node[xi] {} -- (0,0); 
 \draw[kernels2] (0,0)  -- (1,1) node[xic] {};
\draw[kernels2] (0,0) node[not] {} -- (0,1.5) node[xi] {}; }

\DeclareSymbol{Xi4eac1b}{1.5}{\draw (-1,2.5) node[xic] {} -- (-1,1) node[xi] {} -- (0,0); 
 \draw[kernels2] (0,0)  -- (1,1) node[xiesf] {};
\draw[kernels2] (0,0) node[not] {} -- (0,1.5) node[xi] {}; }

\DeclareSymbol{Xi4eac2}{1.5}{\draw (-1,2.5) node[xic] {} -- (-1,1) node[xic] {} -- (0,0); 
 \draw[kernels2] (0,0)  -- (1,1) node[xi] {};
\draw[kernels2] (0,0) node[not] {} -- (0,1.5) node[xi] {}; }

\DeclareSymbol{Xi4eact1}{1.5}{\draw (-1,2.5) node[xic] {} -- (-1,1) node[xi] {} -- (0,0); 
 \draw (0,0)  -- (1,1) node[xic] {};
\draw[rho] (0,0) node[not] {} -- (0,1.5) node[xi] {}; }

\DeclareSymbol{Xi4eact2}{1.5}{\draw[rho] (-1,2.5) node[xic] {} -- (-1,1) node[xi] {} -- (0,0); 
 \draw (0,0)  -- (1,1) node[xic] {};
\draw (0,0) node[not] {} -- (0,1.5) node[xi] {}; }

\DeclareSymbol{Xi4eabis}{1.5}{\draw (-1,2.5) node[xi] {} -- (-1,1) ; \draw[kernels2] (-1,1) node[xi] {} -- (0,0); 
 \draw (0,0)  -- (1,1) node[xi] {};
\draw[kernels2] (0,0) node[not] {} -- (0,1.5) node[xi] {}; }

\DeclareSymbol{Xi4eabisc1}{1.5}{\draw (-1,2.5) node[xic] {} -- (-1,1) ; \draw[kernels2] (-1,1) node[xi] {} -- (0,0); 
 \draw (0,0)  -- (1,1) node[xi] {};
\draw[kernels2] (0,0) node[not] {} -- (0,1.5) node[xic] {}; }

\DeclareSymbol{Xi4eabisc1b}{1.5}{\draw (-1,2.5) node[xic] {} -- (-1,1) ; \draw[kernels2] (-1,1) node[xi] {} -- (0,0); 
 \draw (0,0)  -- (1,1) node[xi] {};
\draw[kernels2] (0,0) node[not] {} -- (0,1.5) node[xiesf] {}; }

\DeclareSymbol{Xi4eabisc1bis}{1.5}{\draw (-1,2.5) node[xi] {} -- (-1,1) ; \draw[kernels2] (-1,1) node[xi] {} -- (0,0); 
 \draw (0,0)  -- (1,1) node[xi] {};
\draw[kernels2] (0,0) node[not] {} -- (0,1.5) node[xi] {};
\draw (-2,1) node[] {\tiny{$i$}};
\draw (-2,2.5) node[] {\tiny{$\ell$}};
\draw (2,1) node[] {\tiny{$k$}};
\draw (0,2.5) node[] {\tiny{$j$}};
 }

\DeclareSymbol{Xi4eabisc1tris}{1.5}{\draw (-1,2.5) node[xi] {} -- (-1,1) ; \draw[kernels2] (-1,1) node[xi] {} -- (0,0); 
 \draw (0,0)  -- (1,1) node[xi] {};
\draw[kernels2] (0,0) node[not] {} -- (0,1.5) node[xi] {};
\draw (-2,1) node[] {\tiny{i}};
\draw (-2,2.5) node[] {\tiny{j}};
\draw (2,1) node[] {\tiny{j}};
\draw (0,2.5) node[] {\tiny{i}};
 }

\DeclareSymbol{Xi4eabisc1quater}{1.5}{\draw (-1,2.5) node[xic] {} -- (-1,1) ; \draw[kernels2] (-1,1) node[xi] {} -- (0,0); 
 \draw (0,0)  -- (1,1) node[xic] {};
\draw[kernels2] (0,0) node[not] {} -- (0,1.5) node[xi] {};
 }

\DeclareSymbol{Xi4eabisc2}{1.5}{\draw (-1,2.5) node[xic] {} -- (-1,1) ; \draw[kernels2] (-1,1) node[xi] {} -- (0,0); 
 \draw (0,0)  -- (1,1) node[xic] {};
\draw[kernels2] (0,0) node[not] {} -- (0,1.5) node[xi] {}; }

\DeclareSymbol{Xi4eabisc2l}{1.5}{\draw (-1,2.5) node[xiesf] {} -- (-1,1) ; \draw[kernels2] (-1,1) node[xi] {} -- (0,0); 
 \draw (0,0)  -- (1,1) node[xic] {};
\draw[kernels2] (0,0) node[not] {} -- (0,1.5) node[xi] {}; }

\DeclareSymbol{Xi4eabisc2r}{1.5}{\draw (-1,2.5) node[xic] {} -- (-1,1) ; \draw[kernels2] (-1,1) node[xi] {} -- (0,0); 
 \draw (0,0)  -- (1,1) node[xiesf] {};
\draw[kernels2] (0,0) node[not] {} -- (0,1.5) node[xi] {}; }

\DeclareSymbol{Xi4eabisc3}{1.5}{\draw (-1,2.5) node[xic] {} -- (-1,1) ; \draw[kernels2] (-1,1) node[xic] {} -- (0,0); 
 \draw (0,0)  -- (1,1) node[xi] {};
\draw[kernels2] (0,0) node[not] {} -- (0,1.5) node[xi] {}; }

\DeclareSymbol{Xi4eb}{0}{
\draw[kernels2] (0,2) node[xi] {} -- (-1,1) ; \draw[kernels2] (-2,2)  node[xi] {} -- (-1,1) ; \draw (-1,1)  node[not] {} -- (0,0); 
 \draw (0,0) node[xi] {}  -- (1,1) node[xi] {};
}

\DeclareSymbol{Xi4eab}{1.5}{\draw[kernels2] (-1,2.5) node[xi] {} -- (-1,1) ; \draw[kernels2] (-2,2)  node[xi] {} -- (-1,1) ; \draw (-1,1)  node[not] {} -- (0,0); 
 \draw[kernels2] (0,0)  -- (1,1) node[xi] {};
\draw[kernels2] (0,0) node[not] {} -- (0,1.5) node[xi] {}; 
}

\DeclareSymbol{Xi4eabdis}{1.5}{\draw[kernels2] (-1,2.5) node[xies] {} -- (-1,1) ; \draw[kernels2] (-2,2)  node[xies] {} -- (-1,1) ; \draw (-1,1)  node[not] {} -- (0,0); 
 \draw[kernels2] (0,0)  -- (1,1) node[xies] {};
\draw[kernels2] (0,0) node[not] {} -- (0,1.5) node[xies] {}; 
}

\DeclareSymbol{Xi4eabc1}{1.5}{\draw[kernels2] (-1,2.5) node[xic] {} -- (-1,1) ; \draw[kernels2] (-2,2)  node[xi] {} -- (-1,1) ; \draw (-1,1)  node[not] {} -- (0,0); 
 \draw[kernels2] (0,0)  -- (1,1) node[xic] {};
\draw[kernels2] (0,0) node[not] {} -- (0,1.5) node[xi] {}; 
}

\DeclareSymbol{Xi4eabc2}{1.5}{\draw[kernels2] (-1,2.5) node[xi] {} -- (-1,1) ; \draw[kernels2] (-2,2)  node[xi] {} -- (-1,1) ; \draw (-1,1)  node[not] {} -- (0,0); 
 \draw[kernels2] (0,0)  -- (1,1) node[xic] {};
\draw[kernels2] (0,0) node[not] {} -- (0,1.5) node[xic] {}; 
}

\DeclareSymbol{Xi4eabbis}{1.5}{\draw[kernels2] (-1,2.5) node[xi] {} -- (-1,1) ; \draw[kernels2] (-2,2)  node[xi] {} -- (-1,1) ; \draw[kernels2] (-1,1)  node[not] {} -- (0,0); 
 \draw (0,0)  -- (1,1) node[xi] {};
\draw[kernels2] (0,0) node[not] {} -- (0,1.5) node[xi] {}; 
}

\DeclareSymbol{Xi4eabbisc1}{1.5}{\draw[kernels2] (-1,2.5) node[xic] {} -- (-1,1) ; \draw[kernels2] (-2,2)  node[xi] {} -- (-1,1) ; \draw[kernels2] (-1,1)  node[not] {} -- (0,0); 
 \draw (0,0)  -- (1,1) node[xic] {};
\draw[kernels2] (0,0) node[not] {} -- (0,1.5) node[xi] {}; 
}

\DeclareSymbol{Xi4eabbisc1perm}{1.5}{\draw[kernels2] (-1,2.5) node[xi] {} -- (-1,1) ; \draw[kernels2] (-2,2)  node[xic] {} -- (-1,1) ; \draw[kernels2] (-1,1)  node[not] {} -- (0,0); 
 \draw (0,0)  -- (1,1) node[xic] {};
\draw[kernels2] (0,0) node[not] {} -- (0,1.5) node[xi] {}; 
}

\DeclareSymbol{Xi4eabbisc2}{1.5}{\draw[kernels2] (-1,2.5) node[xi] {} -- (-1,1) ; \draw[kernels2] (-2,2)  node[xi] {} -- (-1,1) ; \draw[kernels2] (-1,1)  node[not] {} -- (0,0); 
 \draw (0,0)  -- (1,1) node[xic] {};
\draw[kernels2] (0,0) node[not] {} -- (0,1.5) node[xic] {}; 
}

\DeclareSymbol{Xi2cbis}{0}{\draw[kernels2] (0,1) -- (0.8,2.2) node[xi] {};\draw[kernels2] (0,-0.25) node[not] {} -- (0,1); \draw[kernels2] (0,1) node[not] {} -- (-0.8,2.2) node[xi] {};}

\DeclareSymbol{Xi2cbis1}{0}{\draw (0,1) -- (-0.8,2.2) node[xi] {};\draw[kernels2] (0,-0.25) node[not] {} -- (0,1) node[xi] {}; }

\DeclareSymbol{Xi2Xbis}{-2}{\draw[kernels2] (0,-0.25)  -- (-1,1) ; \draw (-1,1) node[xix] {};
\draw[kernels2] (0,-0.25) node[not] {} -- (1,1) node[xi] {};}

\DeclareSymbol{XXi2bis}{-2}{\draw[kernels2] (0,-0.25) -- (-1,1) node[xi] {};
\draw[kernels2] (0,-0.25) node[X] {} -- (1,1) node[xi] {};}

\DeclareSymbol{I1XiIXi}{0}{\draw[kernels2] (0,-0.25) -- (1,1) node[xi] {};
\draw (0,-0.25) node[not] {} -- (-1,1) node[xi] {};}

\DeclareSymbol{I1XiIXib}{0}{\draw  (0,-0.25) node[xi] {} -- (0,1) node[not] {};
\draw[kernels2] (0,1) -- (0,2.25) ; \draw (0,2.25) node[xi]{}; }

\DeclareSymbol{I1XiIXic}{0}{
\draw[kernels2] (0,0) -- (1,1) node[xi] {} ; 
\draw[kernels2] (0,0) node[not] {}  -- (-1,1) node[not] {} -- (0,2) node[xi] {};
}

\DeclareSymbol{thin}{1.4}{\draw[pagebackground] (-0.3,0) -- (0.3,0); \draw  (0,0) -- (0,2);}
\DeclareSymbol{thin2}{1.4}{\draw[pagebackground] (-0.3,0) -- (0.3,0); \draw[tinydots]  (0,0) -- (0,2);}

\DeclareSymbol{thick}{1.4}{\draw[pagebackground] (-0.3,0) -- (0.3,0); \draw[kernels2]  (0,0) -- (0,2);}

\DeclareSymbol{thick2}{1.4}{\draw[pagebackground] (-0.3,0) -- (0.3,0); \draw[kernels2,tinydots]  (0,0) -- (0,2);}

\DeclareSymbol{Xi4ind}{2}{\draw (0,0) node[xi,label={[label distance=-0.2em]right: \scriptsize  $ i $}]  { } -- (-1,1) node[xi,label={[label distance=-0.2em]left: \scriptsize  $ j $}] {} -- (0,2) node[xi,label={[label distance=-0.2em]right: \scriptsize  $ k $}] {} -- (-1,3) node[xi,label={[label distance=-0.2em]left: \scriptsize  $ \ell $}] {};}

\DeclareSymbol{Xi4c1}{2}{\draw (0,0) node[xic] {} -- (-1,1) node[xi] {} -- (0,2) node[xic] {} -- (-1,3) node[xi] {};} 
\DeclareSymbol{IXi2ex}{0}{\draw (0,-0.25) node[xie] {} -- (-1,1) node[xi] {} ; \draw (0,-0.25)-- (1,1) node[xi] {};}
\DeclareSymbol{IXi2ex1}{0}{\draw (0,-0.25) node[xie] {} -- (-1,1) node[xi] {} -- (0,2) node[xi] {};}

\DeclareSymbol{Xi4b1}{0}{\draw(0,1.5) node[xic] {} -- (0,0); \draw (-1,1) node[xic] {} -- (0,0) node[xi] {} -- (1,1) node[xi] {};}

\DeclareSymbol{Xi4ec1}{0}{\draw (0,2) node[xi] {} -- (-1,1) node[xic] {} -- (0,0) node[xic] {} -- (1,1) node[xi] {};}
\DeclareSymbol{Xi4ec2}{0}{\draw (0,2) node[xic] {} -- (-1,1) node[xi] {} -- (0,0) node[xic] {} -- (1,1) node[xi] {};}
\DeclareSymbol{Xi4ec3}{0}{\draw (0,2) node[xic] {} -- (-1,1) node[xic] {} -- (0,0) node[xi] {} -- (1,1) node[xi] {};}

\DeclareSymbol{I1Xi4ac1}{2}{\draw[kernels2] (0,0) node[not] {} -- (-1,1) ; \draw[kernels2] (0,0) node[not] {} -- (1,1) node[xic] {} ;
\draw (-1,1) node[xi] {} -- (0,2) node[xic] {} -- (-1,3) node[xi] {};}

\DeclareSymbol{I1Xi4ac2}{2}{\draw[kernels2] (0,0) node[not] {} -- (-1,1) ; \draw[kernels2] (0,0) node[not] {} -- (1,1) node[xic] {} ;
\draw (-1,1) node[xi] {} -- (0,2) node[xi] {} -- (-1,3) node[xic] {};}

\DeclareSymbol{I1Xi4bp}{2}{\draw (0,0) node[not] {} -- (-1,1) node[xi] {} -- (0,2) ; \draw[kernels2] (0,2) node[not] {} -- (-1,3) node[xi] {};\draw[kernels2] (0,2)  -- (1,3) node[xi] {};
}

\DeclareSymbol{I1Xi4bc1}{2}{\draw (0,0) node[xic] {} -- (-1,1) node[xi] {} -- (0,2) ; \draw[kernels2] (0,2) node[not] {} -- (-1,3) node[xi] {};\draw[kernels2] (0,2)  -- (1,3) node[xic] {};
}

\DeclareSymbol{I1Xi4bc2}{2}{\draw (0,0) node[xic] {} -- (-1,1) node[xi] {} -- (0,2) ; \draw[kernels2] (0,2) node[not] {} -- (-1,3) node[xic] {};\draw[kernels2] (0,2)  -- (1,3) node[xi] {};
}

\DeclareSymbol{I1Xi4cp}{2}{\draw (0,0) node[not] {} -- (-1,1) node[not] {}; \draw[kernels2] (-1,1) -- (0,2) ; 
\draw[kernels2] (-1,1) -- (-2,2) node[xi] {} ;
\draw (0,2) node[xi] {} -- (-1,3) node[xi] {};}

\DeclareSymbol{I1Xi4cc1}{2}{\draw (0,0) node[xic] {} -- (-1,1) node[not] {}; \draw[kernels2] (-1,1) -- (0,2) ; 
\draw[kernels2] (-1,1) -- (-2,2) node[xi] {} ;
\draw (0,2) node[xic] {} -- (-1,3) node[xi] {};}

\DeclareSymbol{I1Xi4cc2}{2}{\draw (0,0) node[xic] {} -- (-1,1) node[not] {}; \draw[kernels2] (-1,1) -- (0,2) ; 
\draw[kernels2] (-1,1) -- (-2,2) node[xi] {} ;
\draw (0,2) node[xi] {} -- (-1,3) node[xic] {};}

\DeclareSymbol{I1Xi4abc1}{2}{\draw[kernels2] (0,0) node[not] {} -- (-1,1) ; \draw[kernels2] (0,0) node[not] {} -- (1,1) node[xic] {};\draw (-1,1) node[xi] {} -- (0,2) ; \draw[kernels2] (0,2) node[not] {} -- (-1,3) node[xic] {};\draw[kernels2] (0,2)  -- (1,3) node[xi] {}; }

\DeclareSymbol{I1Xi4abc2}{2}{\draw[kernels2] (0,0) node[not] {} -- (-1,1) ; \draw[kernels2] (0,0) node[not] {} -- (1,1) node[xic] {};\draw (-1,1) node[xi] {} -- (0,2) ; \draw[kernels2] (0,2) node[not] {} -- (-1,3) node[xi] {};\draw[kernels2] (0,2)  -- (1,3) node[xic] {}; }

\DeclareSymbol{R1}{0}{\draw (-1,1) node[xi] {} -- (0,0) node[not] {};
\draw[kernels2] (0,1.5) node[xic] {} -- (0,0) -- (1,1) node[xic] {};}
\DeclareSymbol{R2}{0}{\draw (-1,1) node[xic] {} -- (0,0) node[not] {};
\draw[kernels2] (0,1.5)  {} -- (0,0) -- (1,1)  {};
\draw (0,1.5) node[xi] {};
\draw (1,1) node[xic] {};
}
\DeclareSymbol{R3}{1}{\draw[kernels2] (-1,1.5)  {} -- (0,0) node[not] {} -- (1,1.5);
\draw (-1,1.5) node[xi] {};
\draw[kernels2] (0,3) {} -- (1,1.5) -- (2,3)  {};
\draw  (0,3) node[xic] {} ;
\draw (2,3) node[xic] {};}
\DeclareSymbol{R4}{1}{\draw[kernels2] (-1,1.5) node[xic] {} -- (0,0) node[not] {} -- (1,1.5);
\draw[kernels2] (0,3) {} -- (1,1.5) -- (2,3) node[xic] {};
\draw (0,3) node[xi] {};}

\DeclareSymbol{I1Xi4bcp}{2}{\draw (0,0) node[not] {} -- (-1,1) node[not] {}; \draw[kernels2] (-1,1) -- (0,2) ; 
\draw[kernels2] (-1,1) -- (-2,2) node[xi] {} ; \draw[kernels2] (0,2) node[not] {} -- (-1,3) node[xi] {};\draw[kernels2] (0,2)  -- (1,3) node[xi] {};
}

\DeclareSymbol{I1Xi4bcc1}{2}{\draw (0,0) node[xic] {} -- (-1,1) node[not] {}; \draw[kernels2] (-1,1) -- (0,2) ; 
\draw[kernels2] (-1,1) -- (-2,2) node[xi] {} ; \draw[kernels2] (0,2) node[not] {} -- (-1,3) node[xi] {};\draw[kernels2] (0,2)  -- (1,3) node[xic] {};
}

\DeclareSymbol{I1Xi4bcc2}{2}{\draw (0,0) node[xic] {} -- (-1,1) node[not] {}; \draw[kernels2] (-1,1) -- (0,2) ; 
\draw[kernels2] (-1,1) -- (-2,2) node[xi] {} ; \draw[kernels2] (0,2) node[not] {} -- (-1,3) node[xic] {};\draw[kernels2] (0,2)  -- (1,3) node[xi] {};
} 

\DeclareSymbol{2I1Xi4bc1}{2}{\draw[kernels2] (0,0) node[not] {} -- (-1,1) ;
\draw[kernels2] (0,0) -- (1,1);
\draw (-1,1) node[xic] {} -- (-1,2.5) node[xi] {};
\draw (1,1)  node[xic] {} -- (1,2.5) node[xi] {};
}

\DeclareSymbol{2I1Xi4bc2}{2}{\draw[kernels2] (0,0) node[not] {} -- (-1,1) ;
\draw[kernels2] (0,0) -- (1,1);
\draw (-1,1) node[xi] {} -- (-1,2.5) node[xic] {};
\draw (1,1)  node[xic] {} -- (1,2.5) node[xi] {};
}

\DeclareSymbol{diff2I1Xi4bc2}{2}{\draw (0,0) node[diff] {} -- (-1,1) ;
\draw (0,0) -- (1,1);
\draw (-1,1) node[xi] {} -- (-1,2.5) node[xic] {};
\draw (1,1)  node[xic] {} -- (1,2.5) node[xi] {};
}

\DeclareSymbol{2I1Xi4bc3}{2}{\draw[kernels2] (0,0) node[not] {} -- (-1,1) ;
\draw[kernels2] (0,0) -- (1,1);
\draw (-1,1) node[xic] {} -- (-1,2.5) node[xic] {};
\draw (1,1)  node[xi] {} -- (1,2.5) node[xi] {};
}

\DeclareSymbol{Xi41}{0}{\draw (0,1) -- (0.8,2.2) node[xic] {};\draw (0,-0.25) node[xi] {} -- (0,1) node[xi] {} -- (-0.8,2.2) node[xic] {};} 

\DeclareSymbol{Xi42}{0}{\draw (0,1) -- (0.8,2.2) node[xi] {};\draw (0,-0.25) node[xic] {} -- (0,1) node[xi] {} -- (-0.8,2.2) node[xic] {};}

\DeclareSymbol{Xi4ca1}{0}{\draw (0,1) -- (-1,2.2) node[xic] {};\draw (0,-0.25) node[xi] {} -- (0,1) ; \draw[kernels2] (0,1) node[not] {} -- (1,2.2) node[xic] {};
\draw[kernels2] (0,1) {} -- (0,2.7) node[xi] {};
}

\DeclareSymbol{Xi4ca2}{0}{\draw (0,1) -- (-1,2.2) node[xi] {};\draw (0,-0.25) node[xi] {} -- (0,1) ; \draw[kernels2] (0,1) node[not] {} -- (1,2.2) node[xic] {};
\draw[kernels2] (0,1) {} -- (0,2.7) node[xic] {};
}

\DeclareSymbol{Xi4cap}{0}{\draw (0,1) -- (-1,2.2) node[xi] {};\draw (0,-0.25) node[not] {} -- (0,1) ; \draw[kernels2] (0,1) node[not] {} -- (1,2.2) node[xi] {};
\draw[kernels2] (0,1) {} -- (0,2.7) node[xi] {};
}

\DeclareSymbol{Xi3a}{0}{
 \draw (-1,1)  node[xi] {} -- (0,0); 
 \draw (0,0) node[xi] {}  -- (1,1) node[xi] {};
}

\DeclareSymbol{Xi4ebc1}{0}{
\draw[kernels2] (0,2) node[xi] {} -- (-1,1) ; \draw[kernels2] (-2,2)  node[xic] {} -- (-1,1) ; \draw (-1,1)  node[not] {} -- (0,0); 
 \draw (0,0) node[xic] {}  -- (1,1) node[xi] {};
}

\DeclareSymbol{Xi4ebc2}{0}{
\draw[kernels2] (0,2) node[xi] {} -- (-1,1) ; \draw[kernels2] (-2,2)  node[xi] {} -- (-1,1) ; \draw (-1,1)  node[not] {} -- (0,0); 
 \draw (0,0) node[xic] {}  -- (1,1) node[xic] {};
}

\DeclareSymbol{Xi2cbispex}{0}{\draw[kernels2] (0,1) -- (0.8,2.2) node[xi] {};\draw (0,-0.25) node[xie] {} -- (0,1); \draw[kernels2] (0,1) node[not] {} -- (-0.8,2.2) node[xi] {};}

\DeclareSymbol{Xi2cbis1p}{0}{\draw (0,1) -- (-0.8,2.2) node[xi] {};\draw (0,-0.25) node[not] {} -- (0,1) node[xi] {}; }

\DeclareSymbol{Xi2Xp}{-2}{\draw (0,-0.25) node[not] {} -- (-1,1) node[xix] {};} 

\DeclareSymbol{I1XiIXib}{0}{\draw  (0,-0.25) node[xi] {} -- (0,1) node[not] {};
\draw[kernels2] (0,1) -- (0,2.25) ; \draw (0,2.25) node[xi]{}; }

\DeclareSymbol{IXi2b}{0}{\draw  (0,-0.25) node[xi] {} -- (0,1) node[not] {};
\draw (0,1) -- (0,2.25) ; \draw (0,2.25) node[xi]{}; }

\DeclareSymbol{IXi2bex}{0}{\draw  (0,-0.25) node[xi] {} -- (0,1) node[xie] {};
\draw (0,1) -- (0,2.25) ; \draw (0,2.25) node[xi]{}; }

 \def\1{\mathbf{\symbol{1}}}

\def\one{\mathbf{1}}
\def\eps{\varepsilon}

\DeclareSymbol{diff}{0}{
\draw (0,0.5) node[diff] {};
}

\DeclareSymbol{diff1}{0}{
\draw (0,0.5) node[diff1] {};
}

\DeclareSymbol{diff2}{0}{
\draw (0,0.5) node[diff2] {};
}

\DeclareSymbol{geo}{0}{
\draw (0,0) node[diff] {};
\draw (0.3,0) node[diff] {};
}

\DeclareSymbol{generic}{0}{
\draw (0,0.6) node[xi] {};
}

\DeclareSymbol{g}{0}{
\draw (0,0.6) node[g] {};
}

\DeclareSymbol{Ito}{0}{
\draw (0,0.6) node[xies] {};
}

\DeclareSymbol{Itob}{0}{
\draw (0,0.6) node[xiesf] {};
}

\DeclareSymbol{greycirc}{0}{
\draw (0,0.3) node[xi] {};
}

\DeclareSymbol{not}{0}{
\draw (0,0.6) node[not] {};
\draw[tinydots] (0,0.6) circle (0.8);
}

\DeclareSymbol{genericb}{0}{
\draw (0,0.6) node[xic] {};
}

\DeclareSymbol{bluecirc}{0}{
\draw (0,0.3) node[xic] {};
}

\DeclareSymbol{genericxix}{0}{
\draw (0,0.6) node[xix] {};
}

\DeclareSymbol{genericX}{0}{
\draw (0,0.6) node[X] {};
}

\DeclareSymbol{diffIto}{1}{
\draw  (0,2.5) -- (0,0) ;
\draw (0,-0.1) node[diff] {};
\draw (0,2.5) node[xies] {};
}
\DeclareSymbol{Itodiff}{2}{
\draw(0,2.9) -- (0,-0.2);
\draw (0,2.9) node[diff] {};
\draw (0,-0.1) node[xies] {};
}

\DeclareSymbol{diffgeneric}{1}{
\draw  (0,2.5) -- (0,0) ;
\draw (0,-0.1) node[diff] {};
\draw (0,2.5) node[xi] {};
}

\DeclareSymbol{genericdiff}{2}{
\draw(0,2.9) -- (0,-0.2);
\draw (0,2.9) node[diff] {};
\draw (0,-0.1) node[xi] {};
}

\DeclareSymbol{diffdot}{2}{
\draw  (0,3) -- (0,-0.1) ;
\draw (0,3) node[not] {};
\draw (0,-0.1) node[diff] {};
}

\DeclareSymbol{diffdotmini}{0}{
\draw  (0,0) -- (0,1.2) ;
\draw (0,1.2) node[not] {};
\draw (0,0) node[diffmini] {};
}

\DeclareSymbol{dotdiff}{2}{
\draw[kernelsmod]  (0,3) -- (0,-0.1) ;
\draw (0,3) node[diff] {};
\draw (0,-0.1) node[not] {};
}

\DeclareSymbol{dotdiff1}{2}{
\draw[kernelsmod]  (0,3) -- (0,-0.1) ;
\draw (0,3) node[diff1] {};
\draw (0,-0.1) node[not] {};
}

\DeclareSymbol{dotdiff1mini}{0}{
\draw[kernelsmod]  (0,1.2) -- (0,0) ;
\draw (0,1.2) node[diffmini] {};
\draw (0,0) node[not] {};
}

\DeclareSymbol{dotdiff2}{2}{
\draw (0,3) -- (0,-0.1) ;
\draw (0,3) node[diff] {};
\draw (0,-0.1) node[not] {};
}

\DeclareSymbol{dotdiff2mini}{0}{
\draw (0,1.2) -- (0,0) ;
\draw (0,1.2) node[diffmini] {};
\draw (0,0) node[not] {};
}

\DeclareSymbol{dotdiffstraight}{0}{
\draw  (0,3) -- (0,-0.1) ;
\draw (0,3) node[diff] {};
\draw (0,-0.1) node[not] {};
}

\DeclareSymbol{arbre1}{0}{
\draw  (0,0) -- (1.5,1.5) ;
\draw (1.5,1.5) node[not] {};
\draw (0,0) node[not] {};
}

\DeclareSymbol{arbre2}{0}{
\draw  (0,0) -- (1.5,1.5) ;
\draw[kernelsmod] (0,0) -- (-1.5,1.5);
\draw (1.5,1.5) node[not] {};
\draw (0,0) node[not] {};
\draw (-1.5,1.5) node[xi] {};
}

\DeclareSymbol{arbre3}{0}{
\draw  (0,0) -- (1.5,1.5) ;
\draw[kernelsmod] (1.5,1.5) -- (0,3);
\draw (0,0) node[not] {};
\draw (1.5,1.5) node[not] {};
\draw (0,3) node[xi] {};
}

\DeclareSymbol{treeeval}{0}{
\draw (0,0) -- (1,1);
\draw (0,0) node[xi] {};
\draw (1.25,1.25) node[xi] {};
\draw (-0.6,0.6) node[]{\tiny{$i$}};
\draw (0.65,1.85) node[]{\tiny{$j$}};
}

\DeclareSymbol{testeval}{0}{
\draw (0,0) -- (1,1);
\draw (0,0) -- (-1,1);
\draw (0,0) node[xi] {};
\draw (1.25,1.25) node[xi] {};
\draw (-1.25,1.25) node[xi] {};
\draw (-0.6,-0.6) node[]{\tiny{$i$}};
\draw (0.65,1.85) node[]{\tiny{$j$}};
\draw (-1.95,1.85) node[]{\tiny{$k$}};
}

\DeclareSymbol{treeeval2}{0}{
\draw[kernelsmod] (-0.25,-1) -- (1,0.5) ;
\draw[kernelsmod] (1,0.5) -- (-0.25,2);
\draw (1,0.5) node[diff2] {};
\draw (-0.25,-1) node[not] {};
\draw (-0.25,2) node[xi] {};
\draw (-0.6,1.2) node[]{\tiny{1}};
}

\DeclareSymbol{arbreact}{1}{
\draw (0,0) node[not] {};
\draw[kernelsmod] (0,0) -- (1,1);
\draw[kernelsmod] (0,0) -- (-1,1);
\draw (-1,1) node[xic] {};
\draw  (0,2) -- (1,1) ;
\draw (0,2) node[xic] {};
\draw (1,1) node[xi] {};
}

\DeclareSymbol{arbreact1}{0}{
\draw (0,-1.5) -- (0,0);
\draw[kernelsmod] (0,0) -- (1,1);
\draw[kernelsmod] (0,0) -- (-1,1);
\draw  (0,2) -- (1,1) ;
\draw (0,-1.5) node[diff] {};
\draw (0,0) node[not] {};
\draw (-1,1) node[xic] {};
\draw (0,2) node[xic] {};
\draw (1,1) node[xi] {};
}

\DeclareSymbol{arbreact2}{0}{
\draw (0,-0.75) -- (-1,0.5); 
\draw (0,-0.75) -- (1,0.5);
\draw (0,1.5) -- (1,0.5);
\draw (0,1.5) node[xic] {};
\draw (1,0.5) node[xi] {};
\draw (-1,0.5) node[xic] {};
\draw (0,-0.75) node[diff] {};
}

\DeclareSymbol{arbreact3}{0}{
\draw[kernelsmod] (0,-0.75) -- (-1,0.5); 
\draw[kernelsmod] (0,-0.75) -- (1,0.5);
\draw (0,1.75) -- (1,0.5);
\draw (2,1.75) -- (1,0.5);
\draw (0,1.75) node[xic] {};
\draw (1,0.5) node[diff] {};
\draw (-1,0.5) node[xic] {};
\draw (2,1.75) node[xi] {};
\draw (0,-0.75) node[not] {};
}

\DeclareSymbol{pre_im_I1Xitwo}{0}{
\draw[kernels2] (0,-0.3) node[not] {} -- (-0.6,0.7) ;
\draw[kernels2] (0,-0.3) -- (0.6,0.7);
\draw (0,0.9) node[g] {};
}

\DeclareSymbol{pre_im_cI1Xi4ab}{2}{
\draw[kernels2] (0,-1) node[not] {} -- (-0.6,0) ;
\draw[kernels2] (0,-1) -- (0.6,0);
\draw (0,0.2) node[g] {};
\draw (0,0.6) -- (0,1.5);
\draw[kernels2] (0,1.5) node[not] {} -- (-0.6,2.5) ;
\draw[kernels2] (0,1.5) -- (0.6,2.5);
\draw (0,2.7) node[g] {};
}

\DeclareSymbol{pre_im_I1Xi4acc2}{0}{
\draw[kernels2] (-1,-0.5) node[not] {} -- (-1.6,0.5) ;
\draw[kernels2] (-1,-0.5) -- (-0.4,0.5);
\draw[kernels2] (-1,-0.5) -- (0.2,-1.5) node[not] {} ;
\draw (-1,1.1) -- (-1,2);
\draw[kernels2] (0.2,-1.5) -- (0.2,2);
\draw (-1,0.7) node[g] {};
\draw (-0.3,2.2) node[g] {};
}

\DeclareSymbol{pre_im_I1Xi4abcc2}{2}{
\draw[kernels2] (0,-1) node[not] {} -- (-1,0) node[not] {};
\draw[kernels2] (-1,1.2) node[not] {} -- (-1,0);
\draw[kernels2] (-1,1.2) -- (-1.5,2.5);
\draw[kernels2] (-1,1.2) -- (-0.5,2.5);
\draw[kernels2] (-1,0) -- (0.7,2.5);
\draw[kernels2] (0,-1) -- (1.5,2.5);
\draw (-1,2.7) node[g] {};
\draw (1,2.7) node[g] {};
}

\DeclareSymbol{pre_im_2I1Xi4c1}{2}{
\draw[kernels2] (0,-0.5) node[not] {} -- (-1,0.5) node[not] {};
\draw[kernels2] (0,-0.5) -- (1,0.5) node[not] {};
\draw[kernels2] (-1,0.5) node[not] {}-- (-1.7,2);
\draw[kernels2]  (-1,2) -- (1,0.5);
\draw[kernels2] (-1,0.5) -- (1,2);
\draw[kernels2] (1,0.5) -- (1.7,2);
\draw (-1.2,2.2) node[g] {};
\draw (1.2,2.2) node[g] {};
}

\DeclareSymbol{pre_im_Xi4eabisc2}{2}{
\draw[kernels2] (1.2,-0.5) node[not] {} -- (-0.7,0.8) ;
\draw[kernels2] (1.2,-0.5) -- (0.4,0.8);
\draw (0,1.4)  -- (0,2.2);
\draw (1.2,2.2) -- (1.2,-0.6);
\draw (0,1) node[g] {};
\draw (0.6,2.4) node[g] {};
}

\DeclareSymbol{pre_im_Xi4eabisc22}{2}{
\draw (1.2,-0.5) node[not] {} -- (-0.7,0.8) ;
\draw[kernels2] (1.2,-0.5) -- (0.4,0.8);
\draw (0,1.4)  -- (0,2.2);
\draw[kernels2] (1.2,2.2) -- (1.2,-0.6);
\draw (0,1) node[g] {};
\draw (0.6,2.4) node[g] {};
}

\DeclareSymbol{pre_im_Xi4eabisc222}{2}{
\draw[kernels2] (0.4,-0.5) node[not] {} -- (-0.6,1) ;
\draw[kernels2] (1.2,0) -- (0.3,1);
\draw (0,1.1)  -- (0,2.5);
\draw[kernels2] (1.2,2.5) -- (1.2,0) node[not] {} -- (0.4,-0.6);
\draw (0,1.2) node[g] {};
\draw (0.6,2.5) node[g] {};
}

\DeclareSymbol{pre_im_Xi4eabc2}{2}{
\draw (0,-0.5) node[not] {} -- (-1,0.5) node[not] {};
\draw[kernels2] (-1,0.5) -- (-1.5,2);
\draw[kernels2] (0,-0.5)  -- (0.7,2);
\draw[kernels2] (-1,0.5) -- (-0.5,2);
\draw[kernels2] (0,-0.5) -- (1.5,2);
\draw (-1,2.2) node[g] {};
\draw (1,2.2) node[g] {};
}

\DeclareSymbol{pre_im_Xi4eabbisc2}{2}{
\draw[kernels2] (0,-0.5) node[not] {} -- (-1,0.5) node[not] {};
\draw[kernels2] (-1,0.5) -- (-1.5,2);
\draw[kernels2] (0,-0.5)  -- (0.7,2);
\draw[kernels2] (-1,0.5) -- (-0.5,2);
\draw (0,-0.5) -- (1.5,2);
\draw (-1,2.2) node[g] {};
\draw (1,2.2) node[g] {};
}

\DeclareSymbol{pre_im_I1Xi4abcc1}{2}{
\draw[kernels2] (0,-1) node[not] {} -- (-1,0) node[not] {};
\draw[kernels2] (0,1.1) node[not] {} -- (-1,0);
\draw[kernels2] (-1,0) -- (-1.5,2.5);
\draw[kernels2] (0,1.1) node[not] {} -- (-0.5,2.5);
\draw[kernels2] (0,1.1) -- (0.5,2.5);
\draw[kernels2] (0,-1) -- (1.5,2.5);
\draw (-1,2.7) node[g] {};
\draw (1,2.7) node[g] {};
}

\DeclareSymbol{pre_im_Xi4eabc1}{2}{
\draw (0,-0.5) node[not] {} -- (-1,0.5) node[not] {};
\draw[kernels2] (-1,0.5) -- (-1.5,2);
\draw[kernels2] (0,-0.5)  -- (-0.5,2);
\draw[kernels2] (-1,0.5) -- (0.8,2);
\draw[kernels2] (0,-0.5) -- (1.5,2);
\draw (-1,2.2) node[g] {};
\draw (1,2.2) node[g] {};
}

\DeclareSymbol{pre_im_Xi4ba1b}{2}{
\draw[kernels2] (0,0) node[not] {}  -- (1.8,1.5);
\draw[kernels2] (0,0) -- (0.8,1.5);
\draw (0,-0.1) -- (-1.8,1.5);
\draw (0,-0.1) -- (-0.8,1.5);
\draw (-1,1.7) node[g] {};
\draw (1,1.7) node[g] {};
}

\DeclareSymbol{pre_im_Xi4ba2}{2}{
\draw (0,-0.1) node[not] {}  -- (1.8,1.5);
\draw[kernels2] (0,0) -- (0.8,1.5);
\draw (0,-0.1) -- (-1.8,1.5);
\draw[kernels2] (0,0) -- (-0.8,1.5);
\draw (-1,1.7) node[g] {};
\draw (1,1.7) node[g] {};
}

\DeclareSymbol{pre_im_Xi4cabc2}{2}{
\draw[kernels2] (0,-0.5) node[not] {} -- (-1,0.5) node[not] {};
\draw[kernels2] (-1,0.5) -- (-1.5,2);
\draw (-1,0.5)  -- (0.7,2);
\draw[kernels2] (-1,0.5) -- (-0.5,2);
\draw[kernels2] (0,-0.5) -- (1.5,2);
\draw (-1,2.2) node[g] {};
\draw (1,2.2) node[g] {};
}

\DeclareSymbol{pre_im_Xi4cabc1}{2}{
\draw[kernels2] (0,-0.5) node[not] {} -- (-1,0.5) node[not] {};
\draw (-1,0.5) -- (-1.5,2);
\draw[kernels2] (-1,0.5)  -- (0.7,2);
\draw[kernels2] (-1,0.5) -- (-0.5,2);
\draw[kernels2] (0,-0.5) -- (1.5,2);
\draw (-1,2.2) node[g] {};
\draw (1,2.2) node[g] {};
}

\DeclareSymbol{pre_im_Xi4eabbisc1}{2}{
\draw[kernels2] (0,-0.5) node[not] {} -- (-1,0.5) node[not] {};
\draw[kernels2] (-1,0.5) -- (-1.5,2);
\draw[kernels2] (0,-0.5)  -- (-0.5,2);
\draw[kernels2] (-1,0.5) -- (0.8,2);
\draw (0,-0.5) -- (1.5,2);
\draw (-1,2.2) node[g] {};
\draw (1,2.2) node[g] {};
}

\DeclareSymbol{pre_im_1}{0}{
\draw[kernels2] (0,-0.5) node[not] {} -- (-0.6,0.5) ;
\draw[kernels2] (0,-0.5) -- (0.6,0.5);
\draw (0,1.1)  -- (-0.55,2);
\draw (0,1.1)  -- (0.55,2);
\draw (0,0.7) node[g] {};
\draw (0,2.2) node[g] {};
}

\DeclareSymbol{disconnect}{0}{
\draw[kernels2] (0,-0.5) node[not] {} -- (-0.6,0.5) ;
\draw[kernels2] (0,-0.5) -- (0.6,0.5);
\draw (-0.55,1.1)  -- (-0.55,2.3);
\draw (0.55,2.3) -- (0.55,1.5) -- (1.2,1.5) -- (1.2,3.5) -- (0.55,3.5) -- (0.55,2.7);
\draw (0,0.7) node[g] {};
\draw (0,2.5) node[g] {};
}

\DeclareSymbol{pre_im_2}{2}{\draw[kernels2] (0,0) node[not] {} -- (-1,1) node[not] {};
\draw[kernels2] (0,0) -- (1,1) node[not] {};
\draw[kernels2] (-1,1) -- (-1.5,2.5);
\draw[kernels2] (-1,1) -- (-0.5,2.5);
\draw[kernels2] (1,1) -- (0.5,2.5);
\draw[kernels2] (1,1) -- (1.5,2.5);
\draw (-1,2.7) node[g] {};
\draw (1,2.7) node[g] {};
}

\DeclareSymbol{CX_rec}{0}{
\draw [black] (-0.3,1) to (-0.3,-0.3);
\draw [black] (0.3,1) to (0.3,-0.3);
\draw [black] (-0.3,1) to (-0.3,2.3);
\draw [black] (0.3,1) to (0.3,2.3);
\draw (0,1) node[rec] {};
}

\DeclareSymbol{CX_cerc}{0}{
\draw [black] (0,1) to (0,-0.3);
\draw (0,1) node[cerc] {};
}

\DeclareMathAlphabet{\mathpzc}{OT1}{pzc}{m}{it}

\let\eps\varepsilon

\def\eqref#1{(\ref{#1})}

\makeatletter 
\newcommand*{\bigcdot}{}
\DeclareRobustCommand*{\bigcdot}{%
  \mathbin{\mathpalette\bigcdot@{}}%
}
\newcommand*{\bigcdot@scalefactor}{.5}
\newcommand*{\bigcdot@widthfactor}{1.15}
\newcommand*{\bigcdot@}[2]{%
  \sbox0{$#1\vcenter{}$}
  \sbox2{$#1\cdot\m@th$}%
  \hbox to \bigcdot@widthfactor\wd2{%
    \hfil
    \raise\ht0\hbox{%
      \scalebox{\bigcdot@scalefactor}{%
        \lower\ht0\hbox{$#1\bullet\m@th$}%
      }%
    }%
    \hfil
  }%
}
\makeatother

\tcbset
{colframe=boxcolor,colback=symbols!7!pagebackground,coltext=pageforeground,
fonttitle=\bfseries,nobeforeafter,center title,size=fbox,boxsep=1.5pt,
top=0mm,bottom=0mm,boxsep=0mm,tcbox raise base}

\def\two{{\<generic>\kern0.05em\<genericb>}}
\def\twoI{{\<Ito>\kern0.05em\<Itob>}}

\def\mail#1{\burlalt{#1}{mailto:#1}}

\usepackage{thmtools}

\declaretheorem[style=definition]{example}

\begin{document}
\renewcommand\thmcontinues[1]{Continued}

\title{Algebraic deformation for (S)PDEs}
\author{Yvain Bruned$^1$, Dominique Manchon$^2$}
\institute{University of Edinburgh \and
LMBP, CNRS - Universit\'e Clermont-Auvergne, \\
Email:\ \begin{minipage}[t]{\linewidth}
\mail{Yvain.Bruned@ed.ac.uk}, \mail{Dominique.Manchon@uca.fr}.
\end{minipage}}

\maketitle

\begin{abstract}
We introduce a new algebraic framework based on the deformation of pre-Lie products. This allows us to provide a new construction of the algebraic objects at play in Regularity Structures in the works by Bruned, Hairer
and Zambotti (2019) and by Bruned and Schratz (2022) for
deriving a general scheme for dispersive PDEs at low regularity. This construction also explains how the algebraic structure  by Bruned et al. (2019) cited above can be viewed as a deformation of the Butcher-Connes-Kreimer   and the extraction-contraction Hopf algebras. We start by deforming various pre-Lie products via a Taylor deformation and then we apply the Guin-Oudom procedure which gives us an associative product whose adjoint can be compared with known coproducts. This work reveals that pre-Lie products and their deformation can be a central object in the study of (S)PDEs.
\\[.4em]
\end{abstract}
\setcounter{tocdepth}{2}
\tableofcontents

\section{Introduction}
\label{sect:intro}

Decorated trees and their various Hopf algebra structures have revealed themselves to be very efficient for solving (stochastic) partial differential equations ((S)PDEs). Indeed, the theory of singular SPDEs has developed very fast in the last years. This is due to the theory of Regularity Structures introduced by Martin Hairer in \cite{reg}. Many equations whose well-posedness was open are now covered by this theory like the KPZ equation, $ \phi^{4}_3 $ (see \cite{reg}),   $ \phi^{4}_{4-\delta} $ (see \cite{BCCH}), sine-Gordon in the full subcritical regime (see \cite{CHS}), generalised KPZ equation (see \cite{BGHZ}), Yang-Mills in dimension two (see \cite{CCHS}). For surveys on this developement one can look at \cite{EMS,BaiHos} or the textbook \cite{FrizHai}.
The idea of Regularity Structures takes its root in rough paths \cite{Lyo98,Gub04,Gub10}; one wants to approximate the solution of a singular SPDE by a Taylor expansion whose monomials are iterated integrals. This type of expansion is truncated to a sufficiently high order but it could be arbitrarily large when we consider a model close to subcriticality like $ \phi^{4}_{4-\delta} $. This is where algebraic structures enter the game for dealing with the computational challenge of such an expansion.\\

The core of this structure resides in decorated trees which are suitable combinatorial objects for encoding iterated integrals.
Indeed, decorations on the edges encode noises, convolution with kernels and derivatives on the previous objects, whereas decorations on the vertices correspond to monomials over $ d+1 $ variables for the time and the spatial components. Recently, decorated trees have been also used in the context of PDEs for writing a resonance-based scheme for dispersive equations in \cite{BS}. They allow to push an efficient numerical scheme from the first to the higher orders of approximation. It also covers a large class of equations. This reveals that more applications to PDEs could be expected in the future.  \\
 
Such a structure was sketched in \cite{reg} and formalised in \cite{BHZ}. Two Hopf algebras act on the solution expansion: the first one recenters iterated integrals around a base point, the second renormalises these integrals which can be ill-defined in many situations due to the singularity of the noises. The convergence of the renormalised iterated integrals is proved in \cite{CH16}. These two structures can be understood as a generalisation of the  Butcher-Connes-Kreimer Hopf algebra \cite{Butcher72,CK1,CK2} and the extraction-contraction Hopf algebra \cite{CHV10,CEM} which appear both in Numerical analysis and Quantum field theory.  We want to derive a  rigorous mathematical statement describing this analogy. The idea is to construct a map that will transport all the properties from the original structures to the ones used for singular SPDEs.

\begin{theorem} \label{deformation}
The structures defined in  \cite{BHZ}  and \cite{BS} are  deformations of the  Butcher-Connes-Kreimer and the extraction-contraction Hopf algebras.
\end{theorem}

For proving this theorem, one has to build on the specificity given by (S)PDEs. Indeed, one uses in the various applications Taylor expansions which are encoded at the algebraic level in \cite{BHZ}. The main idea is therefore to deform all classical structures by means of "algebraic" Taylor expansions. Derivatives and monomials appear as decorations on decorated trees and by changing them along with the transformation one gets a formal Taylor series with a leading term and terms of lower degrees. \\

Behind Theorem~\ref{deformation}, one can think about a larger programme. Indeed, many algebraic structures use trees and graphs and they have prescribed rules on how to carry some operations as extraction-contraction and cutting edges.
The main difficulty is to untangle all these operations in order to get the basic operations that will generate all the structure. \\

 Starting with the grafting product which is a pre-Lie product and applying an algebraic procedure, D. Guin and J. M. Oudom recovered in \cite{Guin1,Guin2} the Grossman-Larson product \cite{GL}, which is an associative product obtained from the Butcher-Connes-Kreimer coproduct by dualisation and suitable normalisation of the dual basis by symmetry factors. Their procedure works for any pre-Lie product and gives an alternative way of presenting the Hopf algebras at hand. A pre-Lie perspective on renormalisation has already been proposed in  rough paths (see \cite{BCFP}) which was the inspiration for the work \cite{BCCH}. For edge-decorated trees, one needs to use multi-pre-Lie algebras introduced in \cite{BCCH}. 
Multi-pre-Lie algebras are vector spaces (on some field $\mathbf k$ of characteristic zero) endowed with a family of pre-Lie products such that any finite linear combination of those remains a pre-Lie product. For decorated trees, it will be 
 a family of grafting products indexed by the edge decorations. In fact, one can see the whole family as a pre-Lie product by considering only planted trees. In the diagram below we provide a partial answer to theorem~\ref{deformation} and it reveals the approach followed in this paper
 \begin{equs}\label{diag_1}
 \begin{aligned}
\xymatrix{\curvearrowright
  \ar[rr]^{\scriptsize{\text{Guin-Oudom}}}\ar[d]_{\Theta} &&  \tilde\star_{0}  \ar[d]^{\Phi} \ar[rr]^{\scriptsize{\text{Dual}}}  && \Delta_{\scriptsize{\text{CK}}} \ar[d]_{\Psi} \\ \wh \curvearrowright
\ar[rr]^{\scriptsize{\text{Guin-Oudom}}} && \mathcal \star_0 \ar[rr]^{\scriptsize{\text{Dual}}} && \Delta_{\scriptsize{\text{DCK}}}
}
\end{aligned}
\end{equs}
The upper part of the diagram corresponds to the application of the Guin-Oudom procedure to the pre-Lie product $\curvearrowright$  (collection of grafting products). This step has been performed in \cite{F2018} and leads to the construction of a Butcher-Connes-Kreimer coproduct on decorated trees. The lower part of the diagram is the same procedure applied to 
 the deformation $ \wh{\curvearrowright} $ of $ \curvearrowright $ introduced in \cite{BCCH}. This is the content of Theorem~\ref{DCK_def} and leads to the construction of $ \Delta_{\scriptsize{\text{DCK}}} $. It is more involved due to the fact that one has to introduce infinite sums and use bigraded spaces as in \cite{BHZ}. Then, one wants a way to move from the upper part to the lower part. The maps $ \Phi $ and $ \Psi $ are not easy to define. Instead, we construct the isomorphism $ \Theta $ in Theorem~\ref{main} which can be extended naturally to decorated planted trees such that:
\begin{equs} \label{pre_theta}
\Theta \left(  \bar{\sigma} \curvearrowright \bar{\tau} \right) = \Theta(\bar{\sigma}) \, \widehat{\curvearrowright} \, \Theta(\bar{\tau})
\end{equs}
where $ \bar{\sigma}, \bar{\tau} $ are decorated planted trees. Then, the pre-Lie structure of $ \curvearrowright $ is transported to $ \widehat{\curvearrowright} $ via the isomorphism $ \Theta $. The construction of $ \Theta $ relies on the freeness of $ \curvearrowright $ and the fact that:
\begin{equs} \label{defor_prop}
\bar{\sigma} \, \widehat{\curvearrowright} \, \bar{\tau} =  \bar{\sigma} \curvearrowright \bar{\tau} + \text{lower grading terms}
\end{equs}
where one measures the grading of a decorated trees as the sum of its edge decorations. In this way, $ \widehat{\curvearrowright} $ is an algebraic deformation of $ \curvearrowright $. Therefore, the coproduct $ \Delta_{\text{DCK}} $ in the diagram~\ref{diag_1} can be understood as a deformation of the Butcher-Connes-Kreimer coproduct $ \Delta_{\text{CK}} $.
  The coproduct $ \Delta_{\text{DCK}} $ (to be precise, its flipped version obtained by exchanging both sides of the tensor product) has an important application in Numerical Analysis (see \cite{BS}).
  Yet, coproducts in \cite{BHZ} are different from $ \Delta_{\text{DCK}} $. This is due to derivations compatible with the pre-Lie product that needs to be added. In order to repeat the diagram~\ref{diag_1}, one has to work with another pre-Lie product defined by means of  plugging denoted by $ \rhd  $. We obtain a similar diagram when we construct two coproducts $ \Delta_{\text{P}} $ (resp. $ \Delta_{\text{DP}} $) from $ \rhd $ (resp. a deformation $ \wh \rhd $ of $ \rhd $):
  \begin{equs}\label{diag_2} \begin{aligned}
\xymatrix{ \rhd
  \ar[rr]^{\scriptsize{\text{Guin-Oudom}}} &&  \star   \ar[rr]^{\scriptsize{\text{Dual}}}  && \Delta_{\scriptsize{\text{P}}}  \\ \wh \rhd
\ar[rr]^{\scriptsize{\text{Guin-Oudom}}} && \mathcal \star_2 \ar[rr]^{\scriptsize{\text{Dual}}} && \Delta_{\scriptsize{\text{DP}}}
}
\end{aligned}
\end{equs}
Then, after performing an identification one can get from $\Delta_{\scriptsize{\text{DP}}}  $ the coproduct $ \Delta_2 $ which has been introduced in \cite{BHZ} for recentering iterated integrals. The lower part of the diagram is the subject of Theorem~\ref{DP_deformation}. One cannot find an isomorphism that will preserve the pre-Lie structure like in \eqref{pre_theta}, sending $ \rhd $ to $ \wh \rhd $ (see Proposition~\ref{not_commute}). One still has the deformation property \eqref{defor_prop} namely:
\begin{equs} \label{defor_prop_plugging}
\sigma \, \widehat{\rhd} \, \tau =  \sigma \rhd \tau + \text{lower grading terms}
\end{equs}
where $ \sigma, \tau $ are decorated trees. For proving that $ \widehat{\rhd} $ is a pre-Lie product one needs to use the Chu-Vandermonde identity, which we recall in Paragraph \ref{chuvdm}, which is also crucial for deriving the coassociativity in \cite[Prop. 3.11]{BHZ} for the coproduct denoted by $ \Delta_2 $ therein. The pre-Lie product given by $ \Theta $ is $ \widetilde{\rhd} $ satisfying
\begin{equs}
\Theta \left(  \sigma \rhd \tau \right) = \Theta(\sigma) \, \widetilde{\rhd}  \, \Theta(\tau).
\end{equs}
 Then, one is able to give an expression for $ \wh{\rhd} $ in Proposition~\ref{insertion_poly}:
\begin{equs} \label{def_plugging_p}
\sigma \, \wh{\rhd}_{v} \, \tau = \uparrow^{n_{\sigma}}_{v} \left( \Pi \sigma \, \widetilde{\rhd}_{v} \,  \tau \right)
\end{equs}
where $ \wh{\rhd}_{v} $ is the deformed plugging at the node $ v \in N_{\tau}  $ ($N_{\tau}$ nodes of $\tau$), $ \Pi $ removes the node decoration at the root of $ \sigma $ denoted by  $ n_{\sigma} $ and $ \uparrow^{n_{\sigma}}_{v} $ adds the decoration $ n_{\sigma} $ to the node $ v $. 
  We explain in Proposition~\ref{plugging_grafting} how the product $ \star_2 $ can be connected to $ \wh \curvearrowright $:
\begin{equs} \label{link_plugging_grafting_g}
\CI_b \left( \sigma \star_2  \tau \right) = \tilde{\uparrow}^{n_{\sigma}}_{N_{ \tau}} \left( \Pi \sigma \,  \wh \curvearrowright \, \CI_b ( \tau) \right)
\end{equs} 
 where $ \tilde{\uparrow}^{n_{\sigma}}_{N_{ \tau}} $ splits the decoration $ n_{\sigma} $ among the nodes of $ \tau $ and $ \CI_b(\tau) $ grafts the tree $ \tau $ onto a new root via an edge decorated by $ b $. The two formulae \eqref{def_plugging_p} and \eqref{link_plugging_grafting_g} give a different perspective with the use of derivations $ \uparrow^{n_{\sigma}}_{v}  $ and $ \tilde{\uparrow}^{n_{\sigma}}_{N_{ \tau}} $.

For the extraction-contraction coproduct $ \Delta_{\text{EC}} $, we use the Chapoton-Livernet insertion product $ \blacktriangleright $ of a tree into another. It is a pre-Lie product and can be directly defined via the product $ \star $. This definition seems new in comparison with the literature. Then, the deformation $ \wh{\blacktriangleright} $ is obtained by replacing $ \star $ by $ \star_2 $. One can run the Guin-Oudom procedure and get the following diagram:

\begin{equs}\label{diag_3}
\begin{aligned}
\xymatrix{\star\ar[rr]^{\scriptsize{\text{Def}}} &&   \blacktriangleright
  \ar[rr]^{\scriptsize{\text{Guin-Oudom}}} &&  \tilde\star_1   \ar[rr]^{\scriptsize{\text{Dual}}}  && \Delta_{\scriptsize{\text{EC}}}  \\
  \star_2
\ar[rr]^{\scriptsize{\text{Def}}} && \wh{\blacktriangleright}
\ar[rr]^{\scriptsize{\text{Guin-Oudom}}} && \mathcal \star_1 \ar[rr]^{\scriptsize{\text{Dual}}} && \Delta_{1}
}
\end{aligned}
\end{equs}
The fact that $ \star_1 $ is the dual of $ \Delta_1  $ is given in Theorem~\ref{deformed_E_C} wich concludes the proof of Theorem~\ref{deformation}.
We  are also able to write the cointeraction between $ \Delta_2 $ and $ \Delta_1 $ obtained in \cite{BHZ}  at the level of the deformed products, between $ \star_2 $ and $ \wh{\blacktriangleright} $ (see Theorem~\ref{theo_cointeraction}). This cointeraction has been previously observed in \cite{CHV10,CEM} for the non-deformed structures. The coproduct $ \Delta_1 $ is the same one as in \cite{BHZ} but for $ \Delta_2 $ the order of the factors in the tensor product is opposite compared with \cite{BHZ}.

 In the end, we have  proposed one single procedure allowing to recover the structure of \cite{BS} and \cite{BHZ} from the standard Butcher-Connes-Kreimer coproduct. Indeed, the isomorphism $ \Theta$ transports the grafting product used for constructing the Butcher-Connes-Kreimer coproduct to  $ \widehat{\curvearrowright}$ which allows us to recover \cite{BS}. Then the formulae \eqref{def_plugging_p} and \eqref{link_plugging_grafting_g} shows how the construction in \cite{BHZ} is connected to the isomorphism $ \Theta $ and $ \widehat{\curvearrowright} $. In this paper, we view the extraction-contraction coproduct as constructed from the pre-Lie insertion that depends on the  product dual of the Butcher-Connes-Kreimer coproduct. Therefore, the deformed extraction-contraction coproduct follows from this property. Let us mention that the formalism developed in this work has already some impact on the study of singular SPDEs. The work proposes \cite{BB21} a shorter proof of the renormalised equations first established in \cite{BCCH}. This proof is used in \cite{BB21b} for establishing the renormalised equation in a non-translation invariant setting. The spirit of our construction is also used in \cite{Felix} for constructing a structure group for quasi-linear equations. 
\\

Let us outline the paper by summarising the content of its sections. In Section~\ref{section::2}, we recall basics on multi-pre-Lie algebras and explore derivation for these structures. We see  how they could be encoded directly into the grafting product for decorated trees. Then, we present the Taylor deformation considered in this paper. It is constructed via the isomorphism $ \Theta $ in Theorem~\ref{main} which transports the pre-Lie structure of
$ \curvearrowright^a $ to $ \wh{\curvearrowright}^a $. \\

In Section~\ref{section::3}, we consider various products and their deformations. This is where we provide the main results for making Theorem~\ref{deformation} precise. We start by the product $  \curvearrowright $ on decorated  planted trees which groups the products $ \curvearrowright^a $ together. It is deformed via $ \Theta $ into $ \wh{\curvearrowright} $. By applying the Guin-Oudom procedure, we get an associative product $ \star_0 $.  In Theorem~\ref{DCK_def}, we identify the dual of $ \star_0 $ that is the deformed Butcher-Connes-Kreimer coproduct $ \Delta_{\text{\tiny{DCK}}} $. This explicit formula is very useful in Section~\ref{section::4}. Then, we move to another product $ \rhd $, the plugging of a decorated tree into another. We obtain a deformed product $ \wh \rhd $ not coming directly from $ \Theta $. The Guin-Oudom approach gives again an associative product $ \star_2 $ whose adjoint is $ \Delta_2 $ (see Theorem~\ref{DP_deformation}). The last product is the insertion product $\blacktriangleright$ whose deformation $\wh \blacktriangleright$ is constructed from $ \star_2 $. We then define an associative product $ \star_1 $ from $ \wh \blacktriangleright $ whose adjoint $ \Delta_1 $ is a deformed extraction-contraction coproduct (see Theorem~\ref{deformed_E_C}).
The section ends on studying a cointeraction property between deformed products. \\

In Section~\ref{section::4}, we consider  the two main applications of these algebraic structures which are a numerical scheme for dispersive PDEs and the theory of Regularity Structures for singular SPDEs. In subsection~\ref{subsection::4.1}, we focus on the numerical applications and show that the coproduct used for the local error analysis is actually very close to the one obtained in Theorem~\ref{DP_deformation}.
In subsection~\ref{subsection::4.2}, we show that the pre-Lie approach developped in Section~\ref{section::4} gives an immediate and elegant construction of the algebraic objects at play in the theory of Regularity Structures. One gets the recentering coproduct from Theorem~\ref{Deformation_NA} and the extraction-contraction coproduct from Theorem~\ref{Deformation_EC} (the coproduct which is used for implementing the BPHZ algorithm \cite{BP,Hepp,Zim})
Then, the cointeraction between these two maps is related to the pre-Lie cointeraction given in Section~\ref{section::4}.

\subsection*{Acknowledgements}
{\small
We wish to thank the anonymous referee for her/his extremely valuable remarks. First discussions on this work
were initiated while the authors participated in the workshop "Algebraic and geometric aspects of numerical methods for differential equations"
 held at the  Institut Mittag-Leffler in July 2018.
  The authors thank the organisers of this workshop for putting together a stimulating program  bringing different communities together, and the members of the institute for providing a friendly working atmosphere. 
}  

\section{Algebraic deformation of multi-pre-Lie algebras}

\label{section::2}

In this section, we present the notion of multi-pre-Lie algebra introduced in \cite{BCCH}. One of the main examples are decorated trees equipped with a family of grafting products. We consider derivations for this structure and define a Taylor deformation. The latter is the deformation used for constructing various products in  Section~\ref{section::3}.

\subsection{Multi-pre-Lie algebras}
\label{sect:def}
\begin{definition}
Let $E$ be any set. A multi-pre-Lie algebra indexed by $E$ is a vector space $\mathcal P$ over a field $\mathbf k$ of characteristic zero, endowed with a family $(\curvearrowright^a)_{a\in E}$ of bilinear products such that
\begin{equation}\label{mpl}
x\curvearrowright^a (y\curvearrowright^b z)-(x\curvearrowright^a y)\curvearrowright^b z=
y\curvearrowright^b (x\curvearrowright^a z)-(y\curvearrowright^b x)\curvearrowright^a z
\end{equation}
for any $a,b\in E$ and for any $x,y,z\in \mathcal P$.
\end{definition}
In particular, each product $\curvearrowright^a$ is left pre-Lie as well as any finite linear combination of those. The corresponding brace elements are recursively defined by
\begin{equs}
(x_1\cdots x_n) & \curvearrowright^{a_1\cdots a_n} z :=x_1\curvearrowright^{a_1}\Big((x_2\cdots x_n)\curvearrowright^{a_2\cdots a_n}z\Big)\\
& -\sum_{i=2}^n\Big(x_2\cdots x_{i-1}(x_1\curvearrowright^{a_1}x_i)x_{i+1}\cdots x_n
\Big)\curvearrowright^{a_2\cdots a_n}z
\end{equs}
for any $x_1,\ldots,x_n,z\in \mathcal P$. They are invariant under any permutation of the pairs $(x_1,a_n),\ldots,(x_n,a_n)$.
Let us give a few important examples:
\begin{example}\rm
Let $E$ and $V$ be two sets, let $T_E^V$ be the set of rooted trees with vertices decorated by $V$ and edges decorated by $E$, and let $\mathcal T_E^V$ be the linear $\mathbf k$-span of $T_E^V$. \label{rooted_r} The family of pre-Lie products is given by grafting by means of decorated edges, namely:
\begin{equation}
\sigma \curvearrowright^a \tau:=\sum_{v\in  N_{\tau} } \sigma \curvearrowright^a_v  \tau,
\end{equation}
\label{grafting_a}
where $\sigma $ and $\tau$ are two decorated rooted trees, $ N_\tau $ is the set of vertices of $ \tau $ and where $\sigma \curvearrowright^a_v \tau$ is obtained by grafting the tree $\sigma$ on the tree $\tau$ at vertex $v$ by means of a new edge decorated by $a\in E$. The proof of \eqref{mpl} is similar to the proof of the fact that grafting is a pre-Lie operation, and can be visualized by the following picture, symmetric with respect to the two branches:
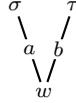
\begin{equs}
\begin{tikzpicture}[scale=0.19]
         \node at (-2,6)  (f) {}; 
           \node at (2,6)  (g) {}; 
        \node at (0,0) (c) {}; 
        
     \draw[kernel1] (c) --   (f) ;  
     \draw[kernel1] (c) --   (g) ; 
    \draw (g) node [fill=white,label={[label distance=0em]center: \scriptsize  $ \tau $}] {}  ;
    \draw (f) node [fill=white,label={[label distance=0em]center: \scriptsize  $ \sigma $}] {}  ;
    \draw (c) node [fill=white,label={[label distance=0em]center: \scriptsize  $ w $} ] {}  ;
    \node at (-1,3) [fill=white,label={[label distance=0em]center: \scriptsize  $ a $} ] () {};
    \node at (1,3) [fill=white,label={[label distance=0em]center: \scriptsize  $ b $} ] () {};
\end{tikzpicture}
\end{equs}
where $ \sigma $, $ \tau $ and $ w $ are decorated trees.
Below we present an example of grafting:
\begin{equs} \label{ex_grafting}
\bullet_\alpha\curvearrowright^{a} \begin{tikzpicture}[scale=0.2,baseline=0.1cm]
        \node at (0,0)  [dot,label= {[label distance=-0.2em]below: \scriptsize  $ \gamma $} ] (root) {};
         \node at (0,4)  [dot,label={[label distance=-0.2em]above: \scriptsize  $ \beta $}] (right) {};
            \draw[kernel1] (right) to
     node [sloped,below] {\small }     (root); 
     \node at (0,2) [fill=white,label={[label distance=0em]center: \scriptsize  $ b $} ] () {};
     \end{tikzpicture}   = \begin{tikzpicture}[scale=0.2,baseline=0.1cm]
        \node at (0,0)  [dot,label= {[label distance=-0.2em]below: \scriptsize  $ \gamma $} ] (root) {};
         \node at (2,4)  [dot,label={[label distance=-0.2em]above: \scriptsize  $ \beta $}] (right) {};
         \node at (-2,4)  [dot,label={[label distance=-0.2em]above: \scriptsize  $ \alpha $} ] (left) {};
            \draw[kernel1] (right) to
     node [sloped,below] {\small }     (root); \draw[kernel1] (left) to
     node [sloped,below] {\small }     (root);
     \node at (-1,2) [fill=white,label={[label distance=0em]center: \scriptsize  $ a $} ] () {};
    \node at (1,2) [fill=white,label={[label distance=0em]center: \scriptsize  $ b $} ] () {};
     \end{tikzpicture}
     + \begin{tikzpicture}[scale=0.2,baseline=0.1cm]
        \node at (0,0)  [dot,label= {[label distance=-0.2em]below: \scriptsize  $ \gamma $} ] (root) {};
         \node at (-2,3)  [dot,label={[label distance=-0.2em]left: \scriptsize  $ \beta $} ] (left) {};
          \node at (0,6)  [dot,label={[label distance=-0.2em]right: \scriptsize  $ \alpha $} ] (center) {};
          \draw[kernel1] (left) to
     node [sloped,below] {\small }     (root);
      \draw[kernel1] (center) to
     node [sloped,below] {\small }     (left);
     \node at (-1,1.5) [fill=white,label={[label distance=0em]center: \scriptsize  $ b $} ] () {};
      \node at (-1,4.5) [fill=white,label={[label distance=0em]center: \scriptsize  $ a $} ] () {};
     \end{tikzpicture}
\end{equs}
\end{example}
\begin{proposition}
For any set $E$, the $E$-multiple pre-Lie operad is given by the $\mathbf k$-span of the labellised rooted trees with $E$-decorated edges, where the labellisation is understood with respect to the vertices.
\end{proposition}
\begin{proof}
The proof (see \cite{F2018}) is a straightforward adaptation of the description of the pre-Lie operad by F. Chapoton and M. Livernet \cite{ChaLiv}.
\end{proof}
\begin{corollary}\label{free}
For any pair of sets $(E,V)$, the free $E$-multiple pre-Lie algebra generated by $V$ is given by the $\mathbf k$-span of rooted trees with $E$-decorated edges and $V$-decorated vertices.
\end{corollary}

\begin{example}\rm
We keep the notations of the previous example, but we suppose that a commutative monoid $\Omega$ acts on $V$. For any $a\in E$ and $\omega\in \Omega$, the binary product $\curvearrowright^{a,\omega}$ is defined as follows:
\begin{equation} \label{curve_insertion}
\sigma \curvearrowright^{a,\omega} \tau :=\sum_{v\in N_{\tau}}\uparrow_v^\omega(\sigma \curvearrowright^a_v \tau),
\end{equation}
where the operator $\uparrow_v^\omega$ stands for changing the decoration $\Labn_v$ of the vertex $v$ into $\omega\Labn_v$. As $\uparrow_v^\omega$ and $\sigma \curvearrowright_v^a-$ commute, we also have:
\begin{equation}
\sigma \curvearrowright^{a,\omega} \tau=\sum_{v\in N_{\tau}} \sigma \curvearrowright^a_v (\uparrow_v^\omega \tau).
\end{equation}
One can notice that the operators $\uparrow^\omega$ defined by
\begin{equation}
\uparrow^\omega \tau :=\sum_{v\in N_{\tau}}\uparrow_v^\omega \tau
\end{equation}
form a family of commuting derivations for all pre-Lie products $\curvearrowright^{a,\omega'}$. By derivation, we mean a linear map satisfying the Leibniz rule.
\end{example}

\begin{proposition}\label{e-omega-mpl}
The vector space $\mathcal T^V_E$ endowed with the binary products $\curvearrowright^{a,\omega}$ is a $E\times\Omega$-multiple pre-Lie algebra.
\end{proposition}
\begin{proof}
In the particular case when the three trees are made of a single vertex, Equation \eqref{mpl} comes from the fact that
\begin{equs}
\bullet_\alpha\curvearrowright^{a,\omega}(\bullet_\beta\curvearrowright^{b,\omega'}\bullet_\gamma)-
(\bullet_\alpha\curvearrowright^{a,\omega}\bullet_\beta)\curvearrowright^{b,\omega'}\bullet_\gamma = \begin{tikzpicture}[scale=0.2,baseline=0.1cm]
        \node at (0,0)  [dot,label= {[label distance=-0.2em]below: \scriptsize  $ \omega \omega'\gamma $} ] (root) {};
         \node at (2,6)  [dot,label={[label distance=-0.2em]above: \scriptsize  $ \beta $}] (right) {};
         \node at (-2,6)  [dot,label={[label distance=-0.2em]above: \scriptsize  $ \alpha $} ] (left) {};
            \draw[kernel1] (right) to
     node [sloped,below] {\small }     (root); \draw[kernel1] (left) to
     node [sloped,below] {\small }     (root);
     \node at (-1,3) [fill=white,label={[label distance=0em]center: \scriptsize  $ a $} ] () {};
    \node at (1,3) [fill=white,label={[label distance=0em]center: \scriptsize  $ b $} ] () {};
     \end{tikzpicture}
\end{equs}
is manifestly symmetric in $(a,\omega)$ and $(b,\omega')$, due to the commutativity of the monoid $\Omega$. The general case is straightforward and left to the reader.
\end{proof}

We abbreviate $\curvearrowright^{a,1}$ with $\curvearrowright^{a}$ where $1$ is the unit of $\Omega$.
\begin{proposition}\label{free-e-omega}
Suppose that $V=S \times\Omega$, and that the action of $\Omega$ on $V$ is given by $\omega.(s,\omega'):=(s,\omega\omega')$. Then $\Big(\mathcal T_E^V,(\curvearrowright^{a})_{a\in E},(\uparrow^\omega)_{\omega\in \Omega}\Big)$ is the free $S$-generated $E$-multiple pre-Lie algebra endowed with an $\Omega$-indexed family of commuting derivations.
\end{proposition}
\begin{proof}
Let $\Big(A,(\rhd^a)_{a\in E}\Big)$ be an $E$-multiple pre-Lie algebra endowed with an $\Omega$-indexed family of commuting derivations $(D^\omega)_{\omega\in\Omega}$. Fix a collection $(\alpha_s)_{s\in S}$ of elements of $A$. By Corollary \ref{free}, there is a unique morphism $\Phi_\alpha$ of $E$-multiple pre-Lie algebras from $\mathcal T_E^S$ to $A$ such that $\Phi_\alpha(\bullet_s)=\alpha_s$ for any $s\in S$. We want to extend $\Phi_\alpha$ to a morphism
$$\Psi_\alpha: \Big(\mathcal T_E^V,(\curvearrowright^{a})_{a\in E},(\uparrow^\omega)_{\omega\in \Omega}\Big)
\longrightarrow \Big(A,(\rhd^a)_{a\in E},(D^\omega)_{\omega\in \Omega}\Big).$$
From $\Psi_{\alpha}\circ\uparrow^\omega=D^\omega\circ \Psi_\alpha$
we immediately get:
$$\Psi_\alpha(\bullet_{s,\omega})=D^\omega(\alpha_s).$$
By Corollary \ref{free} again, there is a unique morphism $\chi_\alpha$ of $E$-multiple pre-Lie algebras from $\mathcal T_E^{V}$ to $A$ such that $\chi_\alpha(\bullet_{s,\omega})=D^\omega(\alpha_s)$ for any $\omega\in\Omega$ and $s\in S$. To prove that $\Psi_\alpha$ and $\chi_\alpha$ coincide, it only remains to prove $\chi_\alpha\circ\uparrow^\omega(\tau)=D^\omega\circ\chi_\alpha(\tau)$ for any $\omega\in\Omega$ and for any $\tau\in T_E^V$.\\

We proceed by induction on the number $|\tau|$ of vertices of $\tau$, the case $|\tau|=1$ being obvious. Any $\tau$ with at least two vertices can be written as
$$\tau=\sum_{j=1}^k \tau_1^j\curvearrowright^{a_j}\tau_2^j,$$
where the components $\tau_1^j$ and $\tau_2^j$ are strictly smaller with respect to the number of vertices. We can then compute:
\begin{eqnarray*}
\chi_\alpha\circ \uparrow^\omega(\tau)&=&\chi_\alpha\circ \uparrow^\omega\left(\sum_{j=1}^k\tau_1^j\curvearrowright^{a_j}\tau_2^j\right)\\
&=&\chi_\alpha\left(\sum_{j=1}^k\uparrow^\omega\tau_1^j\curvearrowright^{a_j}\tau_2^j+\tau_1^j\curvearrowright^{a_j}\uparrow^{\omega}\tau_2^j\right)\\
&=&\sum_{j=1}^k D^\omega\chi_\alpha(\tau_1^j)\rhd^{a_j}\chi_\alpha(\tau_2^j)+\chi_\alpha(\tau_1^j)\rhd^{a_j}D^\omega\chi_\alpha(\tau_2^j)\\
&=&D^\omega\left(\sum_{j=1}^k \chi_\alpha(\tau_1^j)\rhd^{a_j}\chi_\alpha(\tau_2^j)+\chi_\alpha(\tau_1^j)\rhd^{a_j}\chi_\alpha(\tau_2^j)\right)\\
&=&D^\omega\chi_\alpha(\tau).
\end{eqnarray*}
\end{proof}
\subsection{Taylor deformation  of free multi-pre-Lie algebras}
We suppose here that $\Omega$ is the monoid $\N^{d+1}$ endowed with componentwise addition.  A grading is given by 
\begin{equs}
|\mathbf n|_{\s}:=s_1n_1+\cdots +s_{d+1}n_{d+1}
\end{equs}
 where $\mathfrak s:=(s_1,\ldots, s_{d+1})\in \N_{>0}^{d+1}$ is fixed.   We suppose that $V = S \times \N^{d+1}$ and $E = S' \times \N^{d+1}$  where $ S $ and $ S' $ are two finite sets. Then $\Omega$ acts freely on both $E$ and $V$ in a graded way. We denote by $+$ the addition in $\Omega$ as well as both actions of $\Omega$ on $E$ and $V$. A family of deformed grafting products on $\mathcal T_E^V$ is defined as follows:
\begin{equation} \label{deformation_preLie}
\sigma \widehat{\curvearrowright}^a \tau:=\sum_{v\in N_{\tau}}\sum_{\ell\in\N^{d+1}}{\Labn_v \choose \ell} \sigma  \curvearrowright_v^{a-\ell}(\uparrow_v^{-\ell} \tau).
\end{equation}
\label{deformed_grafting_a}
Here $ \Labn_v \in\N^{d+1}$ denotes the second component of the decoration at the vertex $ v $. The generic term is self-explanatory if there exists a (unique) pair $(b,\alpha)\in E\times V$ such that $a=\ell+b$ and $\Labn_v =\ell+\alpha$. It vanishes by convention if this condition is not satisfied. We define the \textsl{grading} of a tree in $\mathcal T_E^V$  by the sum of the gradings of its edges given by $ |\cdot|_{\text{grad}} $:
\begin{equation} \label{grading}
|\tau|_{\text{grad}}:=\sum_{e\in E_{\tau}}\big|\Labe(e) \big|_{\s}
\end{equation}
where $ E_{\tau} $ are the edges of $ \tau $ and $ \Labe(e) $ is the decoration of the edge $ e $.
Then, $ \widehat{\curvearrowright}^a $ is a deformation of $ \curvearrowright^a $ in the sense that:
\begin{equs}
\sigma \widehat{\curvearrowright}^a \tau = \sigma \curvearrowright^a \tau +
\hbox{ lower grading terms}.
\end{equs}
We repeat the example in \eqref{ex_grafting} but now with the deformation:
\begin{equs} \label{ex_grafting_deformed_d}
\bullet_\alpha \, \wh \curvearrowright^{a} \begin{tikzpicture}[scale=0.2,baseline=0.1cm]
        \node at (0,0)  [dot,label= {[label distance=-0.2em]below: \scriptsize  $ \gamma $} ] (root) {};
         \node at (0,4)  [dot,label={[label distance=-0.2em]above: \scriptsize  $ \beta $}] (right) {};
            \draw[kernel1] (right) to
     node [sloped,below] {\small }     (root); 
     \node at (0,2) [fill=white,label={[label distance=0em]center: \scriptsize  $ b $} ] () {};
     \end{tikzpicture} &  = \begin{tikzpicture}[scale=0.2,baseline=0.1cm]
        \node at (0,0)  [dot,label= {[label distance=-0.2em]below: \scriptsize  $ \gamma $} ] (root) {};
         \node at (2,4)  [dot,label={[label distance=-0.2em]above: \scriptsize  $ \beta $}] (right) {};
         \node at (-2,4)  [dot,label={[label distance=-0.2em]above: \scriptsize  $ \alpha $} ] (left) {};
            \draw[kernel1] (right) to
     node [sloped,below] {\small }     (root); \draw[kernel1] (left) to
     node [sloped,below] {\small }     (root);
     \node at (-1,2) [fill=white,label={[label distance=0em]center: \scriptsize  $ a $} ] () {};
    \node at (1,2) [fill=white,label={[label distance=0em]center: \scriptsize  $ b $} ] () {};
     \end{tikzpicture}
     + \begin{tikzpicture}[scale=0.2,baseline=0.1cm]
        \node at (0,0)  [dot,label= {[label distance=-0.2em]below: \scriptsize  $ \gamma $} ] (root) {};
         \node at (-2,3)  [dot,label={[label distance=-0.2em]left: \scriptsize  $ \beta $} ] (left) {};
          \node at (0,6)  [dot,label={[label distance=-0.2em]right: \scriptsize  $ \alpha $} ] (center) {};
          \draw[kernel1] (left) to
     node [sloped,below] {\small }     (root);
      \draw[kernel1] (center) to
     node [sloped,below] {\small }     (left);
     \node at (-1,1.5) [fill=white,label={[label distance=0em]center: \scriptsize  $ b $} ] () {};
      \node at (-1,4.5) [fill=white,label={[label distance=0em]center: \scriptsize  $ a $} ] () {};
     \end{tikzpicture} \\ & + \sum_{\ell \neq 0} {\gamma \choose \ell}
     \begin{tikzpicture}[scale=0.2,baseline=0.1cm]
        \node at (0,0)  [dot,label= {[label distance=-0.2em]below: \scriptsize  $ \gamma -\ell $} ] (root) {};
         \node at (2,4)  [dot,label={[label distance=-0.2em]above: \scriptsize  $ \beta $}] (right) {};
         \node at (-2,4)  [dot,label={[label distance=-0.2em]above: \scriptsize  $ \alpha $} ] (left) {};
            \draw[kernel1] (right) to
     node [sloped,below] {\small }     (root); \draw[kernel1] (left) to
     node [sloped,below] {\small }     (root);
     \node at (-1,2) [fill=white,label={[label distance=-1em]left: \scriptsize  $ a-\ell $} ] () {};
    \node at (1,2) [fill=white,label={[label distance=0em]center: \scriptsize  $ b $} ] () {};
     \end{tikzpicture}
     +  \sum_{\ell \neq 0} {\beta \choose \ell}  \begin{tikzpicture}[scale=0.2,baseline=0.1cm]
        \node at (0,0)  [dot,label= {[label distance=-0.2em]below: \scriptsize  $ \gamma $} ] (root) {};
         \node at (-2,3)  [dot,label={[label distance=-0.2em]left: \scriptsize  $ \beta - \ell $} ] (left) {};
          \node at (0,6)  [dot,label={[label distance=-0.2em]right: \scriptsize  $ \alpha $} ] (center) {};
          \draw[kernel1] (left) to
     node [sloped,below] {\small }     (root);
      \draw[kernel1] (center) to
     node [sloped,below] {\small }     (left);
     \node at (-1,1.5) [fill=white,label={[label distance=0em]center: \scriptsize  $ b $} ] () {};
      \node at (-1,4.5) [fill=white,label={[label distance=0em]center: \scriptsize  $ a -\ell $} ] () {};
     \end{tikzpicture}
\end{equs}
The second line of \eqref{ex_grafting_deformed_d} contains the lower grading terms.

\begin{remark}
The deformation \eqref{deformation_preLie} has been first introduced for singular SPDEs in \cite{BCCH}. The leading term corresponds to the higher order of a Taylor jet. Indeed, if we consider $ \tau $ to be $ \bullet_k $, the tree composed of a single node decorated by $ k $, and $ \sigma $ to be $ \bullet $, we get:
\begin{equs} \label{sum_Taylor}
\bullet \, \widehat{\curvearrowright}^a \, \bullet_k = \sum_{\ell\in\N^{d+1}}{k \choose \ell} \bullet \curvearrowright^{a-\ell} \bullet_{k-\ell}
\end{equs}
We interpret the tree $ \bullet_k $ as a  monomial of degree $ k $ and the edge $ \curvearrowright^a $ is associated to the derivatives of some smooth function $ f : \R^{d+1} \rightarrow \R$.  Indeed, if we multiply a decorated tree by the product of the factors $ \frac{1}{\Labn_v !} $ where $ v $ is any vertex, one has
\begin{equs} \label{sum_Taylor_2}
\bullet \, \widehat{\curvearrowright}^a \, \frac{1}{k!}\bullet_k = \sum_{\ell\in\N^{d+1}} \frac{1}{\ell!} \bullet \curvearrowright^{a-\ell} \frac{1}{(k-\ell)!}\bullet_{k-\ell}.
\end{equs}
 We suppose that $ a \geq k $ and we can rewrite \eqref{sum_Taylor_2} as a Taylor jet of $ f^{(a-k)} $ when the bound is given by $ k $:
\begin{equs}
 \sum_{\ell \leq k}\frac{x ^{k-\ell}}{(k-\ell)!}f^{(a-\ell)}(0),
\end{equs}
where we have made the following identifications:
\begin{equs}
\frac{1}{(k-\ell)!}\bullet_{k-\ell} \equiv \frac{x^{k-\ell}}{(k-\ell)!}, \quad f^{(a-\ell)}(0) \equiv \frac{1}{\ell!} \bullet \curvearrowright^{a-\ell} .
\end{equs}

\end{remark}


\label{theta_iso}
\begin{theorem}\label{main}
There exists a unique linear isomorphism $\Theta:\mathcal T_E^V\to\mathcal T_E^V$ such that
$$\Theta(\bullet_\alpha)=\bullet_\alpha \, \hbox{ for any }\alpha\in V$$
and
\begin{equs} \label{deformation_curve}
\Theta(\sigma {\curvearrowright}^a \tau)=\Theta(\sigma)\,\widehat{\curvearrowright}^a \, \Theta(\tau)
\end{equs}
for any $\sigma,\tau\in \mathcal T_E^V$ and for any $a\in E$.
\end{theorem}
\begin{proof}
We introduce the set $PT_E^V$ of \textsl{planar} rooted trees with edges decorated by $E$ and vertices decorated by $V$, and $\mathcal {PT}_E^V$ will stand for the linear span of $PT_E^V$. A collection of \textsl{left graftings} is defined by
\begin{equation}
\sigma \ccurvearrowright^a \tau:=\sum_{v\in N_{\tau}} \sigma \ccurvearrowright^a_v  \tau
\end{equation}
for any $\sigma,\tau\in PT_E^V$, where $\sigma  \ccurvearrowright^a_v  \tau$ is obtained by grafting $\sigma$ on $\tau$ at vertex $v$ \textsl{on the left of the other branches}, by means of a new edge decorated by $a$.
\begin{lemma}
$\big(\mathcal{PT}_E^V,(\ccurvearrowright^a)_{a\in E}\big)$ is the free algebra endowed with a family of magmatic products indexed by $E$, generated by $V$.
\end{lemma}

\begin{proof}
The free magma endowed with a family of magmatic products indexed by $E$ generated by $V$ is obtained as the set of planar binary trees with leaves decorated by $V$ and internal vertices decorated by $E$. The magmatic product $\vee_a$ of two such trees $S$ and $T$, with $a\in E$, is obtained by plugging $S$ and $T$ on the left (resp. right) branch of the $Y$-shaped tree with a unique internal node decorated by $a$.\\

D. Knuth's rotation correspondence maps bijectively this set of planar binary trees onto $PT_E^V$. It transforms leaves into vertices and internal nodes into edges. The magmatic product $\vee_a$ is transformed into the \textsl{left Butcher product indexed by $a$} defined by
\begin{equation}
\sigma\lbutcher^a \tau:=\sigma\ccurvearrowright^a_{\varrho_{\tau}}  \tau
\end{equation}
where $\varrho_{\tau}$ is the root of $\tau$. By freeness property, there is a unique linear map $\Psi:\mathcal{PT}_E^V\to\mathcal{PT}_E^V$ such that $\Psi(\bullet_\alpha)=\bullet_\alpha$ for any $\alpha\in V$, and such that
\begin{equation}
\Psi(\sigma\lbutcher^a\tau)=\Psi(\sigma)\ccurvearrowright^a\Psi(\tau)
\end{equation}
for any $\sigma,\tau\in PT_E^V$. It is obvious that 
$$\sigma\ccurvearrowright^a\tau=\sigma\lbutcher^a\tau+\hbox{higher potential energy terms},$$
 where the potential energy of a tree is the sum of the heights of its vertices. From this we can infer, by induction on the number of vertices, that
$$\Psi(\sigma)=\sigma+ \hbox{ higher potential energy terms }$$
for any $\sigma\in PT_E^V$. This in turns implies that $\Psi$ is a linear isomorphism, the matrix of which is triangular with $1$'s on the diagonal in a suitable basis. Details in the particular case $E=\{*\}$ can be found in \cite{EM14} and in \cite{A14}, adaptation to any finite set $E$ is straightforward.
\end{proof}

Let us now finish up the proof of Theorem \ref{main}: by universal property of the free magmatic algebra, there exists a unique linear map 
$$\overline\Theta:\mathcal{PT}_E^V\to\mathcal{PT}_E^V$$
such that $\overline\Theta(\bullet_\alpha)=\bullet_\alpha$ for any $\alpha\in V$, and such that
\begin{equation}\label{thetabar}
\overline\Theta(\sigma{\ccurvearrowright}^a \tau)=\overline\Theta(\sigma)\wh{\ccurvearrowright}^a \overline\Theta(\tau)
\end{equation}
with
\begin{equation}\label{deform-plan}
\sigma\wh{\ccurvearrowright}^a \tau:=\sum_{v \in N_{\tau}}\sum_{\ell\in\N^{d+1}}{\Labn_v \choose \ell}\sigma \ccurvearrowright_v^{a-\ell}(\uuparrow_v^{-\ell} \tau)
\end{equation}
for any $a\in E$. Here $\uuparrow_v^{-\ell}$ is the obvious analogue of $\uparrow_v^{-\ell}$ in the planar setting. From \eqref{thetabar} and \eqref{deform-plan} one easily shows, by induction on the number of vertices, that
$$\overline\Theta(\sigma)=\sigma+\hbox{ lower grading terms}$$
for any $\sigma\in PT_E^V$, hence $\overline\Theta$ is a linear isomorphism. Moreover, one can also easily see, by induction on the number of vertices, that $\overline\Theta(\sigma)$ is a sum of planar trees all of which coincide with $\sigma$ modulo a change of decoration. It therefore gives rise to a linear isomorphism $\Theta: \mathcal T_E^V\to \mathcal T_E^V$ such that the following diagram commutes (where $\pi$ stands for the canonical projection which forgets the planarity):
\diagramme{
\xymatrix{
\mathcal {PT}_E^V\ar[rr]^{\overline\Theta}\ar@{>>}[d]_\pi && \mathcal {PT}_E^V\ar@{>>}[d]^\pi\\
\mathcal T_E^V\ar[rr]^\Theta && \mathcal T_E^V
}
}
\noindent The map $\Theta$ defined above fulfills the requirements of Theorem \ref{main}.
\end{proof}

\begin{corollary}
The algebra $\big(\mathcal {T}_E^V,(\wh\curvearrowright^a)_{a\in E}\big)$ is $E$-multiple pre-Lie.
\end{corollary}
\begin{proof}
This is immediate from the fact that $\big(\mathcal {T}_E^V,(\curvearrowright^a)_{a\in E}\big)$ is $E$-multiple pre-Lie, and from
$$ \sigma \wh\curvearrowright^a \tau=\Theta\big(\Theta^{-1}(\sigma)\curvearrowright^a\Theta^{-1}(\tau)\big).$$
\end{proof}
It is natural to introduce the associated $\Omega$-indexed family of modified derivations:
\begin{equation} \label{deformed_uparrow}
\wh\uparrow^\omega:=\Theta\uparrow^\omega\Theta^{-1},
\end{equation}
which obviously form a family of commuting derivations with respect to the new family of pre-Lie products $(\wh\curvearrowright^a)_{a\in E}$. As an obvious consequence of Proposition \ref{free-e-omega} and Theorem \ref{main},
\begin{corollary}
 The space $ \CT_E^V $ endowed with $ (\wh{\curvearrowright}^{a})_{a\in E} $ and $(\wh\uparrow^\omega)_{\omega\in \Omega}$ is the free $S$-generated $E$-multiple pre-Lie algebra endowed with an $\Omega$-indexed family of commuting derivations. The isomorphism with the original structure $\Big(\mathcal T_E^V,(\curvearrowright^{a})_{a\in E},(\uparrow^\omega)_{\omega\in \Omega}\Big)$ is given by $\Theta$.
\end{corollary}
\noindent Let us remark that $\wh\uparrow^\omega$ does not coincide with $\uparrow^\omega$: for example,
\begin{equs}
\uparrow^\omega\Theta  \begin{tikzpicture}[scale=0.2,baseline=0.1cm]
        \node at (0,0)  [dot,label= {[label distance=-0.2em]below: \scriptsize  $ \omega'' $} ] (root) {};
         \node at (0,5)  [dot,label={[label distance=-0.2em]above: \scriptsize  $ \omega' $}] (center) {};
            \draw[kernel1] (root) to
     node [sloped,below] {\small }     (center); 
    \node at (0,2.5) [fill=white,label={[label distance=0em]center: \scriptsize  $ a $} ] () {};
     \end{tikzpicture}  & = \sum_{\ell\in\Omega}{\omega''\choose\ell}  \begin{tikzpicture}[scale=0.2,baseline=0.1cm]
        \node at (0,0)  [dot,label= {[label distance=-0.2em]below: \scriptsize  $ \omega'' -\ell $} ] (root) {};
         \node at (0,5)  [dot,label={[label distance=-0.2em]above: \scriptsize  $ \omega + \omega' $}] (center) {};
            \draw[kernel1] (root) to
     node [sloped,below] {\small }     (center); 
    \node at (0,2.5) [fill=white,label={[label distance=0em]center: \scriptsize  $ a -\ell $} ] () {};
     \end{tikzpicture} +   \sum_{\ell\in\Omega}{\omega''\choose\ell}  \begin{tikzpicture}[scale=0.2,baseline=0.1cm]
        \node at (0,0)  [dot,label= {[label distance=-0.2em]below: \scriptsize  $ \omega + \omega'' -\ell $} ] (root) {};
         \node at (0,5)  [dot,label={[label distance=-0.2em]above: \scriptsize  $ \omega' $}] (center) {};
            \draw[kernel1] (root) to
     node [sloped,below] {\small }     (center); 
    \node at (0,2.5) [fill=white,label={[label distance=0em]center: \scriptsize  $ a -\ell $} ] () {};
     \end{tikzpicture}  \neq
     \\ \Theta \uparrow^\omega \begin{tikzpicture}[scale=0.2,baseline=0.1cm]
        \node at (0,0)  [dot,label= {[label distance=-0.2em]below: \scriptsize  $ \omega''  $} ] (root) {};
         \node at (0,5)  [dot,label={[label distance=-0.2em]above: \scriptsize  $ \omega' $}] (center) {};
            \draw[kernel1] (root) to
     node [sloped,below] {\small }     (center); 
    \node at (0,2.5) [fill=white,label={[label distance=0em]center: \scriptsize  $ a $} ] () {};
     \end{tikzpicture}  & = \sum_{\ell\in\Omega}{\omega''\choose\ell}  \begin{tikzpicture}[scale=0.2,baseline=0.1cm]
        \node at (0,0)  [dot,label= {[label distance=-0.2em]below: \scriptsize  $ \omega'' -\ell $} ] (root) {};
         \node at (0,5)  [dot,label={[label distance=-0.2em]above: \scriptsize  $ \omega + \omega' $}] (center) {};
            \draw[kernel1] (root) to
     node [sloped,below] {\small }     (center); 
    \node at (0,2.5) [fill=white,label={[label distance=0em]center: \scriptsize  $ a -\ell $} ] () {};
     \end{tikzpicture} +   \sum_{\ell\in\Omega}{\omega + \omega''\choose\ell}  \begin{tikzpicture}[scale=0.2,baseline=0.1cm]
        \node at (0,0)  [dot,label= {[label distance=-0.2em]below: \scriptsize  $ \omega + \omega'' -\ell $} ] (root) {};
         \node at (0,5)  [dot,label={[label distance=-0.2em]above: \scriptsize  $ \omega' $}] (center) {};
            \draw[kernel1] (root) to
     node [sloped,below] {\small }     (center); 
    \node at (0,2.5) [fill=white,label={[label distance=0em]center: \scriptsize  $ a -\ell $} ] () {};
     \end{tikzpicture}
\end{equs}

\noindent Finally, from Proposition \ref {e-omega-mpl}, we get that the vector space $\mathcal T^V_E$ endowed with the binary products $\wh\curvearrowright^{a,\omega}$ is a $E\times\Omega$-multiple pre-Lie algebra, with
\begin{equation}
\sigma \wh\curvearrowright^{a,\omega} \tau=\Theta\big(\Theta^{-1}(\sigma)\curvearrowright^{a,\omega}\Theta^{-1}(\tau)\big).
\end{equation}

\section{Algebraic deformations of various products}
\label{section::3}

We consider three products on decorated trees: grafting, plugging and insertion. The grafting product is deformed by using the isomorphism $\Theta$ and we construct an associative product via the Guin-Oudom procedure associated to it. Then, the coproduct dual of this product is identified as the deformed Butcher-Connes-Kreimer coproduct on decorated trees. For the plugging, the deformation is not exactly given by $ \Theta $ but follows the same spirit. The associated product obtained by applying Guin-Oudom is the dual of  a Butcher-Connes-Kreimer coproduct where derivations have been added in the definition.
Then, we also define the deformed insertion from this deformed product.
The last part of this section focuses on the cointeraction between deformed products.
\subsection{Reminder on the Chu-Vandermonde identity}\label{chuvdm}
 Let $S$ be any finite set, and let $k,\ell:S\to \N$ with $\N:=\{0,1,2,\ldots\}$. Factorials and binomial coefficients are defined by:
\begin{equation}
\ell!:=\prod_{x\in S}\ell(x)!,\hskip 12mm {k\choose \ell}:=\prod_{x\in S}{k(x)\choose\ell(x)},
\end{equation}
where $\displaystyle {a\choose b}:=\frac{a!}{b!(a-b)!}$ if $a,b\in\N$ and $a\ge b$, with the convention $\displaystyle {a\choose b}=0$ if $a,b\in\N$ and $a<b$. now let $\wt S$ be another finite set, and let $\pi:S\to \wt S$ be a map (not necessarily surjective nor injective). For any map $\ell:S\to\N$ we define $\pi_*\ell:\wt S\to \N$ by:
\begin{equation}
\pi_*\ell(x)=
\begin{cases}
0 \hbox{ if } x \hbox{ is not in the image of }\pi,\\
\displaystyle \sum_{y\in S,\,\pi(y)=x}\ell(y)  \hbox{ if }x \hbox{ is in the image of }\pi.
\end{cases}
\end{equation}
\begin{proposition}(Chu-Vandermonde identity)
For any $\wt\ell:\wt S\to\N$ and for any $k:S\to \N$ the following holds:
\begin{equation}\label{chu-vdm}
{\pi_*k\choose \wt \ell}=\sum_{\ell,\,\pi_*\ell=\wt\ell}{k\choose\ell}.
\end{equation}
\end{proposition}
\begin{proof}
Consider the obvious equality between polynomials with variables indexed by $\wt S$:
\begin{equation}
\prod_{s\in S}\left(1+X_{\pi(s)}\right)^{k(s)}=\prod_{\wt s\in\wt S}\left(1+X_{\wt s}\right)^{\pi_*k(\wt s)},
\end{equation}
and pick the coefficient of the monomial $\displaystyle \prod_{\wt s\in\wt S}X_{\wt s}^{\wt\ell(\wt s)}$ on both sides.
\end{proof}
\subsection{Deformation of grafting product}
We denote by $ P_E^V $ the set of planted trees and by $ \mathcal{P}_E^V $ their linear span. \label{planted_p} A planted tree is of the form $ \mathcal{I}_{a}(\tau) $ where $ \tau \in \mathcal{T}_E^V $ and $ \mathcal{I}_{a}(\tau) $ denotes the grafting of the tree $ \tau $ onto a new root via an edge decorated by $ a \in E$. Remark that, according to our conventions, the root of a planted tree is not decorated. We define two pre-Lie products on the planted trees by:
\begin{equs} \label{def_pre_a}
\mathcal{I}_{a}(\sigma) \curvearrowright \mathcal{I}_{b}( \tau) =   \mathcal{I}_{b}( \sigma \curvearrowright^{a}  \tau), \quad \mathcal{I}_{a}(\sigma) \, \wh\curvearrowright \, \mathcal{I}_{b}( \tau) =   \mathcal{I}_{b}( \sigma \wh\curvearrowright^{a}   \tau).
\end{equs}
\label{grafting_planted_trees} \label{deformed_grafting_planted_trees} We have the following property
\begin{proposition}
The space $\mathcal{P}_E^V$ endowed with $ \curvearrowright  $ or $\wh \curvearrowright  $ is a pre-Lie algebra.
\end{proposition}

\begin{proof}
This comes from the obvious identification of $\mathcal{P}_E^V$ with $\mathcal{T}_E^V\otimes \mathbf kE$ given by
$$\mathcal{I}_{a}(\tau)\longmapsto \tau\otimes a.$$
It is easily seen that the structure $ \curvearrowright^{a}  $ resp. $(\wh\curvearrowright^{a})_{a \in E}$ on $\mathcal{T}_E$ is $E$-multiple pre-Lie if and only if the bilinear product $ \curvearrowright$ resp. $\wh \curvearrowright $ on $\mathcal{T}_E^V\otimes \mathbf kE$ given by
$$(\sigma \otimes a)\curvearrowright (\tau\otimes b):=(\sigma\curvearrowright^a \tau)\otimes b
\quad \hbox{resp. }(\sigma \otimes a) \, \wh\curvearrowright \, (\tau\otimes b):=(\sigma \wh\curvearrowright^a \tau)\otimes b$$
is pre-Lie \cite{F2018}.
\end{proof}

Below we present an example of computation with $ \curvearrowright  $:
\begin{equs} \label{grafting_ident}
  \begin{tikzpicture}[scale=0.2,baseline=0.1cm]
        \node at (0,0)  [dot,label= {[label distance=-0.2em]below: \scriptsize  $  $} ] (root) {};
         \node at (0,4)  [dot,label={[label distance=-0.2em]above: \scriptsize  $ \alpha $}] (right) {};
            \draw[kernel1] (right) to
     node [sloped,below] {\small }     (root); 
     \node at (0,2) [fill=white,label={[label distance=0em]center: \scriptsize  $ a $} ] () {};
     \end{tikzpicture}   \curvearrowright \begin{tikzpicture}[scale=0.2,baseline=0.1cm]
        \node at (0,0)  [dot,label= {[label distance=-0.2em]below: \scriptsize  $  $} ] (root) {};
         \node at (0,4)  [dot,label={[label distance=-0.2em]above: \scriptsize  $ \beta $}] (right) {};
            \draw[kernel1] (right) to
     node [sloped,below] {\small }     (root); 
     \node at (0,2) [fill=white,label={[label distance=0em]center: \scriptsize  $ b $} ] () {};
     \end{tikzpicture}   =  \begin{tikzpicture}[scale=0.2,baseline=0.1cm]
        \node at (0,0)  [dot,label= {[label distance=-0.2em]below: \scriptsize  $ $} ] (root) {};
         \node at (-2,3)  [dot,label={[label distance=-0.2em]left: \scriptsize  $ \beta $} ] (left) {};
          \node at (0,6)  [dot,label={[label distance=-0.2em]right: \scriptsize  $ \alpha $} ] (center) {};
          \draw[kernel1] (left) to
     node [sloped,below] {\small }     (root);
      \draw[kernel1] (center) to
     node [sloped,below] {\small }     (left);
     \node at (-1,1.5) [fill=white,label={[label distance=0em]center: \scriptsize  $ b $} ] () {};
      \node at (-1,4.5) [fill=white,label={[label distance=0em]center: \scriptsize  $ a $} ] () {};
     \end{tikzpicture}
\end{equs}
The main difference with \eqref{ex_grafting} is that we do not have decorations at the root and one cannot graft at the root: we get one term less.
If we extend $ \Theta $ defined in Theroem~\eqref{main} to $ \CP_E^V $ by:
\begin{equs} \label{extension_theta}
\Theta(\CI_a(\tau)) := \CI_a(\Theta(\tau)),
\end{equs}
we obtain that $\wh \curvearrowright $ is a deformation of $  \curvearrowright $. For every $ \tau, \sigma \in \CT_E^V $, one has
\begin{equs} \label{deformation_curve}
\Theta \left( \mathcal{I}_{a}(\sigma)  \curvearrowright \mathcal{I}_{b}( \tau) \right) = \Theta \left( \mathcal{I}_{a}(\sigma) \right) \, \wh \curvearrowright \, \Theta \left( \mathcal{I}_{b}( \tau) \right).
\end{equs}
Indeed, 
\begin{equs}
\Theta \Big( \mathcal{I}_{a}(\sigma)  \curvearrowright \mathcal{I}_{b}( \tau) \Big) 
 & = \Theta \Big( \CI_b( \sigma \curvearrowright^a  \tau )\Big)
 \\  & = \CI_b \Big( \Theta \left( \sigma \curvearrowright^a  \tau \right)\Big)
\\ & = \mathcal{I}_{b}\Big( \Theta \left( \sigma \right) \, \wh \curvearrowright^a \, \Theta \left( \tau \right) \Big)
\\ & = \CI_{a}\Big(\Theta \left( \sigma \right)\Big) \, \wh \curvearrowright \,\mathcal{I}_{b}\Big( \Theta \left(  \tau \right)\Big)
\\ & = \Theta \Big( \mathcal{I}_{a}(\sigma) \Big) \, \wh \curvearrowright \, \Theta \Big( \mathcal{I}_{b}( \tau) \Big) 
\end{equs}
where we have used \eqref{deformation_curve}, \eqref{def_pre_a} and \eqref{extension_theta}.
The deformation of \eqref{grafting_ident} is given by
\begin{equs} \label{ex_grafting_deformed}
\begin{tikzpicture}[scale=0.2,baseline=0.1cm]
        \node at (0,0)  [dot,label= {[label distance=-0.2em]below: \scriptsize  $  $} ] (root) {};
         \node at (0,4)  [dot,label={[label distance=-0.2em]above: \scriptsize  $ \alpha $}] (right) {};
            \draw[kernel1] (right) to
     node [sloped,below] {\small }     (root); 
     \node at (0,2) [fill=white,label={[label distance=0em]center: \scriptsize  $ a $} ] () {};
     \end{tikzpicture} \, \wh \curvearrowright \begin{tikzpicture}[scale=0.2,baseline=0.1cm]
        \node at (0,0)  [dot,label= {[label distance=-0.2em]below: \scriptsize  $  $} ] (root) {};
         \node at (0,4)  [dot,label={[label distance=-0.2em]above: \scriptsize  $ \beta $}] (right) {};
            \draw[kernel1] (right) to
     node [sloped,below] {\small }     (root); 
     \node at (0,2) [fill=white,label={[label distance=0em]center: \scriptsize  $ b $} ] () {};
     \end{tikzpicture}   = 
      \begin{tikzpicture}[scale=0.2,baseline=0.1cm]
        \node at (0,0)  [dot,label= {[label distance=-0.2em]below: \scriptsize  $ $} ] (root) {};
         \node at (-2,3)  [dot,label={[label distance=-0.2em]left: \scriptsize  $ \beta $} ] (left) {};
          \node at (0,6)  [dot,label={[label distance=-0.2em]right: \scriptsize  $ \alpha $} ] (center) {};
          \draw[kernel1] (left) to
     node [sloped,below] {\small }     (root);
      \draw[kernel1] (center) to
     node [sloped,below] {\small }     (left);
     \node at (-1,1.5) [fill=white,label={[label distance=0em]center: \scriptsize  $ b $} ] () {};
      \node at (-1,4.5) [fill=white,label={[label distance=0em]center: \scriptsize  $ a $} ] () {};
     \end{tikzpicture} 
     +  \sum_{0 < |\ell|_s \leq |\beta|_s \wedge |a|_s} {\beta \choose \ell}  \begin{tikzpicture}[scale=0.2,baseline=0.1cm]
        \node at (0,0)  [dot,label= {[label distance=-0.2em]below: \scriptsize  $ $} ] (root) {};
         \node at (-2,3)  [dot,label={[label distance=-0.2em]left: \scriptsize  $ \beta - \ell $} ] (left) {};
          \node at (0,6)  [dot,label={[label distance=-0.2em]right: \scriptsize  $ \alpha $} ] (center) {};
          \draw[kernel1] (left) to
     node [sloped,below] {\small }     (root);
      \draw[kernel1] (center) to
     node [sloped,below] {\small }     (left);
     \node at (-1,1.5) [fill=white,label={[label distance=0em]center: \scriptsize  $ b $} ] () {};
      \node at (-1,4.5) [fill=white,label={[label distance=0em]center: \scriptsize  $ a -\ell $} ] () {};
     \end{tikzpicture}
\end{equs}

We want to construct a Hopf algebra from the pre-Lie algebra $ (\mathcal{P}_E^V,\wh \curvearrowright)$. We follow the path initiated by Guin and Oudom in \cite{Guin1, Guin2}. It has also been used in \cite{F2018} for the pre-Lie algebra $ (\mathcal{P}_E^V, \curvearrowright)$ for constructing a Butcher-Connes-Kreimer type coproduct. We want to see how this perspective is robust toward the deformation of the pre-Lie product. We denote by $ \mathfrak{g} $ the Lie algebra $ (\mathcal{P}_E^V, \left[ \cdot \right]) $ where the bracket is given for $ \bar \sigma, \bar \tau \in \CT_E^V $ by:
\begin{equs}
\left[ \bar \sigma,  \bar \tau  \right] = \bar \sigma \, \wh \curvearrowright \,  \bar\tau -  \bar \tau \, \wh \curvearrowright \,  \bar \sigma.
\end{equs} 

We first consider the symmetric algebra $S( \mathcal{P}_E^V )$ \label{forest_planted_p} with its usual multiplicative coproduct $\Delta$ defined for every $ \bar \tau \in \mathcal{P}_E^V $ by:
\begin{equs} \label{delta_p}
\Delta \bar \tau &=\bar \tau \otimes \one + \one \otimes \bar \tau. 
\end{equs}
A natural basis of this symmetric algebra is given by trees with undecorated root. We denote by $X^k$ the decorated one-vertex tree $\bullet_k$. We extend the map $\Delta$ multiplicatively to trees with decorated root by 
\begin{equs}\label{delta_p_trees}
\Delta X^k = X^k \otimes \one +  \one \otimes X^k.
\end{equs}
Multiplicativity is not intended for the tree product but for the natural action of $ S(\CP_E^V) $  on $ \CT_E^V $ obtained by identifying the roots of the planted trees with the one of the tree in $ \CT_E^V $. We identify the forest $ X^k \widetilde{\prod}_{i \in I} \CI_{a_i}(\sigma_i) $ with the decorated tree $ X^k \prod_{i \in I} \CI_{a_i}(\sigma_i)  $. Using Sweedler's notation it reads for $ w \in S(\mathcal{P}_E^V) $: $
\Delta w =\sum_{(w)} w^{(1)}\otimes w^{(2)} $.\\

We are now in position to apply the Guin-Oudom construction to the pre-Lie algebra $(\mathcal{P}_E^V,\, \wh \curvearrowright )$. One can define a product on $ S( \mathcal{P}_E^V ) $ for every  $u,v, w \in S(\mathcal{P}_E^V)$, $ x, y  \in \mathcal{P}_E^V $ by
\begin{equs}\label{e::symmetric product}
\begin{aligned}
\one \bullet w&= w, \quad 
u\bullet \one =\one^{\star}(u), \quad  x \bullet y = x  \, \wh \curvearrowright \, y, \\
w \bullet uv &=\sum_{(w)} ( w^{(1)} \bullet u)(w^{(2)} \bullet v),\\
 x v \bullet y  &=  x \, \wh \curvearrowright \,   ( v\bullet y)  - (x \, \wh \curvearrowright \, v) \bullet y
 \end{aligned}
\end{equs}
 where $\one^\star$ stands for the counit, and where $ \wh \curvearrowright $ is extended to $   \mathcal{P}_E^V \otimes S( \mathcal{P}_E^V ) $ in the following way:
\begin{align*}
x \, \wh \curvearrowright \, x_1\ldots x_k 
&=\sum_{i=1}^k x_1\ldots (x \, \wh \curvearrowright \, x_i)\ldots x_k,
\end{align*}
with $ x_i \in  \mathcal{P}_E^V $. \label{star_0}
Then, we define the associative product $ \star_0 $ as:
\begin{equs} \label{def_*}
w \star_0 v = \sum_{(w)}\left(  (w^{(1)} \bullet v) w^{(2)} \right).
\end{equs}
One has the following characterisation from \cite[Theorem 2.12]{Guin2}:

\begin{proposition}
The space $ (S(\mathcal{P}_E^V ), \star_0, \Delta) $ is a Hopf algebra isomorphic to the enveloping algebra $ \CU(\mathfrak{g}) $ where the Lie algebra $ \mathfrak{g} $  is $ \mathcal{P}_E^V $ 
endowed with the
antisymmetrization of the pre-Lie product $\wh \curvearrowright$.
\end{proposition}

In the sequel, we want to identify the dual of the product $ \star_0 $. We suppose that the set $V$ of vertex decorations coincides with the commutative monoid $\Omega$. It  will be given by the following maps $ \Delta_{\text{\tiny{DCK}}} : S(\mathcal{P}_E^V ) \rightarrow S(\mathcal{P}_E^V ) \hattimes S(\mathcal{P}_E^V )$ and $ \bar{\Delta}_{\text{\tiny{DCK}}} : \mathcal{T}_E^V  \rightarrow S(\mathcal{P}_E^V ) \hattimes \mathcal{T}_{E}^V   $ \label{delta_dck} \label{delta_dck_bar} defined recursively by:
\begin{equs} \label{recur_def_deltaDCK}
\begin{aligned}
\Delta_{\text{\tiny{DCK}}}  \CI_a(\tau) & = 
\left( \id  \hattimes \CI_a \right) \bar{\Delta}_{\text{\tiny{DCK}}} \tau + \CI_a(\tau) \hattimes \one,
\quad \bar{\Delta}_{\text{\tiny{DCK}}} X^k  = \one \hattimes X^{k} 
\\ \bar{\Delta}_{\text{\tiny{DCK}}} \CI_a(\tau) & = \left( \id \hattimes \CI_a \right)  \bar{\Delta}_{\text{\tiny{DCK}}} \tau + \sum_{\ell \in \N^{d+1}} \frac{1}{\ell !} \CI_{a + \ell}(\tau) \hattimes X^{\ell}.
\end{aligned}
\end{equs}
The map $ \Delta_{\text{\tiny{DCK}}}  $ is extended using the product of $ S(\mathcal{T}_E^V) $ (denoted by $ \widetilde{\prod} $). We use the tree product (denoted by $ \prod $) for extending the map $ \bar{\Delta}_{\text{\tiny{DCK}}} $.
We also define the map $ \bar{\Delta}_{\text{\tiny{DCK}}}^{\root} $ by
\begin{equs}
\bar{\Delta}_{\text{\tiny{DCK}}}^{\root}  \CI_a(\tau) = \sum_{\ell \in \N^{d+1}} \frac{1}{\ell !} \CI_{a + \ell}(\tau) \hattimes X^{\ell}, \quad \bar{\Delta}_{\text{\tiny{DCK}}}^{\root}  X^k = \one \hattimes X^k
\end{equs}
and extend it multiplicatively with respect to the tree product. We have replaced the standard tensor product $ \otimes $ by $ \hattimes $ for making sense of the infinite sum over $ \ell $. Indeed, the maps $ \Delta_{\text{\tiny{DCK}}}  $ and $ \bar{\Delta}_{\text{\tiny{DCK}}} $ are triangular maps for the bigrading  given in \cite{BHZ} by
\[ 
	(|T_{\Labe}^{\Labn}|_{\text{bi}}) = ( |T_{\Labe}^{\Labn}|_{\text{grad}}, |N_T \setminus \lbrace \varrho_T \rbrace| + |E_T| ), 
\] 
where $ T $ is a rooted tree, $ \Labe : E_T \rightarrow E $ are the edge decorations and $ \Labn : N_T \rightarrow V  $ are the node decorations, $ T_{\Labe}^{\Labn} $ is the decorated tree with such decorations. We have used the grading $ |\cdot|_{\text{grad}} $ defined in \eqref{grading}.

We recall some basics on bigraded spaces extracted from \cite[Sec. 2.3]{BHZ}. A bigraded space $ V $ is given by
\begin{equs}
V = \bplus_{n \in \N^2} V_n\;,
\end{equs}
where each $ V_n $ is a vector space and any element $v$ of $ V $ is given as a formal sum $ v = \sum_{n \in \N^2} v_n$ with
$v_n \in V_n$ and such that there exists $k \in \N$, $v_n= 0$ when
$n_2 > k$.
For two bigraded spaces $V$ and $W$, we define $V \hattimes W$ as the bigraded space
\begin{equ}\label{newtensor}
V \hattimes W \eqdef \bplus_{n \in \N^2} \left[ \,
\bigoplus_{m+\ell = n}(V_m \otimes W_\ell)\right]. 
\end{equ}
We consider the following partial order on $\N^2$
\begin{equ}
(m_1, m_2) \ge (n_1, n_2)  \qquad\Leftrightarrow\qquad m_1 \ge n_1 \;\&\; m_2 \le n_2\;.
\end{equ}
 A family $\{A_{m,n}\}_{m,n \in \N^2}$ of linear maps  $A_{m,n}:V_n\to\bar V_m$  between two bigraded spaces $ V $ and $ \bar V $  is called  \textit{triangular} if $A_{m,n} = 0$ unless $m\ge n$.
 
 Following \cite[Remark 2.16]{BHZ}, one has 
 a canonical inclusion $V^* \otimes W^* \subset (V \hattimes W)^*$ and the dual $V^*$ consists of formal sums $\sum_n v^*_n$ with $v_n^* \in V_n^*$,
and  for every $k \in \N$ there exists $f(k)$ such that $v^*_n = 0$
for every $n \in \N^2$ with $n_1 \ge f(n_2)$.

From this definition of the dual, we clearly see from \eqref{deformation_preLie} that $   \wh \curvearrowright  $ is a map defined from $ {\mathcal{P}_E^V}^{*}  \hattimes {\mathcal{P}_E^V}^{* } $ into ${\mathcal{P}_E^V }^{*}$. 
Indeed, the product $ \wh  \curvearrowright $ gives a finite sum which is clearly an element of the dual. 
We can derive a recursive formulation of $  \wh  \curvearrowright $ in order to match the one for $ \Delta_{\text{\tiny{DCK}}} $. We also associate a grafting product to $ \bar{\Delta}_{\text{\tiny{DCK}}} $ denoted by $ \bar \curvearrowright  $. This product plays a central role in the recursive expression of $  \wh  \curvearrowright  $: 
\begin{equs} \label{def_deltabar}
\begin{aligned}
 \sigma \, \wh\curvearrowright \, \mathcal{I}_{b}( \tau) & = \mathcal{I}_{b}( \sigma \, \bar \curvearrowright \,  \tau), 
\\  \mathcal{I}_{a}(\tau) \, \bar \curvearrowright \, X^k  \tau_0 & = \sum_{\ell\in\N^{d+1}}{k\choose \ell}
X^{k-\ell} \mathcal{I}_{a-\ell}(\tau) \tau_0 + X^k \left( \mathcal{I}_{a}(\tau) \, \wh \curvearrowright \,   \tau_0  \right),
\\ 
\widetilde{\prod}_{i \in I} \mathcal{I}_{a_i}(\sigma_i)  \, \bar \curvearrowright \, \tau & = \CI_{a_j}(\sigma_j) \, \bar \curvearrowright \, \left(  \widetilde{\prod}_{i \neq j} \mathcal{I}_{a_i}(\sigma_i) \, \bar \curvearrowright \, \tau \right) 
\\ &  -\sum_{ j \in I \setminus \lbrace n \rbrace}\Big( \left( \CI_{a_j}(\sigma_j) \, \wh\curvearrowright \, \CI_{a_n}(\sigma_n) \right) \widetilde{\prod}_{\ell \notin \lbrace j,n \rbrace} \CI_{a_{\ell}}(\sigma_{\ell})
\Big) \, \bar \curvearrowright \, \tau
\end{aligned}
\end{equs}
\label{recursive_deformed_grafting}
where $ n \in J $, $ \sigma \in S(\mathcal{P}_E^V),  \tau \in \CT_E^V $.
  Here $ \tau_0   $ is a tree with root decoration zero, and can therefore be identified as an element of $ S(\CP_E^V) $.
The identities \eqref{def_deltabar} separate the grafting at the root from the grafting on the other nodes.
For the dual, we consider an inner product on the decorated trees  $  \langle \cdot , \cdot \rangle  $ such that for every $ \sigma, \tau \in \CT_E^V $:
\begin{align*} \label{inner_product}
\langle \sigma ,  \tau \rangle = \delta_{\sigma, \tau} S(\tau)
\end{align*}
where the algebraic symmetry factor $ S $ is defined inductively as follows: for any tree $\tau$ written as 
$
\tau = X^{k} \ \prod_{j=1}^m \mathcal{I}_{a_j}(\tau_j)^{\beta_j} \;, $
where we group terms (uniquely)
in such a way that $(a_i,\tau_i) \neq (a_j,\tau_j)$ for $i \neq j$, we inductively set
\begin{align}
S(\tau)
=
k!
\Big(
\prod_{j=1}^{m}
S(\tau_{j})^{\beta_{j}}
\beta_{j}!
\Big)\;.
\end{align}
The presence of $ k! $ can be seen as an internal hidden symmetry at each vertex, as if each vertex, seen under the microscope, behaves as a list of $ d + 1 $ packets of smaller points.
We adopt the same definition on $ S(\CP_E^V) $ where the product $ \prod $ is replaced by $ \tilde{\prod} $.
We extend naturally the inner product \eqref{inner_product} to the space $ \mathcal{T}_E^V \otimes \mathcal{T}_E^V  $ by setting for $ \sigma_1, \sigma_2,  \tau_1,  \tau_2 \in \CT_E^V $
\begin{align*}
\langle  \sigma_{1} \otimes \sigma_{2},  \tau_{1} \otimes  \tau_{2} \rangle = 
\langle  \sigma_{1} ,  \tau_{1}  \rangle \langle  \sigma_{2},  \tau_{2} \rangle.
\end{align*}

In the proof below and also in the rest of the paper, we will use multinomial coefficients for $ k, \ell_1,...\ell_n \in \N^{d+1} $ :
\begin{equs}
{k \choose \ell_1,...,\ell_n} = {k \choose (\ell_i)_{i \in \lbrace 1,...,n\rbrace}} = \frac{k !} { (k - \sum_i \ell_i)!  \prod_i (\ell_i!) }.
 \end{equs}
 Let $ I $ be a finite set and let $ A $ be a commutative algebra. We consider an element $ \tau = \otimes_{i \in I} \tau_i $ in the tensor algebra  $\otimes_{i \in I} A  $. For every  partition $ \pi $ of $ I $, 
we define the product $ \mathcal{M}^{\pi} $ by:
\begin{equs}
\mathcal{M}^{\pi} \tau = \bigotimes_{J \in \pi} \prod_{j \in J} \tau_j, \quad \tau \in \otimes_{i \in I} A 
\end{equs}
where the product $ \prod $ is the product in $ A $, and where $ J \in \pi $ means that $ J $ is a block
of the partition $ \pi $.
We will also use this definition when the algebras in the tensor product are different. From the context, one can deduce the products used in every element of the partition.

\begin{theorem} \label{DCK_def}
The product $ \star_0 $ is the dual of the deformed Butcher-Connes-Kreimer coproduct $ \Delta_{\text{\tiny{DCK}}} $. One has
\begin{equs} \label{dual_DCK}
\langle \tau_1 \, \wh \curvearrowright \, \bar \tau_2 
 , \bar \tau_3 \rangle = \langle \tau_1 \hattimes \bar \tau_2
 , \Delta_{\text{\tiny{DCK}}} \bar \tau_3 \rangle, \, \, \langle \tau_1  \, \bar \curvearrowright \,  \tau_2
 , \tau_3 \rangle = \langle \tau_1 \hattimes  \tau_2
 , \bar{\Delta}_{\text{\tiny{DCK}}} \tau_3 \rangle 
\end{equs}
where $ \tau_1 \in S(\CP_E^V) $, $ \bar\tau_2, \bar \tau_3 \in \CP_E^V $ and $  \tau_2,  \tau_3 \in \CT_E^V $.
\end{theorem}
\begin{proof} First, one can easily check that for $ \tau_1, \tau_2 \in S(\CP_E^V) $, $  \tau_3 \in \CT_E^V $
\begin{equs} \label{product_delta}
\langle X^k \tau_1 \tau_2,  \tau_3 \rangle = \langle \tau_1 \otimes X^k\tau_2, \Delta \tau_3 \rangle
\end{equs}
where $\Delta$ is the extension given in \eqref{delta_p_trees}.
Such an identity results from manipulating the definition of the symmetry factor $ S $.
 We proceed by induction on the number of edges for proving $\eqref{dual_DCK}$. For $ \tau_1 =  \widetilde{\prod}_{i \in I} \mathcal{I}_{a_i}( \sigma_i) \in S(\CP_E^V) $ and $ \tau_2,  \tau_3 \in \CT_E^V $, one has
\begin{equs}
\langle \tau_1 \, \wh\curvearrowright \, \mathcal{I}_{b}( \tau_2)  , \mathcal{I}_{b}( \tau_3) \rangle  = 
 \langle \tau_1 \, \bar  \curvearrowright \,  \tau_2  ,  \tau_3 \rangle.
\end{equs}
Then by applying the induction hypothesis on $ \bar \curvearrowright $, we get
\begin{equs}
\langle  \tau_1 \, \bar  \curvearrowright \,  \tau_2  ,  \tau_3 \rangle & = \langle  \tau_1 \hattimes  \tau_2  ,  \bar{\Delta}_{\text{\tiny{DCK}}}  \tau_3 \rangle 
\\ & = \langle  \tau_1 \hattimes \CI_b( \tau_2)  ,  \left( \id \hattimes \CI_b \right) \bar{\Delta}_{\text{\tiny{DCK}}}  \tau_3 \rangle
\\ & = \langle  \tau_1 \hattimes \CI_b( \tau_2)  ,   \Delta_{\text{\tiny{DCK}}} \CI_b( \tau_3) \rangle
\end{equs}
where we have used \eqref{recur_def_deltaDCK} in the last line.
For $  \tau_2  = X^k  \tau_0 $ where $   \tau_0 $ has zero node decoration at the root, one gets from \eqref{def_deltabar}
\begin{equs}
  \widetilde{\prod}_{i \in I} \mathcal{I}_{a_i}( \sigma_i) \, \bar \curvearrowright \,
   \tau_2 & = \sum_{J \subset I}  \sum_{\ell_i \in\N^{d+1},}{k\choose (\ell_i)_{i \in J}}
X^{k-\sum_{i \in J} \ell_i} \\ & \left(\prod_{i \in J} \mathcal{I}_{a_i-\ell_i}( \sigma_i) \right) \left( \widetilde{\prod}_{i \in I \setminus J} \mathcal{I}_{a_i}( \sigma_i) \, \wh \curvearrowright \, 
   \tau_0 \right).
\end{equs}
Then by using the recursive definition of the map $ \bar{\Delta}_{\text{\tiny{DCK}}} $ given in \eqref{recur_def_deltaDCK}, one has
\begin{equs}
\bar{\Delta}_{\text{\tiny{DCK}}} & \left( X^k \prod_{i \in I} \CI_{a_i}( \sigma_{i})  \right)  = \left( \one \hattimes X^k \right) \prod_{i \in I} \left(\tilde{\Delta}_{\text{\tiny{DCK}}}  \CI_{a_i}( \sigma_{i}) \right. \\ &  +   \sum_{\ell_i \in \N^{d+1}} \frac{1}{\ell_i !} \CI_{a_i + \ell_i}( \sigma_i) \hattimes X^{\ell_i} 
 )
 \\ &  = \sum_{J \subset I}  \left( \sum_{\ell_i \in \N^{d+1}} \widetilde{\prod}_{i \in J}  \frac{1}{\ell_i !} \CI_{a_i + \ell_i}( \sigma_i) \hattimes X^{k + \sum_{i \in J} \ell_i}  \right) \left(\tilde{\Delta}_{\text{\tiny{DCK}}} \widetilde{\prod}_{i \in I \setminus J}   \CI_{a_i}( \sigma_{i}) \right)
\end{equs}
 where 
\begin{equs}
\tilde{\Delta}_{\text{\tiny{DCK}}}   \CI_{a}( \sigma) =  \Delta_{\text{\tiny{DCK}}} \CI_{a}(\sigma) - \CI_{a}(\sigma) \otimes \one.
\end{equs} 
One can match in the dual the factors of the previous two identities by using \eqref{product_delta}. Indeed, one has
\begin{equs}
 \langle & \tau_1 \, \bar  \curvearrowright \,  \tau_2  ,  \tau_3 \rangle  = \sum_{J \subset I} \langle   \sum_{\ell_i \in\N^{d+1},}{k\choose (\ell_i)_{i \in J}}
X^{k-\sum_{i \in J} \ell_i}  \left(\prod_{i \in J} \mathcal{I}_{a_i-\ell_i}( \sigma_i) \right) \hattimes \\ &  \left( \widetilde{\prod}_{i \in I \setminus J} \mathcal{I}_{a_i}( \sigma_i) \, \wh \curvearrowright \, 
   \tau_0 \right) , \Delta  \tau_3 \rangle
  \\ & =  \sum_{J \subset I} \langle   \left( \widetilde{\prod}_{i \in J} \mathcal{I}_{a_i}(\sigma_i)  \hattimes X^{k} \right) \hattimes   \left( \widetilde{\prod}_{i \in I \setminus J} \mathcal{I}_{a_i}( \sigma_i) \, \hattimes \, 
   \tau_0 \right) , \left(\bar{\Delta}_{\text{\tiny{DCK}}}^{\root}  \hattimes \tilde{\Delta}_{\text{\tiny{DCK}}} \right) \Delta  \tau_3 \rangle
  \\ & = \sum_{J \subset I} \langle   \left( \widetilde{\prod}_{i \in J} \mathcal{I}_{a_i}( \sigma_i)  \hattimes    \widetilde{\prod}_{i \in I \setminus J} \mathcal{I}_{a_i}( \sigma_i) \, \hattimes \, 
  X^k  \tau_0 \right) , \mathcal{M}^{(1)(3)(24)}\left(\bar{\Delta}_{\text{\tiny{DCK}}}^{\root}  \hattimes \tilde{\Delta}_{\text{\tiny{DCK}}} \right) \Delta \bar \tau_3 \rangle
  \\ & =  \langle   \left( \Delta \widetilde{\prod}_{i \in I} \mathcal{I}_{a_i}( \sigma_i)   \, \hattimes \, 
  X^k  \tau_0 \right) , \mathcal{M}^{(1)(3)(24)}\left(\bar{\Delta}_{\text{\tiny{DCK}}}^{\root}  \hattimes \tilde{\Delta}_{\text{\tiny{DCK}}} \right) \Delta \bar \tau_3 \rangle
  \\ & =  \langle   \left(  \widetilde{\prod}_{i \in I} \mathcal{I}_{a_i}( \sigma_i)   \, \hattimes \, 
  X^k  \tau_0 \right) , \mathcal{M}^{(13)(24)}\left(\bar{\Delta}_{\text{\tiny{DCK}}}^{\root}  \hattimes \tilde{\Delta}_{\text{\tiny{DCK}}} \right) \Delta  \tau_3 \rangle
  \\ & =  \langle \tau_1  \hattimes  \tau_2 , \bar \Delta_{\text{\tiny{DCK}}}  \tau_3 \rangle
\end{equs}  
where we have also used the identity:
\begin{equs}
\mathcal{M}^{(13)(24)}\left(\bar{\Delta}_{\text{\tiny{DCK}}}^{\root}  \hattimes \tilde{\Delta}_{\text{\tiny{DCK}}} \right) \Delta = \bar \Delta_{\text{\tiny{DCK}}}.
\end{equs}
Moreover, we have applied the induction hypothesis on $ \wh \curvearrowright   $ and used that
\begin{equs} \label{root_ident}
\langle  & \sum_{\ell_i \in\N^{d+1},}{k\choose (\ell_i)_{i \in J}}
X^{k-\sum_{i \in J} \ell_i}  \left(\prod_{i \in J} \mathcal{I}_{a_i-\ell_i}( \sigma_i) \right) ,  \tau_3 \rangle \\ & = \langle  \widetilde{\prod}_{i \in J} \mathcal{I}_{a_i}( \sigma_i) \hattimes  
X^{k}   , \bar{\Delta}_{\text{\tiny{DCK}}}^{\root}  \tau_3 \rangle.
\end{equs}
For proving \eqref{root_ident}, one has to check that  the combinatorial coefficients provided by the polynomials are the same:
\begin{equs}
{k\choose (\ell_i)_{i \in J}} \langle 
X^{k-\sum_{i \in J} \ell_i}  , X^{k-\sum_{i \in J} \ell_i}  \rangle&  = \frac{k!}{(k-\sum_{i \in J} \ell_i)! \prod_{i \in J} \ell_i !} ( k- \sum_{i \in J} \ell_i  )!
\\ & = \left( \prod_{i \in J} \frac{1}{  \ell_i !}\right) k! 
= \left(\prod_{i \in J} \frac{1}{  \ell_i !}\right) \langle 
X^{k}  , X^{k}  \rangle,
\end{equs}
which concludes the proof.
\end{proof}

\subsection{Deformation of a plugging product}
\label{sec::grafting}
Suppose here that the set $V$ of vertex decorations is equal to the commutative monoid $\Omega$. We define a plugging operator $ \rhd $ on $ \CT_E^{V} $ as the following. For every $ \sigma , \tau \in \CT_E^{V} $, we set
\begin{align*} 
\sigma \rhd \tau = \sum_{v\in N_{\tau}}\sigma \rhd_v  \tau
\end{align*} 
\label{plugging_p}
where $ \sigma \rhd_v  \tau $ is the tree obtained by the identification of the root of $ \sigma $ with the node $ v $. Both decorations of the root of $\sigma$ and $v$ are simply added in this process. Below, we give an example of this operation: 
\begin{equs} \label{ex_plugging}
 \begin{tikzpicture}[scale=0.2,baseline=0.1cm]
        \node at (0,0)  [dot,label= {[label distance=-0.2em]below: \scriptsize  $ \omega  $} ] (root) {};
         \node at (2,4)  [dot,label={[label distance=-0.2em]above: \scriptsize  $ \beta $}] (right) {};
         \node at (-2,4)  [dot,label={[label distance=-0.2em]above: \scriptsize  $ \alpha $} ] (left) {};
            \draw[kernel1] (right) to
     node [sloped,below] {\small }     (root); \draw[kernel1] (left) to
     node [sloped,below] {\small }     (root);
     \node at (-1,2) [fill=white,label={[label distance=0em]center: \scriptsize  $ a $} ] () {};
    \node at (1,2) [fill=white,label={[label distance=0em]center: \scriptsize  $ b $} ] () {};
     \end{tikzpicture}   \rhd \begin{tikzpicture}[scale=0.2,baseline=0.1cm]
        \node at (0,0)  [dot,label= {[label distance=-0.2em]below: \scriptsize  $ \delta $} ] (root) {};
         \node at (0,4)  [dot,label={[label distance=-0.2em]above: \scriptsize  $ \gamma $}] (right) {};
            \draw[kernel1] (right) to
     node [sloped,below] {\small }     (root); 
     \node at (0,2) [fill=white,label={[label distance=0em]center: \scriptsize  $ c $} ] () {};
     \end{tikzpicture}   = \begin{tikzpicture}[scale=0.2,baseline=0.1cm]
        \node at (0,0)  [dot,label= {[label distance=-0.2em]below: \scriptsize  $ \omega + \delta $} ] (root) {};
         \node at (4,4)  [dot,label={[label distance=-0.2em]above: \scriptsize  $ \gamma $}] (right) {};
         \node at (0,6)  [dot,label={[label distance=-0.2em]above: \scriptsize  $ \beta $}] (center) {};
         \node at (-4,4)  [dot,label={[label distance=-0.2em]above: \scriptsize  $ \alpha $} ] (left) {};
            \draw[kernel1] (right) to
     node [sloped,below] {\small }     (root); \draw[kernel1] (left) to
     node [sloped,below] {\small }     (root);
     \draw[kernel1] (center) to
     node [sloped,below] {\small }     (root);
      \node at (2,2) [fill=white,label={[label distance=0em]center: \scriptsize  $ c $} ] () {};
     \node at (-2,2) [fill=white,label={[label distance=0em]center: \scriptsize  $ a $} ] () {};
     \node at (0,3) [fill=white,label={[label distance=0em]center: \scriptsize  $ b  $} ] () {};
     \end{tikzpicture}
     + \begin{tikzpicture}[scale=0.2,baseline=0.1cm]
        \node at (0,0)  [dot,label= {[label distance=-0.2em]below: \scriptsize  $ \delta $} ] (root) {};
         \node at (-3,3)  [dot,label={[label distance=-0.2em]left: \scriptsize  $ \omega + \gamma $} ] (left) {};
          \node at (0,7)  [dot,label={[label distance=-0.2em]right: \scriptsize  $ \beta $} ] (center) {};
          \node at (-6,7)  [dot,label={[label distance=-0.2em]right: \scriptsize  $ \alpha $} ] (centerr) {};
          \draw[kernel1] (left) to
     node [sloped,below] {\small }     (root);
      \draw[kernel1] (center) to
     node [sloped,below] {\small }     (left);
     \draw[kernel1] (centerr) to
     node [sloped,below] {\small }     (left);
     \node at (-1.5,1.5) [fill=white,label={[label distance=0em]center: \scriptsize  $ c $} ] () {};
      \node at (-1.5,5) [fill=white,label={[label distance=0em]center: \scriptsize  $ b $} ] () {};
       \node at (-4.5,5) [fill=white,label={[label distance=0em]center: \scriptsize  $ a $} ] () {};
     \end{tikzpicture}
\end{equs}
This operation is different from \eqref{grafting_ident} where only planted trees are considered with no decorations at the root. It also does not coincide with $\curvearrowright^{a,\omega}$ defined on planted trees in \eqref{curve_insertion}. For later use we decompose the plugging product $\rhd$ into
\begin{equation}\label{splitplug}
\rhd:=\rhd^\root+\rhd^\nonroot,
\end{equation}
where $\sigma\rhd^\root\tau:=\sigma\rhd_{\rho(\tau)}\tau$ is the tree obtained by plugging at the root of $\tau$.
\begin{proposition} \label{pre-lie plugging}
The space $ \CT_E^{V}  $ endowed with $ \rhd $ is a pre-Lie algebra but not free.
\end{proposition}

\begin{proof} The proof of the pre-Lie relation is straightforward and left to the reader.
The non-freeness comes from the fact that any brace element $ (\sigma_1 \cdots \sigma_n) \rhd \tau $ is
obtained by summing up all possibilities of merging the roots of the components $ \sigma_j$
with \textit{distinct} vertices of $ \tau $. It therefore vanishes whenever $n$ exceeds the number of
vertices of $\tau$.
\end{proof}

In the sequel, we will use the shorthand notation $ \ell = (\ell_1,...,\ell_n) \in (\N^{d+1})^n $ and $ |\ell| = \sum_{i} \ell_i $.
On can define a deformed version of this  plugging operator in the same spirit as before. But we graft simultaneously several edges.
We set for a tree $ \sigma $ of the form $ X^k \prod_{i=1}^{n}\mathcal{I}_{a_i}(\sigma_i) $
\label{deformed_plugging_p}
\begin{equs} \label{formula_rhd}
\sigma \, \wh{\rhd} \, \tau :=\sum_{v\in N_{\tau}}\sigma \, \wh{\rhd}_v \, \tau
\end{equs}
where $\sigma \, \wh{\rhd}_v \, \tau  $ is given by
\begin{equs} \label{formula_rhd_bis}
\sigma \, \wh{\rhd}_v \, \tau :=\sum_{\ell \in(\N^{d+1})^n}{\Labn_v \choose \ell }\left( X^{k}\prod_{i=1}^{n}\mathcal{I}_{(a_i - \ell_i)}(\sigma_i)\right)  \rhd_v(\uparrow_v^{- |\ell|} \tau).
\end{equs}
As an example, the deformation of \eqref{ex_plugging} is given by
\begin{equs} \label{ex_plugging_defor}
 \begin{tikzpicture}[scale=0.2,baseline=0.1cm]
        \node at (0,0)  [dot,label= {[label distance=-0.2em]below: \scriptsize  $ \omega  $} ] (root) {};
         \node at (2,4)  [dot,label={[label distance=-0.2em]above: \scriptsize  $ \beta $}] (right) {};
         \node at (-2,4)  [dot,label={[label distance=-0.2em]above: \scriptsize  $ \alpha $} ] (left) {};
            \draw[kernel1] (right) to
     node [sloped,below] {\small }     (root); \draw[kernel1] (left) to
     node [sloped,below] {\small }     (root);
     \node at (-1,2) [fill=white,label={[label distance=0em]center: \scriptsize  $ a $} ] () {};
    \node at (1,2) [fill=white,label={[label distance=0em]center: \scriptsize  $ b $} ] () {};
     \end{tikzpicture}   \wh\rhd \begin{tikzpicture}[scale=0.2,baseline=0.1cm]
        \node at (0,0)  [dot,label= {[label distance=-0.2em]below: \scriptsize  $ \delta $} ] (root) {};
         \node at (0,4)  [dot,label={[label distance=-0.2em]above: \scriptsize  $ \gamma $}] (right) {};
            \draw[kernel1] (right) to
     node [sloped,below] {\small }     (root); 
     \node at (0,2) [fill=white,label={[label distance=0em]center: \scriptsize  $ c $} ] () {};
     \end{tikzpicture}   =
      \sum_{\ell} {\delta \choose \ell}  \begin{tikzpicture}[scale=0.2,baseline=0.1cm]
        \node at (0,0)  [dot,label= {[label distance=-0.2em]below: \scriptsize  $ \omega + \delta- |\ell| $} ] (root) {};
         \node at (4,4)  [dot,label={[label distance=-0.2em]above: \scriptsize  $ \gamma $}] (right) {};
         \node at (0,6)  [dot,label={[label distance=-0.2em]above: \scriptsize  $ \beta $}] (center) {};
         \node at (-4,4)  [dot,label={[label distance=-0.2em]above: \scriptsize  $ \alpha $} ] (left) {};
            \draw[kernel1] (right) to
     node [sloped,below] {\small }     (root); \draw[kernel1] (left) to
     node [sloped,below] {\small }     (root);
     \draw[kernel1] (center) to
     node [sloped,below] {\small }     (root);
      \node at (2,2) [fill=white,label={[label distance=0em]center: \scriptsize  $ c $} ] () {};
     \node at (-2,2) [fill=white,label={[label distance=-1em]left: \scriptsize  $ a - \ell_1 $} ] () {};
     \node at (0,4) [fill=white,label={[label distance=0em]center: \scriptsize  $ b - \ell_2  $} ] () {};
     \end{tikzpicture}
     + \sum_{\ell} {\gamma \choose \ell}  \begin{tikzpicture}[scale=0.2,baseline=0.1cm]
        \node at (0,0)  [dot,label= {[label distance=-0.2em]below: \scriptsize  $ \delta $} ] (root) {};
         \node at (-3,3)  [dot,label={[label distance=-0.2em]left: \scriptsize  $ \omega + \gamma - |\ell| $} ] (left) {};
          \node at (0,7)  [dot,label={[label distance=-0.2em]right: \scriptsize  $ \beta $} ] (center) {};
          \node at (-6,7)  [dot,label={[label distance=-0.2em]right: \scriptsize  $ \alpha $} ] (centerr) {};
          \draw[kernel1] (left) to
     node [sloped,below] {\small }     (root);
      \draw[kernel1] (center) to
     node [sloped,below] {\small }     (left);
     \draw[kernel1] (centerr) to
     node [sloped,below] {\small }     (left);
     \node at (-1.5,1.5) [fill=white,label={[label distance=0em]center: \scriptsize  $ c $} ] () {};
      \node at (-1.5,5) [fill=white,label={[label distance=0em]center: \scriptsize  $ b - \ell_2 $} ] () {};
       \node at (-4.5,5) [fill=white,label={[label distance=-1em]left: \scriptsize  $ a - \ell_1 $} ] () {};
     \end{tikzpicture}
\end{equs}
We decompose the previous bilinear operation as:
\begin{equs} \label{root_plugging}
 \wh{\rhd}  =  \wh{\rhd}^{\root}  +  \wh{\rhd}^{\nonroot} 
\end{equs}
where $  \wh{\rhd}^{\root}  $ (resp. $  \wh{\rhd}^{\nonroot}  $) is defined the same way as $\wh{\rhd}  $ except that we consider only the root (resp. we sum over the nodes except the root). 
\begin{proposition}
The space $ \CT_E^{V}  $ endowed with $ \wh{\rhd}$ is a pre-Lie algebra but not free.
\end{proposition} 
\begin{proof}
Let $ \sigma =  X^k \prod_{i=1}^{n} \mathcal{I}_{a_i}(\sigma_i)  $, $ \tau = X^{\bar k} \prod_{i=1}^{m} \mathcal{I}_{\bar a_i}( \tau_i) $ and $ w $ elements of $ \CT_E^{V} $. We want to prove that
\begin{equs} \label{pre_lie_ident}
\sigma \, \wh{\rhd} \, (\tau \, \wh{\rhd} \, w)-(\sigma \, \wh{\rhd} \, \tau) \, \wh{\rhd} \, w=
\tau \, \wh{\rhd} (\sigma \, \wh{\rhd} \, w)-(\tau \, \wh{\rhd} \, \sigma) \, \wh{\rhd} \, w.
\end{equs}
The only difficult part is when $ \sigma $ and $ \tau $ are plugged on the same node $ v \in N_{w}$ decorated by $ \Labn_v $.
In terms of the binomial coefficient, we get from the left hand side of \eqref{pre_lie_ident}
\begin{align*}
\sum_{\ell \in (\N^{d+1})^{n}} \sum_{\bar \ell \in (\N^{d+1})^{m}}  {\Labn_v\choose \bar \ell} {\Labn_v - |\bar \ell| + \bar k \choose  \ell } - \sum_{\ell,r \in (\N^{d+1})^{n}} \sum_{\bar \ell \in (\N^{d+1})^{m}} {\Labn_v \choose \bar \ell,  r} { \bar k \choose  \ell}.
\end{align*}
On the right hand side, we get
\begin{align*}
\sum_{\ell \in (\N^{d+1})^{n}} \sum_{\bar \ell \in (\N^{d+1})^{m}}  {\Labn_v\choose  \ell} {\Labn_v - | \ell| +  k \choose \bar \ell } - \sum_{\ell \in (\N^{d+1})^{n}} \sum_{\bar \ell, \bar r \in (\N^{d+1})^{m}} {\Labn_v \choose \ell, \bar r} {  k \choose \bar \ell}. 
\end{align*}
We first notice that
\begin{align*}
{\Labn_v \choose \bar \ell,  r} = {\Labn_v  \choose \bar \ell  } {\Labn_v - | \bar \ell | \choose   r}.
\end{align*}
By using Chu-Vandermonde identity, we get
\begin{align*}
\sum_{\ell,r \in (\N^{d+1})^{n}} {\Labn_v - | \bar \ell | \choose   r}  { \bar k \choose  \ell} = \sum_{\ell \in (\N^{d+1})^{n}} {\Labn_v - | \bar \ell | + \bar k \choose   \ell} 
\end{align*}
which allows us to conclude. The non-freeness argument is similar to the one
invoked in the proof of Proposition~\ref{pre-lie plugging}.
\end{proof}

\begin{remark} One cannot repeat the proof given in theorem~\ref{main}  because the product $ \rhd $ is not free. Instead, we have used the Chu-Vandermonde identity. 
\end{remark}

We want to know if the plugging operator $ \wh{\rhd} $ can be derived using the $ \Theta $ defined in Theorem~\eqref{main} before and the operator $ \rhd $. The answer is yes but it does not come from a direct transport of structure. Considering the transported plugging operator $ \widetilde{\rhd} $ defined as
\begin{align*}
\Theta( \sigma \, \rhd \, \tau ) =  \Theta(\sigma) \, \widetilde{\rhd} \, \Theta(\tau)  
\end{align*}
we will see that it does not coincide with $ \wh{\rhd} $. Let us first give an explicit expression of the linear isomorphism $ \Theta $:
\label{theta_plugging}
\begin{proposition} \label{rec_theta} Let $ \tau = X^k \prod_{i=1}^{n} \mathcal{I}_{a_i}(\tau_i) \in \CT_E^V$, one has:
\begin{align*}
\Theta(\tau) = \sum_{\ell \in (\N^{d+1})^{n}}{ k \choose \ell} X^{k- |\ell|} \prod_{i=1}^{n}\mathcal{I}_{(a_i- \ell_i)}( \Theta(\tau_i)).
\end{align*} 
\end{proposition}
\begin{proof}
 One has
$$\tau=(\tau_1\cdots \tau_n)\curvearrowright^{a_1\cdots a_n}X^k,$$
hence we easily get
\begin{eqnarray*}
\Theta(\tau)&=&\big(\Theta(\tau_1)\cdots\Theta(\tau_n)\big)\wh\curvearrowright^{a_1\cdots a_n}\Theta(X^k)\\
&=&\big(\Theta(\tau_1)\cdots\Theta(\tau_n)\big)\wh\curvearrowright^{a_1\cdots a_n}X^k\\
&=& \sum_{\ell \in (\N^{d+1})^{n}}{ k \choose \ell} X^{k- |\ell|} \prod_{i=1}^{n}\mathcal{I}_{(a_i- \ell_i)}\Big( \Theta(\tau_i)\Big)\hskip 8mm  \hbox{by \eqref{deformation_curve} and \eqref{extension_theta}}.
\end{eqnarray*}
\end{proof}

In the next proposition, we give an explicit expression of the product $ \widetilde{\rhd}^{\root} $ defined by
$$\Theta(\sigma\rhd^\root\tau)=\Theta(\sigma) \widetilde{\rhd}^{\root} \Theta(\tau).$$
Similarly to \eqref{splitplug} one has then:
\begin{align*} \label{root_plugging_t}
 \widetilde{\rhd}  =  \widetilde{\rhd}^{\root}  +  \widetilde{\rhd}^{\nonroot}.
\end{align*}

\begin{proposition}
Let $ \sigma =  X^k \prod_{i=1}^{n} \mathcal{I}_{a_i}(\sigma_i) $ and $ \tau = X^{\bar k} \prod_{i=1}^{\bar n} \mathcal{I}_{\bar a_i}( \tau_i) \in \CT_E^V$ , one has
\begin{equs} \label{root_formula}
\sigma \, \widetilde{\rhd}^{\root} \, \tau = & \sum_{\ell \in (\N^{d+1})^{n}} \sum_{\bar \ell \in (\N^{d+1})^{\bar n}} { \bar k \choose \ell }  {  k \choose \bar \ell} \\ & X^{k + \bar k - |\ell|-|\bar \ell|} \prod_{i=1}^{n}\mathcal{I}_{(a_i- \ell_i)}(\sigma_i)  \prod_{i=1}^{\bar n} \mathcal{I}_{(\bar a_i - \bar \ell_i)}(\tau_i)   .
\end{equs}
\end{proposition}
\begin{proof}The proof relies on the recursive definition of  $ \Theta $ given in Proposition~\ref{rec_theta}. Indeed, one has:
\begin{align*}
\Theta( \sigma \rhd^{\root} \tau ) = & \sum_{\ell, \bar \ell} { \bar k + k \choose \ell,\bar \ell}   X^{k + \bar k - |\ell |- |  \bar \ell|} \prod_{i=1}^{n}\mathcal{I}_{(a_i- \ell_i)}\Big(\Theta(\sigma_i)\Big)  \prod_{i=1}^{\bar n} \mathcal{I}_{(\bar a_i - \bar \ell_i)}\Big(\Theta(\tau_i) \Big)   .
\end{align*}
On the other hand, from \eqref{root_formula} one has
\begin{align*}
 \Theta(\sigma) & \, \widetilde{\rhd}^{\root} \, \Theta(\tau)  =  \sum_{\ell, \bar \ell, m, \bar m}  {  k  \choose \ell}  {  \bar k  \choose \bar \ell} { k  - | \ell |  \choose  \bar m } {  \bar k  - |\bar \ell |  \choose m} \\ & X^{k + \bar k - | \ell| - | \bar \ell| - |m|  - |  \bar m| } \prod_{i=1}^{n}\mathcal{I}_{(a_i- \ell_i- m_i)}\Big(\Theta(\sigma_i)\Big)  \prod_{i=1}^{\bar n} \mathcal{I}_{(\bar a_i - \bar \ell_i-  \bar m_i)}\Big(\Theta( \tau_i) \Big).
\end{align*}
Now, we perform the change of variable $ r_i = \ell_i +  m_i $ and $ \bar r_i = \bar \ell_i +  \bar m_i $:
\begin{align*}
 \Theta(\sigma) \, \widetilde{\rhd}^{\root} \, \Theta(\tau) & =  
 \sum_{\ell, \bar \ell, r, \bar r}  {  k  \choose \ell}  {  \bar k  \choose \bar \ell} { k  - | \ell|  \choose  \bar r- \bar \ell } {  \bar k  - |\bar \ell|  \choose  r -  \ell} \\ & X^{k + \bar k - | r|- | \bar r|} \prod_{i=1}^{n}\mathcal{I}_{(a_i- r_i)}\Big(\Theta(\sigma_i)\Big)  \prod_{i=1}^{\bar n} \mathcal{I}_{(\bar a_i - \bar r_i)}\Big(\Theta(\tau_i)\Big).  
\end{align*}
 We fix $ r $ and $ \bar r $ and by the Chu-Vandermonde identity, we get:
 \begin{align*}
 \sum_{\ell, \bar \ell}  {  k  \choose \ell}  {  \bar k  \choose \bar \ell} { k  - | \ell | \choose  \bar r- \bar \ell } {  \bar k  - | \bar \ell|  \choose  r -  \ell} = \sum_{\ell, \bar \ell}  {  k  \choose \bar r - \bar \ell, \ell}  { \bar k  \choose  \bar \ell, r- \ell}  ={ \bar k + k  \choose  \bar r,  r} . 
 \end{align*}
 Then we conclude from the previous identity that:
 \begin{align*}
 \Theta( \sigma \, \rhd^{\root} \, \tau ) =  \Theta(\sigma) \, \widetilde{\rhd}^{\root} \, \Theta(\tau).
 \end{align*}
\end{proof}

\begin{remark} One can derive a formula for $\sigma \,  \widetilde{\rhd}_{v} \, \tau $ when $ v $ is not the root of $ \tau $. It will be similar to \eqref{formula_rhd}  and it will be symmetric like $ \widetilde{\rhd}^{\root} $  when it turns out to modify the edge decoration.
\end{remark}

The next propostion reveals the main difference between $ \wh \rhd $ and $\widetilde{\rhd}  $ which can be observed at the root.

\begin{proposition} \label{not_commute}
The operators  $ \rhd, \wh \rhd, \widetilde{\rhd} $ satisfy the following properties:
\begin{align*}
\sigma \, \rhd^{\root} \tau =  \tau \rhd^{\root} \sigma, \quad \sigma \, \widetilde{\rhd}^{\root} \tau =  \tau \, \widetilde{\rhd}^{\root} \sigma, \quad \sigma \, \wh\rhd^{\root}  \tau \neq  \tau \, \wh \rhd^{\root} \sigma.
\end{align*} 
\end{proposition}
\begin{proof}
For $ \rhd^{\root} $, the property follows from its definition. Then $ \Theta $ transports this structure to $ \widetilde{\rhd}^{\root} $. But with an explicit computation, one can check that $ \sigma \, \wh\rhd^{\root} \, \tau \neq  \tau \, \wh \rhd^{\root} \sigma  $.
\end{proof}

\begin{remark}
The main consequence  of Proposition~\ref{not_commute} is that $ \wh \rhd^{\root}  $ is not directly obtained from $ \Theta $ which makes its construction more subtle. Proposition~\ref{insertion_poly} shows that one needs to take into account insertion of polynomials after the plugging.
\end{remark}

\begin{proposition} \label{insertion_poly}
Let $ \sigma, \tau \in \CT_E^V $,
one has for every $ v \in N_{\tau} $
\begin{equs} \label{def_plugging}
\sigma \, \wh{\rhd}_{v} \, \tau = \uparrow^{n_{\sigma}}_{v} \left( \Pi \sigma \, \widetilde{\rhd}_{v} \,  \tau \right)
\end{equs}
where $ \Pi $ sets to zero the node decoration equal to $ n_{\sigma} $ at the root of $ \sigma $. 
\end{proposition}
\begin{proof} When $ \sigma $ has zero node decoration from \eqref{root_formula}, one has
\begin{equs}
\sigma \, \wh{\rhd}_{v} \, \tau =   \sigma \, \widetilde{\rhd}_{v} \,  \tau.
\end{equs}
Then, we conclude from \eqref{formula_rhd} which gives
\begin{equs}
\sigma \, \wh{\rhd}_{v} \, \tau = \uparrow^{n_{\sigma}}_{v} \left( \Pi \sigma \, \wh{\rhd}_{v} \,  \tau \right).
\end{equs}
\end{proof}

\begin{remark}
The identity \eqref{def_plugging} allows us to write:
\begin{equs}
\sigma \, \wh{\rhd}_{v} \, \tau & = \uparrow^{n_{\sigma}}_{v} \left( \Pi \sigma \, \widetilde{\rhd}_{v} \,  \tau \right) = \uparrow^{n_{\sigma}}_{v} \Theta^{-1}\left( \Theta (\Pi \sigma) \, \rhd_{v} \,  \Theta(\tau) \right) 
\\ & =\Theta^{-1} \, \hat \uparrow^{n_{\sigma}}_{v} \left( \Theta (\Pi \sigma) \, \rhd_{v} \,  \Theta(\tau) \right)
\end{equs}
where $  \hat \uparrow^{n_{\sigma}}_{v}  $ is given as  in \eqref{deformed_uparrow}.
\end{remark}

As for $ \wh{\curvearrowright} $, we apply the Guin-Oudom construction to the pre-Lie product $  \wh{\rhd} $ on the space $\mathcal{T}_E^V$. 
We  now consider the symmetric algebra $S( \mathcal{T}_E^V )$ with its usual coproduct $\Delta$ defined for every $ \tau \in \mathcal{T}_E^V $ by:
\begin{equs} \label{delta_trees}
\Delta \tau &=\tau \otimes \one + \one \otimes \tau. 
\end{equs}
On $ S(\mathcal{T}_E^V) $, we identify single node with no decorations as the empty forest.  We obtain a product $ \bullet_2 $ on $ S(\mathcal{T}_E^V) $ \label{forest_f} defined as in \eqref{e::symmetric product}. One can notice that for $ x \in S(\mathcal{T}_E^V) $, $ y \in \mathcal{T}_E^V $:
\begin{equs}
x \bullet_2 y = 0 
\end{equs}
when the number of trees in $ x $ is larger than the number of nodes in 
$ y $. \label{star_2}
Then, we define the following product $ \star_2 $ as:
\begin{equs}
w \star_2 v = \sum_{(w)} \left(  (w^{(1)} \bullet_2 v) w^{(2)} \right). 
\end{equs}

We denote by $ \mathfrak{g}_2 $ the Lie algebra $ (\mathcal{T}_E^V, \left[ \cdot \right]) $ where the bracket is given for $ \sigma, \tau \in \CT_E^V $ by:
\begin{equs}
\left[ \sigma,  \tau  \right] = \sigma \, \wh \rhd \,  \tau -  \tau \, \wh \rhd \,  \sigma.
\end{equs}

\begin{proposition}
The space $ (S(\mathcal{T}_E^V ), \star_2, \Delta) $ is a Hopf algebra isomorphic to the enveloping algebra $ \CU(\mathfrak{g}_2) $  where the Lie algebra $ \mathfrak{g}_2 $  is $ \mathcal{T}_E^V $ 
endowed with the
antisymmetrization of the pre-Lie product $\wh \rhd$.
\end{proposition}
Let us recall from \cite[Sec. 4.2]{BR18} a symbolic notation for decorated forests with one distinguished tree. This notation was crucial in that paper for deriving a recursive formula for an extraction-contraction coproduct. Here it plays another important role in identifying the dual of the product $ \star_2$. The linear span $ \mathcal{T}_E^V\otimes S(\mathcal{T}_E^V ) $ of forests with one distinguished tree is denoted by \label{forest_f_d} $ \hat S(\mathcal{T}_E^V ) $. The map
$$\tau\otimes\tau_1\cdots\tau_n\longmapsto \tau\tau_1\cdots\tau_n$$
from $ \hat S(\mathcal{T}_E^V ) $ into $S(\mathcal{T}_E^V ) $ is denoted by \label{cal_c} $\mathcal C$. Adopting the notation $\tau\mathcal C(\tau_1)\cdots\mathcal C(\tau_n)$ for the left-hand side, the map $ \mathcal{C}  $ satisfies 
\begin{equs}
\mathcal{C}\big(\tau\mathcal C(\tau_1)\cdots\mathcal C(\tau_n)\big) = \mathcal{C}(\tau)\mathcal C(\tau_1)\cdots\mathcal C(\tau_n).
\end{equs}
 We identify $ S(\mathcal{T}_E^V ) $ as a subspace of $ \hat S(\mathcal{T}_E^V ) $ by adding a distinguished single node tree without decoration to each element of $ S(\mathcal{T}_E^V ) $. Then, $ \mathcal{C} $ is an operator from $ \hat S(\mathcal{T}_E^V ) $ into itself which is in agreement with the first use of this symbol in \cite[Sec. 4.2]{BR18}. 
We endow $ \hat S(\mathcal{T}_E^V ) $ with the following product:
\begin{equs}
 \left(  \sigma \widetilde{\prod}_i \mathcal{C}(\sigma_i) \right) \left(  \tau \widetilde{\prod}_j \mathcal{C}( \tau_j) \right) = \sigma  \tau \widetilde{\prod}_i \mathcal{C}(\sigma_i)  \widetilde{\prod}_j \mathcal{C}(\tau_j) 
\end{equs}
where $ \sigma  \tau  $ is now the tree product between $ \sigma  $ and $ \tau $.
We extend $ \Delta  $ on $ \hat S(\mathcal{T}_E^V )  $ by defining it on the distinguished tree as multiplicative for the tree product. Now for monomials (identified with one-vertex trees), one has:
\begin{equs} \label{ploynomial_splitting}
\Delta X = X \otimes \one +  \one \otimes X
\end{equs}
which gives rise to the following identity
\begin{equs}
\Delta X^k=\sum_{\ell} {k\choose \ell}X^{\ell}\otimes X^{k-\ell}.
\end{equs}

\label{delta_dp} \label{delta_dp_bar}
We consider the following maps $ \Delta_{\text{\tiny{DP}}} :\mathcal{T}_E^V  \rightarrow \hat S(\mathcal{T}_E^V ) \hattimes \mathcal{T}_E^V $ and $ \bar{\Delta}_{\text{\tiny{DP}}} :\mathcal{T}_E^V  \rightarrow  S(\mathcal{T}_E^V ) \hattimes \mathcal{T}_E^V $  defined recursively by:
\begin{equs}
\Delta_{\text{\tiny{DP}}} X & = X \hattimes \one + \one \hattimes X, \quad \bar{\Delta}_{\text{\tiny{DP}}}X= X \hattimes \one , \quad\bar{\Delta}_{\text{\tiny{DP}}} \CI_a(\tau)  = \left( \mathcal{C} \hattimes \CI_a \right)  \Delta_{\text{\tiny{DP}}} \tau   
\\ \Delta_{\text{\tiny{DP}}} \CI_a(\tau) & =  \bar{\Delta}_{\text{\tiny{DP}}} \CI_a(\tau) + \sum_{\ell \in \N^{d+1}} \frac{1}{\ell !} \CI_{a + \ell}(\tau) \hattimes X^{\ell}.
\end{equs}
Then, they are extended multiplicatively.
\begin{theorem} \label{DP_deformation}
The product $ \star_2 $ is the dual of the deformed coproduct $ \Delta_2 = \left( \mathcal{C} \otimes \id \right) \Delta_{\text{\tiny{DP}}} $. \label{delta_2_2} One has
\begin{equs}
\langle \hat \tau_1 \hattimes \tau_2, \Delta_2 \tau_3 \rangle   =  \langle \hat \tau_1 \,  \wh{\rhd} \,  \tau_2,  \tau_3 \rangle, 
\quad \langle \hat \tau_1 \hattimes  \tau_2,  \bar{\Delta}_{\text{\tiny{DP}}} \tau_3 \rangle  =  \langle \hat \tau_1 \wh{\rhd}^{\nonroot}   \tau_2,  \tau_3 \rangle
\end{equs}
where $  \hat \tau_1 \in  S(\mathcal{T}_E^V )$ and $  \tau_2,  \tau_3 \in \mathcal{T}_E^V $.
\end{theorem}
\begin{proof} 
The proof follows essentialy the same steps as in Theorem~\ref{DCK_def}.
First, one can easily check that for $ \tau_1, \tau_2 \in \CT_E^V $, $  \tau_3 \in \CT_E^V $
\begin{equs}
\langle X^k \tau_1 \tau_2,  \tau_3 \rangle = \langle X^n \tau_1 \hattimes X^{n-k}\tau_2, \Delta  \tau_3 \rangle
\end{equs}
where here we view an element of $ \CT_E^V $ as an element of $ \hat S(\CT_E^V) $ having one distinguished tree and an empty forest.
For $  \tau_2 = X^k \tau_0 = X^{k} \prod_{j =1}^n \CI_{a_j}(\sigma_j) \in  \mathcal{T}_E^V  $ and $ \hat \tau_1  = \widetilde\prod_{i \in I} \mathcal{C} \left( X^{k_i} \prod_{j \in J_i} \CI_{a_{ij}}(\tau_{ij}) \right) \in  S(\mathcal{T}_E^V ) $, one has
\begin{equs}
  \hat \tau_1 & \, \wh{\rhd} \, \tau_2  = \sum_{i \in I}  \sum_{\ell_j \in\N^{d+1}}{k\choose (\ell_j)_{j \in J_i}}
X^{k-\sum_{j \in J_i} \ell_j + k_i}  \left(\prod_{j \in J_i} \mathcal{I}_{a_{ij}-\ell_j}(\tau_{ij}) \right) 
\\ &  \widetilde\prod_{r \in I \setminus \lbrace i \rbrace} \mathcal{C} \left( X^{k_r} \prod_{j \in J_r} \CI_{a_{rj}}(\tau_{rj}) \right)  \wh{\rhd}^{\nonroot} \prod_{j =1}^n \CI_{a_j}(\sigma_{j}) + X^k \left(\hat \tau_1 \wh \rhd^{\nonroot } \tau_0 \right).
\end{equs}
Then by using the recursive definition of the coproduct $ \Delta_{\text{\tiny{DP}}} $, one has
\begin{equs}
& \Delta_{\text{\tiny{DP}}}   \left( X^m \prod_{j \in J} \CI_{a_j}(\tau_{j})  \right)  = \Delta_{\text{\tiny{DP}}} X^m \prod_{j \in J} \left(\bar{\Delta}_{\text{\tiny{DP}}}  \CI_{a_j}(\tau_{j}) \right. \\ 
& \hskip 5mm +   \sum_{\ell_j \in \N^{d+1}} \frac{1}{\ell_j !} \CI_{a_j + \ell_j}(\tau_j) \hattimes X^{\ell_j} 
 ) \\ 
 &  = \sum_{I \subset J} \sum_{k, \ell_i \in \N^{d+1}} {m \choose k}  \left( X^{m-k}\prod_{i \in I}  \frac{\CI_{a_i + \ell_i}(\tau_i)}{\ell_i !} \hattimes X^{k + \sum_{i \in J} \ell_i}  \right) \left(\prod_{i \in J \setminus I} \bar{\Delta}_{\text{\tiny{DP}}}  \CI_{a_i}(\tau_{i}) \right).
\end{equs}
Then the proof follows the same lines as in Theorem~\ref{DCK_def}, where now one uses $ \Delta $ defined on $ \CT_E^V $ seperating the polynomial decorations, see \eqref{ploynomial_splitting}.
The only property missing is to match the combinatorial coefficient in the dual when we plug at the root similarly as \eqref{root_ident}. Indeed, one has
\begin{equs}
{n + \sum_{i \in J} \ell_i \choose (\ell_i)_{i \in J}} \langle 
X^{n+k}  , X^{n+k}  \rangle&  = {n + \sum_{i \in J} \ell_i \choose (\ell_i)_{i \in J}} (n+k)!
\end{equs}
On the other hand, one gets:
\begin{equs}
 \frac{1}{  \prod_{i \in J} \ell_i !}{n + k \choose k}  & \langle 
X^{k}  \otimes  X^{n + \sum_{i \in J} \ell_i } , X^{k}  \otimes  X^{n + \sum_{i \in J} \ell_i }  \rangle 
\\ &  = \frac{k! (n+\sum_{i \in J} \ell_i)!}{  \prod_{i \in J} \ell_i !}{n + k \choose k}.
\end{equs}
Then the coefficients coincide, which allows us to conclude.
\end{proof}

We consider the unital algebra morphism $ \CK : S(\CT_E^V) \rightarrow \CT_E^V $ given by $ \CK(\one) = \bullet_0 $ and
\begin{equs}\label{ident_i}
 \CK\left( \widetilde\prod_i \CC(  \tau_i )\right) = \prod_i  \tau_i
\end{equs}
where $ \tau_i \in  \CT_E^V$ and the product $ \prod_i \tau_i$ is the tree product.
The map $ \CK $ merges all the trees into one by identifying the roots. The decorations at the roots are added to the new root. One easily checks that
\begin{equs}
\Delta_2 \CK = \left( \CK \hattimes \CK \right) \Delta_2.
\end{equs}
Then the coproduct $ \Delta_2 : S(\CT_E^V)  \rightarrow S(\CT_E^V) \hattimes S(\CT_E^V) $ can be viewed as a coproduct from $ \CT_E^V  $ into $ \CT_E^V \hattimes \CT_E^V  $ and it will be recursively defined by:
\begin{equs}
\Delta_{2} X & = X \hattimes \one + \one \hattimes X,  
\\ \Delta_{2} \CI_a(\tau) & =  \left( \id \hattimes \CI_a \right)\Delta_{2} \tau + \sum_{\ell \in \N^{d+1}} \frac{1}{\ell !} \CI_{a + \ell}(\tau) \hattimes X^{\ell}.
\end{equs} 
We still denote the product associated to $ \Delta_2 $ by $ \star_2 $.
 The dual map $ \CK^{*} $ of $ \CK $ corresponds to  identifying all the ways to split a tree into blocks:
\begin{equs} \label{block_decomposition}
\CK^*\left( X^k \prod_{i \in I} \CI_{a_i}(\tau_i)\right) & = \sum_{k_1 + \cdots + k_n = k} \sum_{J_1 \sqcup \cdots \sqcup J_{n-m}= I} \widetilde\prod_{j=1}^m \mathcal{C}(X^{k_j}) \\ &  \widetilde\prod_{ \ell = m+1}^n \CC\left(  X^{k_{\ell}} \prod_{j \in J_{\ell-m}} \CI_{a_j}(\tau_j) \right).
\end{equs}

In the next Proposition, we present a connection between the two products $ \wh \curvearrowright $ and $ \star_2 $.

\begin{proposition} \label{plugging_grafting}
One has for $ \sigma = X^k \prod_{i\in I} \CI_{a_i}( \sigma_i),  \tau \in \CT_E^V $
\begin{equs} \label{link_plugging_grafting}
\CI_b \left( \sigma \star_2  \tau \right) = \tilde{\uparrow}^{k}_{N_{ \tau}} \left( \prod_{i\in I} \CI_{a_i}( \sigma_i) \,  \wh \curvearrowright \, \CI_b ( \tau) \right)
\end{equs}
where $ \tilde{\uparrow}^{k}_{N_{ \tau}} $ is defined by
\begin{equs} \label{splitting_polynomials}
\tilde{\uparrow}^{k}_{N_{\tau}} =
\sum_{k = \sum_{v \in N_{\tau}} k_v} \uparrow^{k_v}_{v}.
\end{equs}
\end{proposition}
\begin{proof}  From the definition of $ \star_2 $. one has:
\begin{equs}
 \sigma \star_2  \tau = (\CK^{*}  \sigma) \, \wh{\rhd} \, \tau.
\end{equs}
Due to the property of the product $ \wh{\rhd} $, one can have only $ |N_\tau| $ blocks containing  polynomials decorations at the root in \eqref{block_decomposition}. Therefore, one gets
\begin{equs}
 \sigma \star_2  \tau = \tilde{\uparrow}^{k}_{N_{\tau}} \left( \CK^{*}  \prod_{i\in I} \CI_{a_i}( \sigma_i) \right) \, \wh{\rhd} \, \tau.
\end{equs}
Then, we can conclude with the following identity
\begin{equs}
\CI_b \left( \Big(\CK^{*}  \prod_{i\in I} \CI_{a_i}( \sigma_i) \Big) \, \wh{\rhd} \, \tau \right) =  \prod_{i\in I} \CI_{a_i}( \sigma_i) \,  \wh \curvearrowright \, \CI_b ( \tau).
\end{equs}
Indeed, $ \CK^*$ groups branches that will be grafted onto the same node in $ \tau $.
\end{proof}
 \begin{remark} We could have used the right hand side of \eqref{link_plugging_grafting} for defining the product $ \star_2 $. Then, it is not trivial to see that the product $ \star_2 $ is associative. Moreover, the identity \eqref{splitting_polynomials} corresponds to the identification $ \CK $ on the polynomials.  We find natural to apply this splitting to trees.
 \end{remark}
 
 \begin{remark} When $ k= 0  $, one recovers the product $ \star_0 $ given in \eqref{def_*}.
 \end{remark}
 
 We finish this subsection by presenting an example of computation for $ \star_2 $. We consider the following trees:
 \begin{equs} 
\sigma = \begin{tikzpicture}[scale=0.2,baseline=0.1cm]
        \node at (0,0)  [dot,label= {[label distance=-0.2em]below: \scriptsize  $ \delta  $} ] (root) {};
         \node at (2,4)  [dot,label={[label distance=-0.2em]above: \scriptsize  $ \gamma $}] (right) {};
         \node at (-2,4)  [dot,label={[label distance=-0.2em]above: \scriptsize  $ \beta $} ] (left) {};
            \draw[kernel1] (right) to
     node [sloped,below] {\small }     (root); \draw[kernel1] (left) to
     node [sloped,below] {\small }     (root);
     \node at (-1,2) [fill=white,label={[label distance=0em]center: \scriptsize  $ b $} ] () {};
    \node at (1,2) [fill=white,label={[label distance=0em]center: \scriptsize  $ c $} ] () {};
     \end{tikzpicture}, \qquad 
     \tau =  \begin{tikzpicture}[scale=0.2,baseline=0.1cm]
        \node at (0,0)  [dot,label= {[label distance=-0.2em]below: \scriptsize  $ \omega $} ] (root) {};
         \node at (0,4)  [dot,label={[label distance=-0.2em]above: \scriptsize  $ \alpha $}] (right) {};
            \draw[kernel1] (right) to
     node [sloped,below] {\small }     (root); 
     \node at (0,2) [fill=white,label={[label distance=0em]center: \scriptsize  $ a $} ] () {};
     \end{tikzpicture}, \quad
 \CK^{*} \sigma  = \sum_{\delta=\delta_1 + \delta_2} ( \bullet_{\delta_1}  \begin{tikzpicture}[scale=0.2,baseline=0.1cm]
        \node at (0,0)  [dot,label= {[label distance=-0.2em]below: \scriptsize  $ \delta_2  $} ] (root) {};
         \node at (2,4)  [dot,label={[label distance=-0.2em]above: \scriptsize  $ \gamma $}] (right) {};
         \node at (-2,4)  [dot,label={[label distance=-0.2em]above: \scriptsize  $ \beta $} ] (left) {};
            \draw[kernel1] (right) to
     node [sloped,below] {\small }     (root); \draw[kernel1] (left) to
     node [sloped,below] {\small }     (root);
     \node at (-1,2) [fill=white,label={[label distance=0em]center: \scriptsize  $ b $} ] () {};
    \node at (1,2) [fill=white,label={[label distance=0em]center: \scriptsize  $ c $} ] () {};
     \end{tikzpicture} +  \begin{tikzpicture}[scale=0.2,baseline=0.1cm]
        \node at (0,0)  [dot,label= {[label distance=-0.2em]below: \scriptsize  $ \delta_1 $} ] (root) {};
         \node at (0,4)  [dot,label={[label distance=-0.2em]above: \scriptsize  $ \beta $}] (right) {};
            \draw[kernel1] (right) to
     node [sloped,below] {\small }     (root); 
     \node at (0,2) [fill=white,label={[label distance=0em]center: \scriptsize  $ b $} ] () {};
     \end{tikzpicture} \, \begin{tikzpicture}[scale=0.2,baseline=0.1cm]
        \node at (0,0)  [dot,label= {[label distance=-0.2em]below: \scriptsize  $ \delta_2 $} ] (root) {};
         \node at (0,4)  [dot,label={[label distance=-0.2em]above: \scriptsize  $ \gamma $}] (right) {};
            \draw[kernel1] (right) to
     node [sloped,below] {\small }     (root); 
     \node at (0,2) [fill=white,label={[label distance=0em]center: \scriptsize  $ c $} ] () {};
     \end{tikzpicture} + \cdots ).
 \end{equs}
Therefore, $ \sigma \star_2 \tau$ is given by
\begin{equs} \label{explicit_star_2}
 \sigma  \star_2 \tau & = \sum_{\delta=\delta_1 + \delta_2} ( \bullet_{\delta_1}  \begin{tikzpicture}[scale=0.2,baseline=0.1cm]
        \node at (0,0)  [dot,label= {[label distance=-0.2em]below: \scriptsize  $ \delta_2  $} ] (root) {};
         \node at (2,4)  [dot,label={[label distance=-0.2em]above: \scriptsize  $ \gamma $}] (right) {};
         \node at (-2,4)  [dot,label={[label distance=-0.2em]above: \scriptsize  $ \beta $} ] (left) {};
            \draw[kernel1] (right) to
     node [sloped,below] {\small }     (root); \draw[kernel1] (left) to
     node [sloped,below] {\small }     (root);
     \node at (-1,2) [fill=white,label={[label distance=0em]center: \scriptsize  $ b $} ] () {};
    \node at (1,2) [fill=white,label={[label distance=0em]center: \scriptsize  $ c $} ] () {};
     \end{tikzpicture} +  \begin{tikzpicture}[scale=0.2,baseline=0.1cm]
        \node at (0,0)  [dot,label= {[label distance=-0.2em]below: \scriptsize  $ \delta_1 $} ] (root) {};
         \node at (0,4)  [dot,label={[label distance=-0.2em]above: \scriptsize  $ \beta $}] (right) {};
            \draw[kernel1] (right) to
     node [sloped,below] {\small }     (root); 
     \node at (0,2) [fill=white,label={[label distance=0em]center: \scriptsize  $ b $} ] () {};
     \end{tikzpicture} \, \begin{tikzpicture}[scale=0.2,baseline=0.1cm]
        \node at (0,0)  [dot,label= {[label distance=-0.2em]below: \scriptsize  $ \delta_2 $} ] (root) {};
         \node at (0,4)  [dot,label={[label distance=-0.2em]above: \scriptsize  $ \gamma $}] (right) {};
            \draw[kernel1] (right) to
     node [sloped,below] {\small }     (root); 
     \node at (0,2) [fill=white,label={[label distance=0em]center: \scriptsize  $ c $} ] () {};
     \end{tikzpicture} ) \, \wh{\rhd} \, 
     \begin{tikzpicture}[scale=0.2,baseline=0.1cm]
        \node at (0,0)  [dot,label= {[label distance=-0.2em]below: \scriptsize  $ \omega $} ] (root) {};
         \node at (0,4)  [dot,label={[label distance=-0.2em]above: \scriptsize  $ \alpha $}] (right) {};
            \draw[kernel1] (right) to
     node [sloped,below] {\small }     (root); 
     \node at (0,2) [fill=white,label={[label distance=0em]center: \scriptsize  $ a $} ] () {};
     \end{tikzpicture}
\\ 
 & = \sum_{\delta = \delta_1 + \delta_2} \Big( \sum_{\ell} {\omega \choose \ell}  \begin{tikzpicture}[scale=0.2,baseline=0.1cm]
        \node at (0,0)  [dot,label= {[label distance=-0.2em]below: \scriptsize  $ \omega + \delta_2 - |\ell| $} ] (root) {};
         \node at (4,4)  [dot,label={[label distance=-0.2em]above: \scriptsize  $ \gamma $}] (right) {};
         \node at (0,6)  [dot,label={[label distance=-0.2em]above: \scriptsize  $ \beta $}] (center) {};
         \node at (-4,4)  [dot,label={[label distance=-0.2em]above: \scriptsize  $ \alpha + \delta_1 $} ] (left) {};
            \draw[kernel1] (right) to
     node [sloped,below] {\small }     (root); \draw[kernel1] (left) to
     node [sloped,below] {\small }     (root);
     \draw[kernel1] (center) to
     node [sloped,below] {\small }     (root);
      \node at (2,2) [fill=white,label={[label distance=-1em]right: \scriptsize  $ c -\ell_2$} ] () {};
     \node at (-2,2) [fill=white,label={[label distance=0em]center: \scriptsize  $ a $} ] () {};
     \node at (0,4) [fill=white,label={[label distance=0em]center: \scriptsize  $ b - \ell_1  $} ] () {};
     \end{tikzpicture}
     + \sum_{\ell} {\alpha \choose \ell} \begin{tikzpicture}[scale=0.2,baseline=0.1cm]
        \node at (0,0)  [dot,label= {[label distance=-0.2em]below: \scriptsize  $ \omega + \delta_1 $} ] (root) {};
         \node at (-3,3)  [dot,label={[label distance=-0.2em]left: \scriptsize  $ \alpha + \delta_2 -|\ell| $} ] (left) {};
          \node at (0,7)  [dot,label={[label distance=-0.2em]right: \scriptsize  $ \gamma $} ] (center) {};
          \node at (-6,7)  [dot,label={[label distance=-0.2em]right: \scriptsize  $ \beta $} ] (centerr) {};
          \draw[kernel1] (left) to
     node [sloped,below] {\small }     (root);
      \draw[kernel1] (center) to
     node [sloped,below] {\small }     (left);
     \draw[kernel1] (centerr) to
     node [sloped,below] {\small }     (left);
     \node at (-1.5,1.5) [fill=white,label={[label distance=0em]center: \scriptsize  $ a $} ] () {};
      \node at (-1.5,5) [fill=white,label={[label distance=0em]center: \scriptsize  $ c - \ell_2 $} ] () {};
       \node at (-4.5,5) [fill=white,label={[label distance=-1em]left: \scriptsize  $ b - \ell_1 $} ] () {};
     \end{tikzpicture} \\ & +
     \sum_{\ell} {\omega \choose \ell_1} {\alpha \choose \ell_2}  \begin{tikzpicture}[scale=0.2,baseline=0.1cm]
        \node at (0,0)  [dot,label= {[label distance=-0.2em]below: \scriptsize  $ \omega + \delta_1 -\ell_1 $} ] (root) {};
         \node at (-3,3)  [dot,label={[label distance=-0.2em]left: \scriptsize  $ \alpha + \delta_2 - \ell_2 $} ] (left) {};
          \node at (0,6)  [dot,label={[label distance=-0.2em]right: \scriptsize  $ \gamma $} ] (center) {};
          \node at (3,3)  [dot,label={[label distance=-0.2em]right: \scriptsize  $ \beta $} ] (centerr) {};
          \draw[kernel1] (left) to
     node [sloped,below] {\small }     (root);
      \draw[kernel1] (center) to
     node [sloped,below] {\small }     (left);
     \draw[kernel1] (centerr) to
     node [sloped,below] {\small }     (root);
     \node at (-1.5,1.5) [fill=white,label={[label distance=0em]center: \scriptsize  $ a $} ] () {};
      \node at (-1.5,4.5) [fill=white,label={[label distance=0em]center: \scriptsize  $ c -\ell_2 $} ] () {};
       \node at (1.5,1.5) [fill=white,label={[label distance=-1em]right: \scriptsize  $ b - \ell_1 $} ] () {};
     \end{tikzpicture}      
     + \sum_{\ell} {\alpha \choose \ell_1} {\omega \choose \ell_2} \begin{tikzpicture}[scale=0.2,baseline=0.1cm]
        \node at (0,0)  [dot,label= {[label distance=-0.2em]below: \scriptsize  $ \omega + \delta_2 - \ell_2 $} ] (root) {};
         \node at (-3,3)  [dot,label={[label distance=-0.2em]left: \scriptsize  $ \alpha + \delta_1 - \ell_1 $} ] (left) {};
          \node at (0,6)  [dot,label={[label distance=-0.2em]right: \scriptsize  $ \beta $} ] (center) {};
          \node at (3,3)  [dot,label={[label distance=-0.2em]right: \scriptsize  $ \gamma $} ] (centerr) {};
          \draw[kernel1] (left) to
     node [sloped,below] {\small }     (root);
      \draw[kernel1] (center) to
     node [sloped,below] {\small }     (left);
     \draw[kernel1] (centerr) to
     node [sloped,below] {\small }     (root);
     \node at (-1.5,1.5) [fill=white,label={[label distance=0em]center: \scriptsize  $ a $} ] () {};
      \node at (-1.5,4.5) [fill=white,label={[label distance=0em]center: \scriptsize  $ b - \ell_1 $} ] () {};
       \node at (1.5,1.5) [fill=white,label={[label distance=-1em]right: \scriptsize  $ c - \ell_2 $} ] () {};
     \end{tikzpicture} \Big).
\end{equs}

Note that $ \tau $ has two vertices, therefore one only keeps in \eqref{explicit_star_2}  at most two blocks decomposition from $ \mathcal{K}^{*} \sigma $.

\subsection{Deformation of the insertion product}

\noindent The insertion of a tree $ \sigma $ into another $ \tau $ is defined by 
\begin{align*}
\sigma \blacktriangleright \tau = \sum_{v\in N_{\tau}} \sigma  \blacktriangleright_v  \tau
\end{align*} 
\label{insertion_i} where $ \sigma \blacktriangleright_v  \tau $ is given as \label{star_star}
\begin{equs} \label{identity_insertion}
 \sigma  \blacktriangleright_v  \tau  =  (P_{v}(\tau) \star \sigma) \rhd_v T_{v}(\tau)
\end{equs}
 where
\begin{itemize}
\item $ \star $ is the product associated to $ \rhd $ constructed via the Guin-Oudom procedure and the identification $ \CK $. 
\item $ P_{v}(\tau) $ corresponds to the subtree attached to $ v $ in $ \tau $.
\item $ T_v(\tau) $ is the trunk obtained by removing all the branches attached to $ v $.
\item The decoration $ \Labn_v $ at the node $ v $ is set to be zero in $ T_v(\tau) $ and it is put in the tree $ P_{v}(\tau) $.
\end{itemize}

\begin{remark} The identity \eqref{identity_insertion} is new in comparison to the literature on the insertion product, which was defined for the first time in \cite{ChaLiv}. The idea is to construct $ \blacktriangleright $ directly from $ \star $ which now becomes the central object. One can then transfer properties of the grafting operator to the insertion as the pre-Lie property.
\end{remark}
We provide an example of computation:
\begin{equs} 
 \sigma  =  \begin{tikzpicture}[scale=0.2,baseline=0.1cm]
        \node at (0,0)  [dot,label= {[label distance=-0.2em]below: \scriptsize  $ \omega $} ] (root) {};
         \node at (0,4)  [dot,label={[label distance=-0.2em]above: \scriptsize  $ \alpha $}] (right) {};
            \draw[kernel1] (right) to
     node [sloped,below] {\small }     (root); 
     \node at (0,2) [fill=white,label={[label distance=0em]center: \scriptsize  $ a $} ] () {};
     \end{tikzpicture},   \qquad \tau = 
    \begin{tikzpicture}[scale=0.2,baseline=0.1cm]
        \node at (0,0)  [dot,label= {[label distance=-0.2em]below: \scriptsize  $ \rho $} ] (root) {};
         \node at (-3,3)  [dot,label={[label distance=-0.2em]left: \scriptsize  $ \delta $} ] (left) {};
          \node at (0,7)  [dot,label={[label distance=-0.2em]above: \scriptsize  $ \gamma $} ] (center) {};
          \node at (-6,7)  [dot,label={[label distance=-0.2em]above: \scriptsize  $ \beta $} ] (centerr) {};
          \draw[kernel1] (left) to
     node [sloped,below] {\small }     (root);
      \draw[kernel1] (center) to
     node [sloped,below] {\small }     (left);
     \draw[kernel1] (centerr) to
     node [sloped,below] {\small }     (left);
     \node at (-1.5,1.5) [fill=white,label={[label distance=0em]center: \scriptsize  $ d $} ] () {};
      \node at (-1.5,5) [fill=white,label={[label distance=0em]center: \scriptsize  $ c $} ] () {};
       \node at (-4.5,5) [fill=white,label={[label distance=0em]center: \scriptsize  $ b $} ] () {};
     \end{tikzpicture} \quad P_v(\tau) = \begin{tikzpicture}[scale=0.2,baseline=0.1cm]
        \node at (0,0)  [dot,label= {[label distance=-0.2em]below: \scriptsize  $ \delta  $} ] (root) {};
         \node at (2,4)  [dot,label={[label distance=-0.2em]above: \scriptsize  $ \gamma $}] (right) {};
         \node at (-2,4)  [dot,label={[label distance=-0.2em]above: \scriptsize  $ \beta $} ] (left) {};
            \draw[kernel1] (right) to
     node [sloped,below] {\small }     (root); \draw[kernel1] (left) to
     node [sloped,below] {\small }     (root);
     \node at (-1,2) [fill=white,label={[label distance=0em]center: \scriptsize  $ b $} ] () {};
    \node at (1,2) [fill=white,label={[label distance=0em]center: \scriptsize  $ c $} ] () {};
     \end{tikzpicture} \qquad T_v(\tau) = \begin{tikzpicture}[scale=0.2,baseline=0.1cm]
        \node at (0,0)  [dot,label= {[label distance=-0.2em]below: \scriptsize  $ \rho $} ] (root) {};
         \node at (0,4)  [dot,label={[label distance=-0.2em]above: \scriptsize  $ $}] (right) {};
            \draw[kernel1] (right) to
     node [sloped,below] {\small }     (root); 
     \node at (0,2) [fill=white,label={[label distance=0em]center: \scriptsize  $ d $} ] () {};
     \end{tikzpicture},
\end{equs}
where $ v $ is the vertex of $ \tau $ decorated by $ \delta $. On gets $ P_v(\tau) \star_2 \sigma $ from \eqref{explicit_star_2}.
We proceed as in the previous section by deforming the product $ \blacktriangleright $. 
A good candidate is given by 
\begin{align*}
\sigma \, \wh \blacktriangleright \, \tau = \sum_{v\in N_{\tau}} \sigma \, \wh  \blacktriangleright_v \, \tau
\end{align*} 
\label{deformed_insertion_i}
where $ \sigma \, \wh \blacktriangleright_v \, \tau $ is defined as
\begin{equs} \label{deformed_insertion}
 \sigma \, \wh \blacktriangleright_v \, \tau  =  (P_{v}(\tau) \star_2 \sigma) \rhd_v  T_{v}(\tau).
\end{equs}
The deformation is performed by replacing $ \star $ by $ \star_2 $. The next proposition shows that as $ \blacktriangleright $, the product $ \wh \blacktriangleright $ is  a pre-Lie product. 

\begin{proposition}
The bilinear map $\wh \blacktriangleright$ is a pre-Lie product.
\end{proposition}
\begin{proof}
The proof relies on the fact that $   \star_2 $ is an associative product.
One has from \eqref{deformed_insertion}
\begin{equs} \label{iden_1}
\left( \eta \, \wh{\blacktriangleright}_{u} \, \sigma \right) \, \wh{\blacktriangleright}_{v} \, \tau = \left(P_{v}(\tau)  \star_2 \left(\eta \, \wh{\blacktriangleright}_{u} \, \sigma  \right)  \right)  \rhd_v   T_{v}(\tau).
\end{equs}
Then
\begin{align*}
P_{v}(\tau) \star_2  \left(\eta \, \wh{\blacktriangleright}_{u} \, \sigma  \right)    =  P_{v}(\tau) \, \star_2  \left( (P_{u}(\sigma) \star_2 \eta)  \rhd_u  T_{u}(\sigma) \right).
\end{align*}
On the other hand
\begin{equs} \label{iden_2}
 \eta \, \wh{\blacktriangleright}_{u} \left(\sigma \, \wh{\blacktriangleright}_{v} \, \tau \right) = (P_{u} \left(\sigma  \, \wh{\blacktriangleright}_{v} \tau \right)  \star_2  \eta  )  \rhd_u  T_{u}\left(\sigma  \, \wh{\blacktriangleright}_{v} \, \tau \right).
\end{equs}
Then if we assume that $ u \in N_{\sigma} $,
\begin{align*}
P_{u} \left(\sigma  \, \wh{\blacktriangleright}_{v}\, \tau \right)  & = \sum_{P_{v,1}(\tau)P_{v,2}(\tau) = P_{v}(\tau)} (P_{v,1}(\tau)  \star_2 P_{u}(\sigma)), \\
T_{u} \left(\sigma  \, \wh{\blacktriangleright}_{v} \, \tau \right)& = T_u \left(  P_{v}(\tau) \star_2 \sigma \right) \rhd_v T_{v}(\tau)
\\ & =\sum_{P_{v,1}(\tau)P_{v,2}(\tau) = P_{v}(\tau)}  \left(  P_{v,2}(\tau) \star_2^u T_u(\sigma)  \right) \rhd_v  T_{v}(\tau),
\end{align*}
where $ \star_2^u $ behaves like $ \star_2 $ except that insertions on the node $ u $ are not allowed. The product $ P_{v,1}(\tau)P_{v,2}(\tau) $ under the sum has to be understood as the tree product.
This yields
\begin{equs}
 \eta \, \wh{\blacktriangleright}_{u} \left(\sigma  \, \wh{\blacktriangleright}_{v} \, \tau \right) & = \sum_{P_{v,1}(\tau)P_{v,2}(\tau) = P_{v}(\tau)} \left( (P_{v,1}(\tau) \star_2 P_{u}(\sigma))  \star_2 \eta \right) 
 \\ & \rhd_u \left( \left(  P_{v,2}(\tau) \star_2^u T_u(\sigma)  \right) \rhd_v  T_{v}(\tau) \right).
\end{equs}
In order to conclude, we use the associativity of $ \star_2 $
\begin{equs}
\left( P_{v,1}(\tau)  \star_2 P_{u}(\sigma) \right)  \star_2 \eta  =  P_{v,1}(\tau)  \star_2 \left( P_{u}(\sigma)  \star_2 \eta \right)
\end{equs}
to match \eqref{iden_1} with \eqref{iden_2}. If $ u \notin N_{\sigma} $ then the insertions are disjoint and one has:
\begin{equs}
 \eta \, \wh{\blacktriangleright}_{u} \left(\sigma  \, \wh{\blacktriangleright}_{v} \, \tau \right) =  \sigma \, \wh{\blacktriangleright}_{v} \left(\eta  \, \wh{\blacktriangleright}_{u} \, \tau \right).
\end{equs}
\end{proof}

\begin{remark} One can try to apply $ \Theta $ and see how $ \blacktriangleright $ is deformed. One has
\begin{equs}
 \Theta \left( \sigma   \blacktriangleright_v  \tau \right)  =  \Theta(P_{v}(\tau)  \star_2 \sigma) \rhd_v \Theta( T_{v}(\tau)).
\end{equs}
The product $  \rhd_v  $ is not deformed because $ \Theta( T_{v}(\tau)) $ has zero decoration at the node $ v $.
Then by using 
\begin{equs}
\Theta( T_{v}(\tau)) =  T_{v}(\Theta(\tau)), \quad \Theta( P_{v}(\tau)) =  P_{v}(\Theta(\tau))
\end{equs}
we obtain \label{tilde_star_2}
\begin{equs}
 \Theta \left( \sigma   \blacktriangleright_v  \tau \right)  = \left(  P_{v}(\Theta(\tau))  \, \tilde{\star}_2 \, \Theta(\sigma) \right) \rhd_v  T_{v}(\Theta(\tau))
& =  \Theta \left( \sigma \right)   \widetilde{\blacktriangleright}_v  \Theta \left( \tau \right) 
\end{equs}
where $ \tilde{\star}_2 $ is the product obtained from $ \widetilde{\rhd} $ and
\begin{equs}
\sigma \, \widetilde{\blacktriangleright}_v \, \tau := (P_{v}(\tau) \, \tilde{\star}_2 \, \sigma) \rhd_v T_{v}(\tau).
\end{equs}
\label{theta_insertion}
The product $  \wh{\blacktriangleright} $ is not directly obtained from a deformation of $ \blacktriangleright $ by $ \Theta $. This is due to the construction of $ \star_2 $ from $ \star $ which is also not only an application of the isomorphism $ \Theta $.
\end{remark}

\label{star_1}
We denote by $ \star_1 $ the product obtained by applying the Guin-Oudom procedure on the product $ \wh{\blacktriangleright} $. We want to identify the dual map of $ \star_1 $. A good candidate is  $ \Delta_1 :  S(\CT_E^V) \rightarrow S(\CT_E^V) \hattimes S(\CT_E^V) $ recursively defined as:
 \begin{equs} \label{recursive_def_delta1}
  \Delta_1 =  
\mathcal{M}^{(13)(2)} \left(   \Delta_{\circ}  \hattimes \id \right) \Delta_2
 \end{equs} \label{delta_1_1} \label{delta_0_0}
 where $  \Delta_{\circ} : \CT_E^V \rightarrow S(\CT_E^V) \hattimes \CT_E^V   $ is defined multiplicatively for the tree product by
 \begin{equs} \label{Deltacirc}
 \Delta_{\circ} X^k = \one \hattimes X^k, \quad \Delta_{\circ} \CI_{a}(\tau) = \left( \id \hattimes \mathcal{I}_a \right) \Delta_1 \tau.
 \end{equs}
\begin{remark}
The recursive formulation \eqref{recursive_def_delta1} and \eqref{Deltacirc} has been introduced in \cite{BR18}. It is a decomposition of the extraction-contration coproduct into two different actions: one is the extraction at the root given by $ \Delta_2 $, the other is the extraction outside the root which is iterated by the map $ \Delta_{\circ} $.
This recursive construction can be generalised for designing renormalisation maps. One just needs to replace $ \Delta_2  $ by a linear map $ R : \CT_E^V \rightarrow \CT_E^V $ which has some compatibility with $ \Delta_2 $ that is:
\begin{equs}
\left( \id \hattimes R\right) \Delta_2 = \Delta_2 R.
\end{equs}
Then, one can define a renormalisation map 
\begin{equs}  \label{def::M}
 M = M^{\circ} R, \quad M^{\circ} X^k \tau = X^k M^{\circ}  \tau, \quad M^{\circ} \CI_{a}(\tau) = \CI_a(M \tau),
\end{equs}
where $ M^{\circ} $ is multiplicative for the tree product.
It just corresponds to replace the extraction by another procedure which is iterated deeper in the tree.
These new renormalisation maps can be used either in Regularity Structures (see \cite{BR18}) or Rough Paths (see \cite{Br1111}). This approach has also been used in \cite{CMW} under the name of local product.
 \end{remark}
For proving the next theorem, we will use  a recursive formula for the deformed insertion $ \wh{\blacktriangleright} $ given for   $  \hat \tau_1 = \widetilde{\prod}_{i \in I} \tau_{1i} \in S(\CT_E^V)$, $  \tau_2,  \tau_3 =X^k  \prod_{j=1}^m \CI_{a_j}(\tau_{3j}) \in \CT_E^V $  by
\begin{equs}
\hat \tau_1 \, \wh{\blacktriangleright} \,  \tau_2  =  \hat \tau_1 \, \wh{\blacktriangleright}^{\nonroot} \, \tau_2 + \sum_{j \in I} \widetilde{\prod}_{i \neq j}  \left( \tau_{1i} \, \wh{\blacktriangleright}^{\nonroot} \, \tau_2 \right) \star_2 \tau_{1j}  \\
 \tau_2 \,  \wh{\blacktriangleright}^{\nonroot} \, X^k  \prod_{j=1}^m \CI_{a_j}(\tau_{3j})  = \sum_{i =1 }^m  X^k \CI_{a_i}(  \tau_2 \,  \wh{\blacktriangleright} \, \tau_{3i}) \prod_{j \neq i} \CI_{a_j}(\tau_{3j}).
\end{equs}
 
\begin{theorem} \label{deformed_E_C}
The product $ \star_1 $ is the dual of  $ \Delta_1 $. One has
\begin{align*}
\langle \hat \tau_1 \, \wh{\blacktriangleright} \,  \tau_2 ,  \tau_3 \rangle = \langle \hat \tau_1 \hattimes   \tau_2  ,  \Delta_1  \tau_3 \rangle, \quad \langle \hat \tau_1 \, \wh{\blacktriangleright}^{\nonroot} \,  \tau_2 ,  \tau_3 \rangle = \langle \hat \tau_1 \hattimes  \tau_2 ,  \Delta_{\circ}  \tau_3 \rangle
\end{align*}
where $ \hat \tau_1 \in S(\CT_E^V) $, $  \tau_2,  \tau_3 \in \CT_E^V $.
\end{theorem}

\begin{proof}
We proceed by induction on the size of the trees. From the identity \eqref{recursive_def_delta1}, one has for $ \hat \tau_1 = \widetilde{\prod}_{i \in I}  \tau_{1i} \in S(\CT_E^V)$, $  \tau_2,  \tau_3 \in \CT_E^V $, 
\begin{equs}
\langle \hat \tau_1  \hattimes  \tau_2  ,  \Delta_1  \tau_3\rangle = 
\langle  \hat \tau_1   \hattimes   \tau_2  ,  \mathcal{M}^{(13)(2)} \left(   \Delta_{\circ}  \hattimes \id \right) \Delta_2  \tau_3 \rangle.
\end{equs}
We notice that the dual map $ (\mathcal{M}^{(13)(2)})^{*} $ of $ \mathcal{M}^{(13)(2)} $ is given by:
\begin{equs}
(\mathcal{M}^{(13)(2)})^{*} (\hat \tau_1 \hattimes   \tau_2)  =  \hat \tau_1 \hattimes \tau_2   \hattimes \one  + \sum_{j \in I} \left( \widetilde{\prod}_{i \neq j}  \tau_{1i} \hattimes  \tau_2 \hattimes \tau_{1j} \right).
\end{equs}
Therefore by using the induction hypothesis, one gets 
\begin{equs}
\langle & \hat \tau_1 \hattimes   \tau_2  ,  \mathcal{M}^{(13)(2)} \left(   \Delta_{\circ}  \hattimes \id \right) \Delta_2  \tau_3 \rangle \\
& = \langle  \hat \tau_1 \hattimes \tau_2  \hattimes \one  + \sum_{j \in I} \left( \widetilde{\prod}_{i \neq j}  \tau_{1i}  \hattimes \tau_2 \hattimes \tau_{1j} \right) ,   \left(   \Delta_{\circ}  \hattimes \id \right) \Delta_2  \tau_3 \rangle
\\ & = \langle  \left( \hat \tau_1 \wh{\blacktriangleright}^{\nonroot} \tau_2 \right)   \hattimes \one  + \sum_{j \in I} \left( \widetilde{\prod}_{i \neq j} \left( \tau_{1i} \wh{\blacktriangleright}^{\nonroot} \tau_2 \right) \hattimes \tau_{1j} \right) ,   \Delta_2  \tau_3 \rangle
\\
& = \Big\langle \hat \tau_1 \, \wh{\blacktriangleright}^{\nonroot} \tau_2 + \sum_{j\in J}  \widetilde{\prod}_{i \neq j}  \left( \tau_{1i} \, \wh{\blacktriangleright}^{\nonroot} \tau_2 \right)  \star_2  \tau_{1j},  \tau_3\Big\rangle.
\end{equs}
It remains to match the part given by $ \Delta_{\circ} \bar \tau_3 $.
For $  \tau_3 =X^k  \prod_{j=1}^m \CI_{a_j}(\tau_{3j})  $, we get
\begin{equs} \label{rec_tau3}
\Delta_{\circ}   \tau_3 = (\one \hattimes X^k) \prod_{j=1}^m \left( \id \hattimes \CI_{a_j} \right) \Delta_1 \tau_{3j}.
\end{equs}
On the other hand, 
\begin{equs} \label{rec_tau3b}
 \hat \tau_1 \, & \wh{\blacktriangleright}^{\nonroot} \, X^k  \prod_{j=1}^m \CI_{a_j}(\tau_{3j})  \\ & = \sum_{I = \sqcup_{j=1}^m  J_j} X^k \prod_{j = 1}^m\CI_{a_j}\left( \widetilde{\prod}_{i \in J_j}  \tau_{1i}   \,  \wh{\blacktriangleright}^{\nonroot} \, \tau_{3j} \right).
\end{equs}
One can match the two identities \eqref{rec_tau3} and \eqref{rec_tau3b} in order to conclude.
\end{proof}

\subsection{Cointeraction between deformed products}

The classical cointeraction between grafting and insertion states that for $ \tau_1, \tau_2 \in \CT_E^V $ and $ \tau \in S(\CT_E^V) $:
\begin{equs} \label{preLie_cointeraction}
 \sum_{(\tau)}\left( \tau^{(1)} \blacktriangleright   \tau_{1} \right) \curvearrowright^a   \left( \tau^{(2)} \blacktriangleright   \tau_2 \right) = \tau  \blacktriangleright  \left( \tau_1 \curvearrowright^a \tau_2 \right)
\end{equs}
where Sweedler's notation corresponds to the coproduct $ \Delta $ introduced in \eqref{delta_trees}.
This identity has been enounced at the Hopf algebraic level \cite{CEM} and appears in B-series \cite{CHV10}. It describes the interaction between multiplication and composition of B-series.
The pre-Lie cointeraction \eqref{preLie_cointeraction} is extracted from \cite{MS}. One can also write this identity for the pre-Lie product $ \curvearrowright    $:
\begin{equs} \label{grafting}
\sum_{(\tau)}\left( \tau^{(1)} \blacktriangleright^{\nonroot}   \tau_{1} \right) \curvearrowright   \left( \tau^{(2)} \blacktriangleright   \tau_2 \right) = \tau  \blacktriangleright  \left( \tau_1 \curvearrowright \tau_2 \right)
\end{equs}
where this time $ \tau_1 \in \CP_E^V $ and $ \blacktriangleright^{\nonroot}   $ is the insertion everywhere except at the root. This identity is not true  if one replaces $ \curvearrowright   $ by the plugging product $  \rhd $. But if we use the product $ \star $ with the identification $ \CK $, one gets:
\begin{equs} \label{plugging}
\sum_{(\tau)}\left( \tau^{(1)} \blacktriangleright^{\nonroot}   \tau_{1} \right) \star \left( \tau^{(2)} \blacktriangleright   \tau_2 \right) = \tau  \blacktriangleright  \left( \tau_1 \star \tau_2 \right).
\end{equs}
We want to check that these indentities are preserved by the deformation. Applying the linear isomorphism $ \Theta $ on \eqref{preLie_cointeraction}, we get
\begin{equs}
\Theta \left(  \tau  \blacktriangleright  \left( \tau_1 \curvearrowright^a \tau_2 \right)  \right) =  \Theta(\tau)  \widetilde{\blacktriangleright}  \left( \Theta(\tau_1) \wh{\curvearrowright}^a \Theta(\tau_2) \right),
\end{equs}
which is not the desired identity. Indeed, $ \Theta $ makes appear $ \widetilde{\blacktriangleright} $ instead of $ \wh{\blacktriangleright} $.
One cannot get a deformation of the cointeraction using $ \Theta $ but we are able to get the following statement
\begin{theorem} \label{theo_cointeraction} One has  for $ \tau_1, \tau_2 \in \CT_E^V $ and $ \tau \in S(\CT_E^V) $
\begin{equs} \label{dgrafting}
\sum_{(\tau)}\left( \tau^{(1)} \, \wh{\blacktriangleright}^{\nonroot}   \, \tau_{1} \right) \wh{\curvearrowright}   \, \left( \tau^{(2)} \, \wh{\blacktriangleright} \,  \tau_2 \right) = \tau \, \wh{\blacktriangleright} \, \left( \tau_1 \, \wh{\curvearrowright} \,\tau_2 \, \right),
\end{equs}
\begin{equs} \label{dplugging}
\sum_{(\tau)}\left( \tau^{(1)} \, \wh{\blacktriangleright}^{\nonroot}  \, \tau_{1} \right) \star_2 \left( \tau^{(2)} \, \wh{\blacktriangleright}  \, \tau_2 \right) = \tau \, \wh{\blacktriangleright} \, \left( \tau_1 \star_2 \tau_2 \right).
\end{equs}
\end{theorem}
\begin{proof} We perform the proof using the identity \eqref{deformed_insertion} and with $ \tau \in \CT_E^V$. It turns out that \eqref{dplugging} can be rewritten as 
\begin{equs}
\left( \tau \, \wh{\blacktriangleright}^{\nonroot}  \, \tau_{1} \right) \star_2 \tau_2 +  \tau_{1}  \star_2 \left( \tau \, \wh{\blacktriangleright}  \, \tau_2 \right) = \tau \, \wh{\blacktriangleright} \, \left( \tau_1 \star_2 \tau_2 \right).
\end{equs}
Then
\begin{equs}
\tau \, \wh{\blacktriangleright} \, \left( \tau_1 \star_2 \tau_2 \right) = \sum_{v \in N_{\tau_1} \setminus \lbrace\rho_{\tau_1} \rbrace}\tau \, \wh{\blacktriangleright}_v \, \left( \tau_1 \star_2 \tau_2 \right) + \sum_{v \in N_{\tau_2} }\tau \, \wh{\blacktriangleright}_v \, \left( \tau_1 \star_2 \tau_2 \right).
\end{equs}
It is straightforward to check that
\begin{equs}
\sum_{v \in N_{\tau_1} \setminus \lbrace\rho_{\tau_1} \rbrace}\tau \, \wh{\blacktriangleright}_v \, \left( \tau_1 \star_2 \tau_2 \right)  &  = \left( \sum_{v \in N_{\tau_1} \setminus \lbrace\rho_{\tau_1} \rbrace}\tau \, \wh{\blacktriangleright}_v \,  \tau_1 \right) \star_2 \tau_2 
\\ & = \left( \tau \, \wh{\blacktriangleright}^{\nonroot}  \, \tau_{1} \right) \star_2 \tau_2. 
\end{equs}
For $ v \in N_{\tau_2} $, one has
\begin{equs}
\tau \, \wh{\blacktriangleright}_v \, \left( \tau_1 \star_2 \tau_2 \right) & = (P_{v}\left( \tau_1 \star_2 \tau_2 \right) \star_2 \tau) \rhd_v  T_{v}\left( \tau_1 \star_2 \tau_2 \right)
\\ & = \sum_{(\tau_1)}\left( \tau_1^{(1)} \star_2 P_v(\tau_2) \right) \star_2 \tau) \rhd_v  \left( \tau^{(2)}_1 \star^{v}_2 T_v(\tau_2) \right)
\\ & = \tau_1 \star_2 \Big( \left( P_v(\tau_2) \right) \star_2 \tau) \rhd_v  \left( T_v(\tau_2) \right) \Big)
\\ & = \tau_{1}  \star_2 \left( \tau \, \wh{\blacktriangleright}_{v}  \, \tau_2 \right)
\end{equs}
where Sweedler's notations $ \tau_1^{(1)}, \tau_1^{(2)} $ correspond to the coproduct $ \Delta $ introduced in \eqref{delta_trees}  and we have used the identity:
\begin{equs}
\tau_1 \star_2 \left( \tau_2 \, \rhd_v \,  \tau_3 \right) = \sum_{(\tau_1)} \left( \tau_1^{(1)} \star_2 \tau_2\right) \rhd_v \left( \tau_1^{(2)} \star_2^{v} \tau_3\right)
\end{equs}
with  $ \tau_i \in \CT_E^V $, $ v \in N_{\tau_3} $ and $ \tau_3 $ has zero node decoration at the node $ v $. The proof for \eqref{dgrafting} follows the same lines.
\end{proof}

\begin{remark} The proof of Theorem~\ref{theo_cointeraction} does not depend on the deformation. It is the same proof if one replaces $ \star_2 $ by $ \star $ and $ \wh{\blacktriangleright} $ by $ \blacktriangleright $. The proof exploits the fact that $ \wh{\blacktriangleright} $ (resp. $ \blacktriangleright $) is defined in terms of $ \star_2  $ (resp. $ \star $) and  $ \rhd $.
\end{remark}

\begin{remark} The cointeraction described in Theorem~\ref{theo_cointeraction} has a counterpart for the map $ M $ defined in \eqref{def::M}:
\begin{equs}
\left( M^{\circ} \hattimes  M  \right) \Delta_2 = \Delta_2  M.
\end{equs}
This formulation appears in \cite{BR18} for defining more general renormalisation maps that could interact with the model given by the theory of Regularity Structures.
\end{remark}

\section{Applications}

\label{section::4}

In this section, we will see that the various deformed products introduced in the previous section have many applications among (S)PDEs. The first application is about a numerical scheme for dispersive PDEs. The second application focuses on the main coproducts which are used within singular SPDEs. The deformed grafting insertion has its origin from that field.

\subsection{Numerical Analysis}
\label{subsection::4.1}
In \cite{BS}, the authors derived a resonance-based numerical schemes for equations of the type
\begin{equation}\label{dis}
\begin{aligned}
	& i \partial_t u(t,x) 
	+   \mathcal{L}\left(\nabla, \frac{1}{\varepsilon}\right) u(t,x) 
	=\vert \nabla\vert^\alpha p\left(u(t,x), \overline u(t,x)\right)\\
	& u(0,x) = v(x), \quad (t,x) \in \R_+ \times \T^{d}
\end{aligned}
\end{equation}
where $ \CL$ is a differential operator and $ p $ is a monomial of the form $ p(u,\bar u) = u^{N} \bar u^{M} $.
One can rewrite Duhamel's formulation of \eqref{dis} in Fourier space:
\begin{equation}\label{duhLin_it_Fourier}
	\hat u_k(t) = e^{ it  P(k)} \hat v_k  - i  
	\vert \nabla\vert^{\alpha} (k) \, e^{ it  P(k)} \int_0^t e^{ -i\xi  P(k) } p_k(u(\xi),\bar u(\xi)) d\xi 
\end{equation}
where $P(k)$ denotes the differential operator $\mathcal{L}$ in Fourier space  
\begin{equs}
	P(k) = \mathcal{L}\left(\nabla,\frac{1}{\varepsilon}\right)(k), \quad
	p_k(u(t),\bar u(t)) = \sum_{k =\sum_{i} k_i - \sum_j \bar k_j} 
	\prod_{i=1}^N \hat u_{k_i}(t) \prod_{j=1}^M \bar{\hat{u}}_{\bar k_j}(t).
\end{equs}
The numerical scheme iterates this formulation and produces integrals that can be viewed as a character $ \Pi $ defined on a set of planted trees $ \CP_E^V $.

The decorations encode the $ k $, $ e^{itP} $ and polynomials in $ \xi $.
Decorated trees are of the form:
\begin{equs}
 \CI_{(o,m)}(\lambda_{k}^{\ell} T), \quad T  \in S(\CP_E^V)
\end{equs}
where $ V = \Z^{d} \times \N $ and $ E = \Lab \times \N $. The finite set $ \Lab $ collects different types and allows us to know if one has $ e^{itP} $ or $ \int_0^{t} e^{- i \xi P} \cdots d\xi $ (integrals in time correspond to a subset $ \Lab_+ \subset \Lab $). The second decoration on the edges given by  $ \N $ is associated to derivatives in time that one has to take in some Taylor expansions approximating the previous operators. The decoration on the node denoted by  $ \lambda^{\ell}_k $ has to be understood as an inert decoration $ k \in \Z^{d} $ for the frequency $ k $ and a decoration $ \ell \in \N $ which corresponds to the polynomial in time coming from approximations.
Then, one defines maps  $  \bar{\Delta}_{\text{\tiny{NA}}} : \CT_E^{V} \rightarrow  \CT_E^{V} \hattimes S(\CP_E^V) $ and  $  \Delta_{\text{\tiny{NA}}} : S(\CP_E^{V}) \rightarrow  S(\CP_E^{V}) \hattimes S(\CP_E^V) $ such that
\begin{equation} \label{def_deltas_NA}
\begin{aligned}
\bar{\Delta}_{\text{\tiny{NA}}} \lambda^{\ell}& = \lambda^{\ell} \hattimes \one ,\\
\bar{\Delta}_{\text{\tiny{NA}}} \CI_{(o,m)}(\lambda_{k}^{\ell} T)  & = \left( \CI_{(o,m)}( \lambda_{k}^{\ell} \cdot)  \hattimes \id  \right) \bar{\Delta}_{\text{\tiny{NA}}} T + \sum_{n \in \N} \frac{\lambda^{n}}{n!} \hattimes  \CI_{(o,m+n)}(\lambda_{k}^{\ell} T), \\
\Delta_{\text{\tiny{NA}}} \CI_{(o,m)}(\lambda_{k}^{\ell} T)  & = \left( \CI_{(o,m)}( \lambda_{k}^{\ell} \cdot)  \hattimes \id  \right) \bar{\Delta}_{\text{\tiny{NA}}} T +  \one \hattimes  \CI_{(o,m)}(\lambda_{k}^{\ell} T).
\end{aligned}
\end{equation}
The identities \eqref{def_deltas_NA} can be obtained by applying $ \CM^{(2)(1)} $ to \eqref{recur_def_deltaDCK}:
\begin{equs} \label{ident_NA}
\bar{\Delta}_{\text{\tiny{NA}}} = \CM^{(2)(1)} \bar{\Delta}_{\text{\tiny{DCK}}}, \quad \Delta_{\text{\tiny{NA}}} = \CM^{(2)(1)} \Delta_{\text{\tiny{DCK}}}.
\end{equs}
As a consequence of \eqref{ident_NA} and Theorem~\ref{DCK_def}, we obtain one of our main results
\begin{theorem} \label{Deformation_NA} The map $ \Delta_{\text{\tiny{NA}}} $ is a deformation of the Butcher-Connes-Kreimer coproduct. It is obtained by applying the Guin-Oudom procedure on $ \wh{\curvearrowright} $ which is a deformation of $ \curvearrowright $ given  by the map $ \Theta $.
\end{theorem}
The maps used in \cite{BS} are slightly different from \eqref{def_deltas_NA}. Indeed, for numerical schemes, we define the order $ r $ of the scheme a priori. This produces projections such that the sum over $ n  $ is finite. In the end we only keep trees of a certain size determined by the number of integrations (edges decorated by $ \Lab_+ $) and polynomials (node decorations).
There is another major simplification. In the scheme described in \cite{BS} there is no a priori derivatives on the edges. Therefore, $ m $ is zero inside the tree. The deformation \eqref{deformation_preLie} reads in this context for $ \tau_1, \tau_2 \in \CT_E^V $
\begin{equs} \label{deformation_preLie_NA}
\begin{aligned}
\tau_1 \, \widehat{\curvearrowright}^{(o,m)} \tau_2 & =   \sum_{v\in N_{\tau_2}}\sum_{\ell\in\N}{\Labn_v\choose \ell} \tau_1 \curvearrowright^{(o,m-\ell)}(\uparrow_v^{-\ell} \tau_2)
\\ & = \sum_{v\in N_{\tau_2}} {\Labn_v\choose m}\tau_1 \curvearrowright_v^{(o,0)}(\uparrow_v^{-m} \tau_2).
\end{aligned}
\end{equs}
One other simplification in the scheme comes from the way one approximates the various operators appearing within the iterated integrals. For $ \int_{0}^{t} e^{-i \xi P(k)} \cdots d\xi $ one uses a well-chosen Taylor expansion depending on the resonances. One Taylor-expands the lower part in the resonances and integrates exactly the dominant part. This will guarantee a minimisation of the local error.
For $  e^{ it  P(k)} $, there is no approximation. Therefore, at the algebraic level we will restrain the grafting in \eqref{deformation_preLie_NA} to an edge where $ o \in \Lab_+ $. 

The scheme is described by a family of characters $ (\Pi^n)_{n \in \N} $ on $ \CP_V^E $ where $ n $ corresponds to the regularity assumed a priori on the solution. The map $ \bar{\Delta}_{\text{\tiny{NA}}} $ is used for the local error analysis. Indeed, via a Birkhoff factorisation one constructs another character $ \hat \Pi^n $ such that
\begin{equs} \label{factorisation}
\hat \Pi^n = \left( \Pi^n \otimes (\CQ \circ \Pi^n \CA \cdot)(0) \right) \bar{\Delta}_{\text{\tiny{NA}}},
\end{equs}
where $ \CA $ is the antipode associated to $ \Delta_{\text{\tiny{NA}}} $ and $ \CQ $ is a projector which allows to single out oscillations.
Then one can derive one of the main theorems for the local error analysis of the scheme see \cite[Theorem 3.11]{BS}.
This Birkhoff factorisation is put  into perspective in \cite{BE} and it is compared with the usual one with the Butcher-Connes-Kreimer coproduct.

\subsection{Regularity Structures}
\label{subsection::4.2}

Trees  with both decorations on the edges and vertices were first introduced for singular SPDEs in \cite{BHZ}. They are efficient for encoding  stochastic iterated integrals coming from a mild formulation.
Indeed, let consider a system of SPDEs :
\begin{equs} \label{SPDEs}
	\partial_t u_{\Labhom} - \CL_{\Labhom} u_{\Labhom} = F_{\Labhom}(u,\nabla u, \ldots, \xi), \quad  (t,z) \in \R_+ \times \R^{d}, \, \Labhom \in \Lab_+
\end{equs}
where  $ \xi = (\xi_{\Labhom})_{\Labhom \in \Lab_-} $ are space-time noises, $ u = (u_{\Labhom})_{\Labhom \in \Lab_+} $ and  $ (\CL_{\Labhom})_{\Labhom \in \Lab_+} $ is a family of differential operators. The mild formulation of \eqref{SPDEs} is given by:
\begin{equs} \label{mild}
	u_{\Labhom} = K_{\Labhom} \ast \left(  F_{\Labhom}(u, \nabla u, \xi) \right),
\end{equs}
where $ K_{\Labhom} $ is the kernel associated to $ \CL_{\Labhom} $ and $ \ast  $ denotes space-time convolution. The main idea of Regularity Structures introduced by Martin Hairer in \cite{reg} is to iterate this formulation locally when singular noises $ \xi $ are replaced by their regularisation $ \xi_{\eps} $. Non-linearities $ F $ are Taylor-expanded around a base point $ x \in \R^{d+1} $ and one can construct recentered iterated integrals around that point.
These integrals give a local description
of the solution of \eqref{SPDEs}:
\begin{equs} \label{B-series}
u_{\Labhom}(y) = \sum_{\tau \in \bar \CT} \frac{\Upsilon^{F}_{\Labhom}[\tau](x)}{S(\tau)} \left( 
\Pi_x \tau \right)(y) + R_{\Labhom}(x,y),
\end{equs}
where $\Labhom \in \Lab_+ $, $ \bar \CT_{\Labhom} \subset \CT_E^V$ are trees generated by \eqref{mild} with $ V = \N^{d+1} $ and $ E = \Lab \times \N^{d+1} $ ($ \Lab = \Lab_+ \sqcup \Lab_- $), $ S(\tau) $ is the symmetry factor associated to $ \tau $, $ \Upsilon^{F}_{\Labhom}[\tau] $ are elementary differentials depending on $ u $ and its derivatives. The form of \eqref{B-series} is very similar to the one for B-series in numerical analysis.
The map $ \Pi_x : \CT_E^V \rightarrow \mathcal{C}^{\infty}  $ is defined as a character for the tree product such that
$ \left( \Pi_x \tau \right) \lesssim |x-y|^{|\tau|_{\s}} $ ; $ |\tau|_{\s} $ is the degree associated to $ \tau $ computed from the regularity of the noises and the  Schauder estimates given by the kernels $ K_{\Labhom} $ (see for more details \cite[Def. 5.12]{BHZ}). The remainder $ R_{\Labhom}(x,y)\lesssim |x-y|^{\alpha}   $ is better behaved in comparison to $ \left( \Pi_x \tau \right)(y) , \tau \in \bar \CT_{\Labhom}$ ($\alpha > |\tau|_{\s}$ ). The scaling $ \s $ depends on the differential operator $ \mathcal{L}_{\Labhom} $: for the Laplacian, it is $ (2,1,\ldots,1) $ where time counts double in comparison to the other spatial components.
The deformed pre-Lie product has been introduced in \cite{BCCH} for proving a morphism property of $ \Upsilon^{F}[\tau] $. Indeed, from \cite[Cor. 4.15]{BCCH} one has
\begin{equs}
\Upsilon^{F}_{\Labhom}[\tau_1 \wh{\curvearrowright}^{(o,m)} \tau_2 ] 
= \Upsilon^{F}_o[\tau_1] D_{\nabla^{m} u_{o}} \Upsilon^{F}_{\Labhom} [ \tau_2 ] 
\end{equs}
where $ D_{\nabla^{m} u_{o}} $ is the derivative in the variable $ \nabla^{m} u_{o} $.
This morphism property is crucial for understanding how the renormalisation  acts on the equation. The deformation $ \wh{\curvearrowright}^{(o,m)} $ has been studied in \cite{BCCH}. A universal property is obtained in \cite[Prop. 4.21]{BCCH} based on an extension of the proof given by Chapoton-Livernet in \cite{ChaLiv}. 
What is clearly missing in \cite{BCCH} is the point of view of the deformation and the construction of the coproduct for $ \Pi_x $ within Regularity Structures. Indeed, an adjoint has been identified in \cite[Def. 4.11]{BCCH} but it does not take into account the extraction of polynomials. The original recentering coproduct $ \Delta_{\text{\tiny{RC}}} : \CT_E^{V} \rightarrow  \CT_E^{V} \hattimes \CT_E^V $ is given in \cite{BHZ} (previously in \cite{reg} but with projections) by   
\begin{equation} \label{def_deltas_RC}
\begin{aligned}
\Delta_{\text{\tiny{RC}}} X & = X \hattimes \one + \one \hattimes X, \\
\Delta_{\text{\tiny{RC}}} \CI_{(o,m)}(\tau)  & = \left( \CI_{(o,m)} \hattimes \id  \right) \Delta_{\text{\tiny{RC}}} \tau + \sum_{n \in \N^{d+1}} \frac{X^{n}}{n!} \hattimes  \CI_{(o,m+n)}(\tau) .
\end{aligned}
\end{equation}
The map $ \Delta_{\text{\tiny{RC}}}  $ can be obtained from $ \Delta_2 $ by
\begin{equs} \label{ident_RC}
\Delta_{\text{\tiny{RC}}} = \CM^{(2)(1)} \Delta_2.
\end{equs}

\noindent As a consequence of \eqref{ident_RC} and Theorem~\ref{DP_deformation}, we obtain one of our main results
\begin{theorem} \label{Deformation_NA} The map $ \Delta_{\text{\tiny{RC}}} $ is a deformation of the Butcher-Connes-Kreimer coproduct. It is obtained by applying the Guin-Oudom procedure on $ \wh{\rhd} $ which is a deformation of $ \rhd $.
\end{theorem}

As for numerical analysis, the sum over $ n $ is truncated in \eqref{def_deltas_RC}. One keeps only  trees with branches connected to the root of positive degree: $ |\CI_{(o,m+n)}(\tau)|_{\s} > 0 $. The idea is that one has to subtract a Taylor expansion up to the regularity given by the degree. We introduce a new space $ \CT_+ $ which corresponds to trees having all the branches outgoing the root to be of positive degree:
\begin{equs}
\label{Tplus}
	 \mathcal{T}_{+} 
 	:=   \Big\lbrace X^{k} \prod_{i=1}^{n} \CI_{(o_i,p_i)}(\tau_i), \, 
	| \CI_{(o_i,p_i)}(\tau_i) |_{\s} > 0, \tau_i \in \CT \Big\rbrace, 
\end{equs}
where $ \CT \subset \CT_E^V $ is the set of trees appearing in the right hand side of \eqref{SPDEs}. These trees are generated by a subcritical and normal complete rule $ R $ coming from the non-linearity in \eqref{SPDEs}. We  refer the reader to  \cite[Section 5]{BHZ}, where those rules have been detailed.
One can then define from $ \Delta_{\text{\tiny{RC}}} $ a coaction $ \Deltap : \CT 
\rightarrow \CT \otimes \CT_+ $ obtained from $ \eqref{def_deltas_RC} $ by performing several projections. They allow us to move from $ \CT_E^V $ to $ \CT $ and $ \CT_+ $. Then, one of the central results in \cite{reg,BHZ} is:
\begin{equs} \label{def_Pi_x}
\Pi_x = \left( \Pi \otimes (\Pi \tilde{\CA}_+ \cdot)(x) \right) \Deltap,
\end{equs}
where $ \tilde{\CA}_+  $ is a twisted antipode from $ \CT_+ $ into $ \hat \CT_+ $ and $ \Pi $ is a character on  trees computing  iterated integrals without any recentering. The space $ \hat \CT_+ $ is similar to $ \CT_+ $ except that one does not project according to the degree in \eqref{Tplus}. The identity \eqref{def_Pi_x} has been interpreted as a Bogoliubov recursion in \cite{BE}.

In many cases when one wants to remove the mollification on the noises ($ \eps  $ goes to zero), the map $ \Pi_x \tau$ fails to converge. One needs to perfom some renormalisation via the BPHZ scheme \cite{BP,Hepp,Zim}. This renormalisation has been implemented in \cite{BHZ} via an extraction-contraction type coproduct. Then, from the group of characters associated to this coproduct, one is able to renormalise the model whose convergence is proved in \cite{CH16}.
This extraction-contraction map $ \Delta_{\text{\tiny{RN}}} : S(\CT_E^V) \rightarrow   S(\CT_E^V) \hattimes S(\CT_E^V)$ is given by:
\begin{equs} \label{def_extraction_contraction}
\Delta_{\text{\tiny{RN}}} = \left( \CM \hattimes \id \right)  \left( \id \hattimes \Delta_{\text{\tiny{RN}}}^{\nonroot} \right) \Delta_{\text{\tiny{RC}}}.
\end{equs}
The identity \eqref{def_extraction_contraction} is a recursive formulation based on \cite[Prop. 44]{BR18}. The second term $ \Delta_{\text{\tiny{RN}}}^{\nonroot} $ does  not extract a tree at the root whereas $ \Delta_{\text{\tiny{RC}}} $ extracts such a tree. Then $ \Delta_{\text{\tiny{RN}}}^{\nonroot} $ is iterated on $ \Delta_{\text{\tiny{RC}}} $ for extracting other trees outside the root. A non-recursive definition could be found in \cite{BHZ}. The recursive definition is useful for establishing the connection with the coproducts constructed from the Guin-Oudom procedure. The map $ \Delta_{\text{\tiny{RN}}}  $ is actually equal to 
$ \Delta_1 $.
Indeed, one has from \eqref{ident_RC} 
\begin{equs}
\mathcal{M}^{(13)(2)} \left(   \Delta_{\circ}  \hattimes \id \right) \Delta_2 & = \mathcal{M}^{(13)(2)} \left(   \Delta_{\circ}  \hattimes \id \right) \mathcal{M}^{(1)(2)}\Delta_{\tiny{\text{RC}}}
\\ & = \left( \CM \hattimes \id \right)  \left(  \id  \hattimes \Delta_{\circ}  \right) \Delta_{\tiny{\text{RC}}}.
\end{equs}
We conclude by matching the two recursive identities recursively and $ \Delta_{\circ} = \Delta_{\text{\tiny{RN}}}^{\nonroot} $.
As a consequence of  Theorem~\ref{deformed_E_C}, we obtain one of our main results
\begin{theorem} \label{Deformation_EC} The map $ \Delta_{\text{\tiny{RN}}} $ is a deformation of the extraction-contraction coproduct. It is obtained by applying the Guin-Oudom procedure on $ \wh{\blacktriangleright}  $ which is a deformation of $ \blacktriangleright $.
\end{theorem}

As for the previous coproduct, one also performs projections for $ \Delta_{\text{\tiny{RN}}} $. Indeed, one defines  a coaction $ \Deltam : \CT \rightarrow \CT_- \otimes \CT $ from it, where the extraction is limited to trees of negative degree. This projection excludes single nodes and one denotes by $ \CT_- $ the  forests whose trees are in $ \CT $ and each tree is of negative degree.  
Then, the BPHZ renormalisation map is given in \cite{BHZ} by
\begin{equs}
 \label{def_M_BPHZ}
M_{\text{\tiny{BPHZ}}} = \left( \mathbb{E}(\Pi \tilde{\CA}_- \cdot)\otimes \id \right) \Deltam
\end{equs}
where $  \tilde{\CA}_- : \CT_- \rightarrow \hat \CT_- $ is a twisted antipode and $ \hat \CT_- $ corresponds to forests whose trees are in $  \CT $.  The identity \eqref{def_M_BPHZ} has also been interpreted as a Bogoliubov recursion in \cite{BE}. One of the main results in \cite{BHZ} is Theorem 6.16 which gives a simple formula for the renormalised model $ \hat \Pi_x $:
\begin{equs}
\label{renormalised_model_formula}
\hat \Pi_x = \Pi_x M_{\text{\tiny{BPHZ}}}.
\end{equs}
The fact that this definition gives again a model relies on the cointeraction (see \cite[Theorem 5.37]{BHZ}) obtained when one adds extended decorations in the regularity structures trees. At the level of the coactions, it reads:
\begin{equs} \label{cointeraction_-_+}
	\mathcal{M}^{(13)(2)(4)}\left(  \Deltam \otimes  \Deltam\right)
	 \Deltap = \left( \id \otimes \Deltap \right) \Deltam, 
\end{equs}
where for $ \bar \tau_1, \bar \tau_3 \in \CT_- $, $ \tau_4 \in \CT_+ $ and $ \tau_2 \in \CT $
\begin{equs}
	\mathcal{M}^{(13)(2)(4)}\left( \bar\tau_1 \otimes \tau_2 \otimes \bar \tau_3 
	\otimes \tau_4 \right) = \bar \tau_1 \cdot \bar \tau_3 \otimes 
	\tau_2 \otimes \tau_4.
\end{equs}
We have denoted the forest product by $ \cdot $ and the coaction $ \Deltam  $ is extended as a coaction on $ \CT_+ $. As a consequence of Theorem~\ref{theo_cointeraction}, we obtain the following cointeraction
\begin{equs} \label{cointeraction_3.19}
\mathcal{M}^{(13)(2)(4)}\left(  \Delta_{1}^{\nonroot} \hattimes  \Delta_{1} \right)
	 \Delta_{2} = \left( \id \hattimes \Delta_{2} \right) \Delta_{1}.
\end{equs}
Then by rewritting $\Delta_{2}$ as $\Delta_{\text{\tiny{RC}}}$, we get
\begin{theorem} \label{cointeraction_NA_EC} The map $ \Delta_{\text{\tiny{RC}}} $ is in cointeraction $ \Delta_{\text{\tiny{EC}}} $ in the sense that: 
\begin{equs}
\mathcal{M}^{(13)(2)(4)}\left(  \Delta_{\text{\tiny{RN}}} \hattimes  \Delta_{\tiny{\text{RN}}}^{\nonroot} \right)
	 \Delta_{\text{\tiny{RC}}} = \left( \id \hattimes \Delta_{\text{\tiny{RC}}} \right) \Delta_{\text{\tiny{RN}}}. 
\end{equs}
\end{theorem}

\begin{remark}
The cointeractions \eqref{cointeraction_-_+}  and \eqref{cointeraction_3.19} seem to differ because of the term $ \Delta_{\text{\tiny{RN}}}^{\nonroot}   $. In fact, for avoiding the extraction at the root, the authors in \cite{BHZ} use colors. The root of any tree in $ \CT_+ $ or $ \hat \CT_+ $ is colored and the map $ \Deltam  $ is extended to those trees by extracting trees which do not contain a node of that color. 
\end{remark}

\newpage
\appendix
\section{Symbolic index}
In this appendix, we collect the most used symbols of the article, together
with their meaning and the page where they were first introduced.
 \begin{center}
\renewcommand{\arraystretch}{1.1}
\begin{longtable}{lll}
\toprule
Symbol & Meaning & Page\\
\midrule
\endfirsthead
\toprule
Symbol & Meaning & Page\\
\midrule
\endhead
\bottomrule
\endfoot
\bottomrule
\endlastfoot
 $\curvearrowright^{a}$ & Grafting operator on $ \CT_E^V $ with an edge decorated by $ a $ & \pageref{grafting_a}
\\ 
 $\wh{\curvearrowright}^{a}$ & Deformation of $ \curvearrowright^{a} $ by $ \Theta $ & \pageref{deformed_grafting_a}
\\ 
$\curvearrowright$ & Grafting operator on $ \CP_E^V $ & \pageref{grafting_planted_trees}
\\
 $\wh{\curvearrowright}$ & Deformation of $ \curvearrowright $ by $ \Theta $ & \pageref{deformed_grafting_planted_trees}
 \\ 
 $\bar{\curvearrowright}$ & Product $\bar{\curvearrowright}  : S(\CT_E^V) \hattimes \CT_E^V \rightarrow \CT_E^V $ defined from $ \wh{\curvearrowright} $ & \pageref{recursive_deformed_grafting}
 \\ 
 $ \rhd $  & Plugging product on $ \CT_E^{V} $  & \pageref{plugging_p}
 \\   
 $ \blacktriangleright $  & Insertion product on $ \CT_E^{V} $  & \pageref{insertion_i}
 \\ 
 $ \widetilde{\rhd} $  & Deformed Plugging product by $ \Theta $ & \pageref{theta_plugging}
 \\ 
 $ \widetilde{\blacktriangleright} $  & Deformed Insertion product by $ \Theta $ & \pageref{theta_insertion}
 \\    
 $ \wh{\rhd} $  & Deformed Plugging product obtained from   $ \widetilde{\rhd}  $ & \pageref{deformed_plugging_p}
  \\    
 $ \wh{\blacktriangleright} $  & Deformed Insertion product obtained from   $ \widetilde{\blacktriangleright}  $ & \pageref{deformed_insertion_i}
 \\ 
  $\star $ & Product obtained from  $ \rhd $ by the Guin-Oudom procedure& \pageref{star_star}
\\ 
 $\star_0$ & Product obtained from  $ \wh{\curvearrowright} $ by the Guin-Oudom procedure& \pageref{star_0}
  \\ 
  $\star_2 $ & Product obtained from  $ \wh{\rhd} $ by the Guin-Oudom procedure& \pageref{star_2}
\\ 
  $\tilde{\star}_2 $ & Product obtained from  $ \widetilde{\rhd} $ by the Guin-Oudom procedure& \pageref{tilde_star_2}
  \\ 
  $ \star_1 $ & Product obtained from  $  \wh{\blacktriangleright} $ by the Guin-Oudom procedure& \pageref{star_1}
\\
$\Delta_{\text{\tiny{DCK}}}$ & Coproduct $ \Delta_{\text{\tiny{DCK}}} :  S(\CP_E^V) \rightarrow S(\CP_E^V) \hattimes S(\CP_E^V)$ dual of $ \star_0 $ & \pageref{delta_dck_bar}
\\
$\bar{\Delta}_{\text{\tiny{DCK}}}$ & Coproduct $ \bar{\Delta}_{\text{\tiny{DCK}}} :  \CT_E^V \rightarrow S(\CP_E^V) \hattimes \CT_E^V$ dual of $ \bar{\curvearrowright} $ & \pageref{delta_dck_bar}
\\
$ \Delta_{\text{\tiny{DP}}} $ & Coproduct $\Delta_{\text{\tiny{DP}}}: \mathcal{T}_E^V  \rightarrow \hat S(\mathcal{T}_E^V ) \hattimes \mathcal{T}_E^V $ & \pageref{delta_dp}
\\
$ \bar{\Delta}_{\text{\tiny{DP}}} $ & Coproduct $\bar{\Delta}_{\text{\tiny{DP}}}: \mathcal{T}_E^V  \rightarrow \hat S(\mathcal{T}_E^V ) \hattimes \mathcal{T}_E^V $ defined from
$ \Delta_{\text{\tiny{DP}}} $ & \pageref{delta_dp_bar}
\\
$\Delta_{\circ}$ & Coproduct $ \Delta_{\circ} :  S(\CT_E^V) \rightarrow S(\CT_E^V) \hattimes S(\CT_E^V) $ defined from $ \Delta_1 $ & \pageref{delta_0_0}
\\
$\Delta_1$ & Coproduct $ \Delta_1 :  S(\CT_E^V) \rightarrow S(\CT_E^V) \hattimes S(\CT_E^V) $  dual of $ \star_1 $ & \pageref{delta_1_1}
\\
$\Delta_2$ & Coproduct $ \Delta_2 = \left( \CC \hattimes \id \right) \Delta_{\text{\tiny{DP}}} $ dual of $ \star_2 $   & \pageref{delta_2_2}
\\
$\CC$ & Linear  map $ \CC : \hat{S}(\CT_E^V) \rightarrow S(\CT_E^V) $  & \pageref{cal_c}
\\
$\CK$ & Linear map $ \CK : S(\CT_E^V) \rightarrow \CT_E^V $ merging the roots & \pageref{ident_i}
\\
$\Theta$ & Linear isomorphism $ \Theta : \CT_E^V \rightarrow \CT_E^V  $  sending $\curvearrowright^{a}  $ to $\wh{\curvearrowright}^{a}  $  & \pageref{theta_iso}
\\
$\CP_E^V$ & Planted trees with edges decorated by $ E $ and nodes by $ V $  & \pageref{planted_p}
\\
$S(\CP_E^V)$ & Symmetric space over $ \CP_E^V $  & \pageref{forest_planted_p}
\\
$S(\CT_E^V)$ & Symmetric space over $ \CT_E^V $  & \pageref{forest_f}
\\
$\hat{S}(\CT_E^V)$ & Decorated forests with one distinguished tree & \pageref{forest_f_d}
\\
$\CT_E^V$ & Rooted trees with edges decorated by $ E $ and nodes by $ V $  & \pageref{rooted_r}
 \end{longtable}
 \end{center}
\endappendix


\begin{thebibliography}{99}

      
\bibitem{A14}
	M. J. H. Al-Kaabi,
	\newblock {\em Monomial bases for free pre-Lie algebras.}	S\'em. Loth. Combinatoire \textbf{71}, B71b (2014).

%
 
 

   \bibitem{BB21}
I.~Bailleul, Y.~{Bruned}.
\newblock { \em Renormalised singular stochastic PDEs}. 
 \burlalt{arXiv:2101.11949}{http://arxiv.org/abs/2101.11949}. 
 
 \bibitem{BB21b}
I.~Bailleul, Y.~{Bruned}.
\newblock { \em  Locality for singular stochastic PDEs}. 
 \burlalt{arXiv:2109.00399}{https://arxiv.org/abs/2109.00399}.

\bibitem{BCCH}
 { \rm Y. Bruned, A. Chandra, I. Chevyrev,
  M. Hairer}.
\newblock {\em Renormalising SPDEs in regularity structures}.
\newblock J. Eur. Math. Soc. (JEMS), \textbf{23}, no.~3, (2021), 869-947.
\newblock
  \burlalt{doi:10.4171/JEMS/1025}{http://dx.doi.org/10.4171/JEMS/1025}.



\bibitem{BCFP} 
Y.~{Bruned}, I.~{Chevyrev}, P.~K. {Friz}, 
 R.~{Preiss}.
\newblock { \em A rough path perspective on renormalization}.
\newblock {J. Funct. Anal.} \textbf{277}, no.~11, (2019), 108283.
\newblock
  \burlalt{doi:10.1016/j.jfa.2019.108283}{http://dx.doi.org/10.1016/j.jfa.2019.108283}.
  
  \bibitem{BE}
Y.~{Bruned}, K.~{Ebrahimi-Fard}.
\newblock {\em Bogoliubov type recursions for renormalisation in regularity structures}. 
\burlalt{arXiv:2006.05284}{http://arxiv.org/abs/2006.05284}. 

\bibitem{BGHZ}
{ \rm Y. Bruned, F. Gabriel, M. Hairer, 
  L. Zambotti},
\newblock   {\em Geometric stochastic heat equations}.
 \\
\newblock \textit{J. Amer. Math. Soc.} \textbf{35}, no.~1, (2022), 1-80. 
 \burlalt{doi:10.1090/jams/977}{http://dx.doi.org/10.1090/jams/977}.

      

\bibitem{BaiHos}
I.~{Bailleul}, M.~{Hoshino}.
\newblock {\em A tourist's guide to regularity structures.}
\newblock \burlalt{arXiv:2006.03524}{https://arxiv.org/abs/2006.03524}.

\bibitem{BHZ}
{\rm Y. Bruned, M. Hairer, L. Zambotti}.
\newblock {\em Algebraic renormalisation of regularity structures.}
\newblock Invent. Math. \textbf{215}, no.~3, (2019), 1039--1156.
\newblock
  \burlalt{doi:10.1007/s00222-018-0841-x}{https://dx.doi.org/10.1007/s00222-018-0841-x}.
  
  

\bibitem{EMS}
  {\rm Y. Bruned, M. Hairer, L. Zambotti}.
\newblock {\em Renormalisation of Stochastic Partial Differential Equations.}
\newblock EMS Newsletter \textbf{115}, no.~3, (2020), 7--11.
\newblock
  \burlalt{doi: 10.4171/NEWS/115/3}{http://dx.doi.org/10.4171/NEWS/115/3}.
  

    
   
\bibitem{BP}
N.~N. Bogoliubow, O.~S. Parasiuk.
\newblock { \em \"{U}ber die {M}ultiplikation der {K}ausalfunktionen in der
  {Q}uantentheorie der {F}elder.}
\newblock Acta Math. \textbf{97}, (1957), 227--266.
\newblock
  \burlalt{doi:10.1007/BF02392399}{http://dx.doi.org/10.1007/BF02392399}.  
  
\bibitem{BR18} 
{ \rm Y. Bruned}.
\newblock {\em Recursive formulae in regularity structures.}
\newblock Stoch. Partial Differ. Equ. Anal. and Comput. \textbf{6},
  no.~4, (2018), 525--564.
\newblock 
  \burlalt{doi:10.1007/s40072-018-0115-z}{http://dx.doi.org/10.1007/s40072-018-0115-z}. 
  
   \bibitem{Br1111}
Y.~{Bruned}.
\newblock { \em Renormalisation from non-geometric to geometric rough paths}. To appear in Ann. Inst. H. Poincaré
Probab. Statist., (2021), 13 pages. 
 \burlalt{arXiv:2007.14385}{http://arxiv.org/abs/2007.14385}. 

\bibitem{BS}
Y.~{Bruned}, K.~{Schratz}. \newblock { \em Resonance based schemes for dispersive equations via decorated 
        trees}. Forum of Mathematics, Pi, 10, E2. 
\newblock \burlalt{doi:10.1017/fmp.2021.13}{https://doi.org/10.1017/fmp.2021.13}.
 


\bibitem{Butcher72}
{ \rm J. C. Butcher}.
\newblock {\em An algebraic theory of integration methods.}
\newblock Math. Comp. \textbf{26}, (1972), 79--106.
\newblock \burlalt{doi:10.2307/2004720}{http://dx.doi.org/10.2307/2004720}.

\bibitem{CCHS}
A.~Chandra, I.~Chevyrev, M.~Hairer, H.~Shen.
\newblock {\em Langevin dynamic for the 2D Yang-Mills measure}.
\newblock \burlalt{arXiv:2006.04987}{http://arxiv.org/abs/2006.04987}.


\bibitem{CEM}
D.~Calaque, K.~Ebrahimi-Fard, D.~Manchon.
\newblock { \em Two interacting {H}opf algebras of trees: a {H}opf-algebraic approach
  to composition and substitution of {B}-series.}
\newblock Adv. in Appl. Math. \textbf{47}, no.~2, (2011), 282--308.
\newblock 
  \burlalt{doi:10.1016/j.aam.2009.08.003}{http://dx.doi.org/10.1016/j.aam.2009.08.003}.

 \bibitem{CH16}
A.~Chandra, M.~Hairer.
\newblock {\textsl{An analytic {BPHZ} theorem for regularity structures.}}
\newblock \burlalt{arXiv:1612.08138}{http://arxiv.org/abs/1612.08138}. 

\bibitem{CHS}	
	{ \rm A. Chandra, M. Hairer, H. Shen}.
\newblock {\em The dynamical sine-Gordon model in the full subcritical regime.}
\newblock \burlalt{arXiv:1808.02594}{https://arxiv.org/abs/1808.02594}.

\bibitem{CHV10}
P.~{Chartier}, E.~{Hairer},  G.~{Vilmart}.
\newblock {\em Algebraic structures of {B}-series.}
\newblock  Found. Comput. Math. \textbf{10}, no.~4, (2010), 407--427.
\newblock
  \burlalt{doi:10.1007/s10208-010-9065-1}{http://dx.doi.org/10.1007/s10208-010-9065-1}.


     


	
	\bibitem{CK1}
A.~Connes, D.~Kreimer.
\newblock { \em Hopf algebras, renormalization and noncommutative geometry.}
\newblock Comm. Math. Phys. \textbf{199}, no.~1, (1998), 203--242.
\newblock
  \burlalt{doi:10.1007/s002200050499}{http://dx.doi.org/10.1007/s002200050499}.

\bibitem{CK2}
A.~Connes, D.~Kreimer.
\newblock { \em Renormalization in quantum field theory and the {R}iemann-{H}ilbert
  problem {I}: the {H}opf algebra structure of graphs and the main theorem.}
\newblock  Comm. Math. Phys. \textbf{210}, (2000), 249--73.
\newblock
  \burlalt{doi:10.1007/s002200050779}{http://dx.doi.org/10.1007/s002200050779}.
 
 \bibitem{ChaLiv}
    	F.~Chapoton, M.~Livernet,
    	{\textsl{Pre-Lie algebras and the rooted trees operad}},
    Internat.~Math.~Res.~Notices  \textbf{2001}, no.~8, (2001), 395--408.
     \burlalt{doi:10.1155/S1073792801000198}{https://doi.org/10.1155/S1073792801000198}.  
     
  \bibitem{CMW}	
	{ \rm A. Chandra, A. Moinat, H. Weber}.
\newblock {\em A priori bounds for the $\phi^4$ equation in the full sub-critical regime.}
\newblock \burlalt{arXiv:1910.13854}{https://arxiv.org/abs/1910.13854}.
	
\bibitem{EM14}
	{K. Ebrahimi-Fard, D. Manchon},
	\textsl{On an extension of Knuth's rotation correspondence to reduced planar trees},
	J. Noncomm. Geom. \textbf{8}, no.~2, (2014),  303--320. \burlalt{doi: 10.4171/JNCG/157}{https://doi.org/10.4171/JNCG/157}. 
 

	

	
\bibitem{FrizHai}
P.~K. {Friz}, M.~{Hairer}.
\newblock \emph{{A Course on Rough Paths}}.
\newblock Springer International Publishing, 2020.
\newblock
  \burlalt{doi:10.1007/978-3-030-41556-3}{https://dx.doi.org/10.1007/978-3-030-41556-3}.
  
  \bibitem{F2018}
	L.~Foissy,
	\textsl{Algebraic structures on typed decorated rooted trees}, SIGMA \textbf{17}, no.~86, (2021), 1-28.
\burlalt{doi:10.3842/SIGMA.2021.086}{https://doi.org/10.3842/SIGMA.2021.086}. 
 
  
\bibitem{GL}
 R. Grossman, R. G. Larson,
\newblock {\em Hopf algebraic structure of families of trees},
\newblock J. Algebra \textbf{126}, no.~1 (1989), 184--210. \burlalt{doi:10.1016/0021-8693(89)90328-1}{https://doi.org/10.1016/0021-8693(89)90328-1}. 


\bibitem{Guin1}
D. Guin,  J. M. Oudom, \emph{Sur l'alg\`ebre enveloppante d'une
  alg\`ebre pr\'{e}-{L}ie}, C. R. Math. Acad. Sci. Paris \textbf{340} (2005),
  no.~5, 331--336.
  \burlalt{doi:10.1016/j.crma.2005.01.010}{https://doi.org/10.1016/j.crma.2005.01.010}. 
  

\bibitem{Guin2}
D. Guin,  J. M. Oudom, \emph{On the {L}ie enveloping algebra of a pre-{L}ie algebra}, J.
  K-Theory \textbf{2} (2008), no.~1, 147--167.
  \burlalt{doi:10.1017/is008001011jkt037}{https://doi.org/10.1017/is008001011jkt037}. 
  
 \bibitem{Gub04}
{\rm M.~Gubinelli},
\newblock {\em Controlling rough paths.}
\newblock J. Funct. Anal. \textbf{216}, no.~1, (2004), 86--140. \burlalt{doi:10.1016/j.jfa.2004.01.002}{https://doi.org/10.1016/j.jfa.2004.01.002}. 

\bibitem{Gub10}
{\rm M.~Gubinelli}.
\newblock {\em Ramification of rough paths.}
\newblock J. Differ. Equ. \textbf{248}, no.~4, (2010),
  693 -- 721.
\newblock
  \burlalt{doi:10.1016/j.jde.2009.11.015}{http://dx.doi.org/10.1016/j.jde.2009.11.015}.

  
\bibitem{reg}
{\rm M. Hairer}.
\newblock {\em A theory of regularity structures.}
\newblock Invent. Math. \textbf{198}, no.~2, (2014), 269--504.
\newblock
  \burlalt{doi:10.1007/s00222-014-0505-4}{https://dx.doi.org/10.1007/s00222-014-0505-4}.
  
\bibitem{Hepp}
K.~Hepp.
\newblock { \em On the equivalence of additive and analytic renormalization.}
\newblock Comm. Math. Phys.
 \textbf{14}, (1969), 67--69.
\newblock
  \burlalt{doi:10.1007/BF01645456}{http://dx.doi.org/10.1007/BF01645456}. 
  

 
 
  \bibitem{Felix}
P.~Linares, F.~Otto, M.~Tempelmayr,
\newblock { \em The structure group for quasi-linear equations via universal enveloping algebras}. 
 \burlalt{arXiv:2103.04187}{http://arxiv.org/abs/2103.04187}.  
 
 \bibitem{Lyo98}
{ \rm T.~J. Lyons},
\newblock {\em Differential equations driven by rough signals.}
\newblock Rev. Mat. Iberoamericana \textbf{14}, no.~2, (1998), 215--310.
\burlalt{doi:10.4171/RMI/240}{https://www.ems-ph.org/journals/show_abstract.php?issn=0213-2230&vol=14&iss=2&rank=1}.

%
%
	 
	 \bibitem{MS}
	D.~Manchon, A.~Saïdi,
	{\textsl{Lois pré-Lie en interaction}}, 
	Comm. Alg. {\bf{39}}, No 10,  3662--3680 (2011). 
	 \burlalt{doi:10.1080/00927872.2010.510813}{http://dx.doi.org/10.1080/00927872.2010.510813}.
	

       
       \bibitem{Zim}
W.~Zimmermann.
\newblock {\em Convergence of {B}ogoliubov's method of renormalization in momentum
  space.}
\newblock Comm. Math. Phys. \textbf{15}, (1969), 208--234.
\newblock
  \burlalt{doi:10.1007/BF01645676}{http://dx.doi.org/10.1007/BF01645676}.

\end{thebibliography}
\end{document}